\documentclass{book} 

\usepackage{etex}
\usepackage[utf8]{inputenc} 
\usepackage[english]{babel} 
\usepackage{anysize} 

\usepackage{amsthm}  
\usepackage{amsmath}
\usepackage{amsfonts}
\usepackage{amssymb}

\usepackage{multicol} 
\usepackage{graphicx} 
\usepackage{dcpic}
\usepackage{pictex}
\usepackage{wrapfig}
\usepackage{fancyhdr}
\usepackage{url}
\usepackage{verbatim}
\usepackage{hyperref}

\usepackage{mathrsfs}
\usepackage{capt-of}
\usepackage{pgf,tikz}
\usetikzlibrary{arrows}
\usepackage{leftidx}

\marginsize{2 cm}{2 cm}{1 cm}{2 cm}

\newcommand{\noun}[1]{\textsc{#1}}
\newcounter{mref}
\newcommand{\bitem}[1]{\addtocounter{mref}{1}\bibitem[\themref]{#1}}

\newenvironment{Referencias}
{\vspace{6mm}\setcounter{mref}{0}\noindent{\bf References}\smallskip\small\begin{enumerate}}	
{\end{enumerate}\vspace{6mm}}

\addto\captionsenglish{}

\newtheorem{prop}{Proposition}[section]

\newtheorem{teor}[prop]{Theorem}
\newtheorem{corol}[prop]{Corollary}
\newtheorem{lema}[prop]{Lemma}

\theoremstyle{definition}
\newtheorem{definicion}[prop]{Definition}

\newcommand{\f}[5]{
\begin{array}{rcl}
#1:#2 & \longrightarrow & #3 \\
#4 & \longmapsto & #5  \\
\end{array}
}
\newcommand{\R}{\mathbb R}

\newcommand{\Q}{\mathbb Q}
\newcommand{\N}{\mathbb N}
\newcommand{\Z}{\mathbb Z}
\newcommand{\C}{\mathbb C}

\newcommand{\Id}{\mathrm{id}}
\renewcommand{\phi}{\varphi}

\renewcommand{\P}{\mathbb P}
\newcommand{\eps}{\varepsilon}

\newcommand{\zn}[1]{\Z/ #1\Z}
\newcommand{\Zn}[1]{\frac{\Z}{#1\Z}}
\newcommand{\rd}{\sqrt[3]{2}}
\newcommand{\puntoblanco}[1]{
\draw [color=black] #1 circle (1.5pt);
\fill[color=white] #1 circle (1.3pt);
}
\newcommand{\puntonegro}[1]{
\fill [color=black] #1 circle (1.5pt);
}

\newcommand{\Fd}{\widehat F_2}
\newcommand{\CC}{\mathcal C}

\newcommand{\gal}{\mathrm{Gal}\left(\overline \Q/\Q\right)}
\newcommand{\reg}{\widetilde{(C,f)}}
\newcommand{\Aut}{\mathrm{Aut}}
\newcommand{\wt}{\widetilde}
\newcommand{\D}[1]{\wt{D_#1}}

\newcommand{\oB}{\overline B}
\newcommand{\K}{\overline K}
\newcommand{\F}{\mathbb F}
\newcommand{\QQ}{\overline{\Q}}
\newcommand{\lra}{\longrightarrow}
\newcommand{\Llra}{\Longleftrightarrow}
\DeclareMathOperator{\ord}{ord}
\newcommand{\m}{\mathfrak m}
\renewcommand{\O}{\mathcal O}
\newcommand{\cor}{\mathrm{Core}}
\newcommand{\KK}{\mathcal K}
\newcommand{\Gal}{\mathrm{Gal}}
\newcommand{\gall}{\Gal(\KK/\Q(t))}
\newcommand{\wh}[1]{\widehat{#1}}
\newcommand{\tl}{\triangleleft}
\renewcommand{\o}[1]{\overline{#1}}
\newcommand{\Inn}{\mathrm{Inn}}
\newcommand{\Out}{\mathrm{Out}}
\newcommand{\M}{\mathcal M}

\newcommand{\tri}{\wh{\Delta(l,m,n)}}
\renewcommand{\H}{\mathbb H}
\newcommand{\zz}{\langle z\rangle}
\newcommand{\ZZZ}{\wh \Z(1)}
\newcommand{\GK}{\Gal(\QQ/K)}
\newcommand{\HRule}{\rule{\linewidth}{0.5mm}}

\title{\Huge Dessins d'enfants}
\author{\Large Moisés Herradón Cueto\\\Large Advisor: Andrei Jaikin Zapirain}
\date{\today}

\begin{document}
\begin{titlepage}
\begin{center}

~
\vspace{7 cm}

\textsc{\LARGE Universidad Autónoma de Madrid}\\[1.5cm]

\textsc{\Large Trabajo de Fin de Máster}\\[0.5cm]

\HRule \\[0.4cm]
{ \huge \bfseries The field of moduli and fields of definition of dessins d'enfants \\[0.4cm] }

\HRule \\[1.5cm]

\begin{minipage}{0.4\textwidth}
\begin{flushleft} \large
\emph{Author:}\\
Moisés \textsc{Herradón Cueto}
\end{flushleft}
\end{minipage}
\begin{minipage}{0.4\textwidth}
\begin{flushright} \large
\emph{Advisor:} \\
Andrei \textsc{Jaikin Zapirain}
\end{flushright}
\end{minipage}
\\
\vspace{3 cm}

{\small MSC: 11G32, 14H57}
\vfill

{\large June 20, 2014}

\end{center}
\end{titlepage}
\topskip0pt
\vspace*{\fill}
\section*{\centering \begin{normalsize}Abstract\end{normalsize}}
\begin{quotation}
\noindent We introduce dessins d'enfants from the various existing points of view: As topological covering spaces, as surfaces with triangulations, and as algebraic curves with functions ramified over three points. We prove Belyi's theorem that such curves are defined over number fields, and define the action of the Galois group $\gal$ on dessins d'enfants. We prove that several kinds of dessins d'enfants are defined over their field of moduli: regular dessins, dessins with no nontrivial automorphisms and dessins with one face. In the last part, we give two examples of regular dessins d'enfants with a field of moduli that is not an abelian extension of $\Q$. Both of the examples have genus 61 and field of moduli $\Q(\rd)$.
\end{quotation}
\section*{\centering \begin{normalsize}Resumen\end{normalsize}}
\begin{quotation}
\noindent Introducimos los dessins d'enfants desde los distintos puntos de vista existentes: como espacios recubridores, como superficies con triangulaciones y como curvas algebraicas con funciones ramificadas sobre tres puntos. Probamos el teorema de Belyi, que dice que tales curvas se pueden definir sobre cuerpos de números, y definimos la acción del grupo de Galois $\gal$ sobre los dessins. Probamos que varios tipos de dessins d'enfants están definidos sobre su cuerpo de moduli: los dessins regulares, los que no tienen automorfismos no triviales, y los que sólo tienen una cara. En la última parte, damos dos ejemplos de dessins d'enfants regulares con cuerpo de moduli que no es una extensión abeliana de $\Q$. Ambos ejemplos tienen género 61 y cuerpo de moduli $\Q(\rd)$.
\end{quotation}
\vspace*{\fill}

\chapter*{Introduction}

There are many ways to define dessins d'enfants. The first one is as a compact orientable surface with a graph embedded in it, such that its vertices can be bicolored, and the faces are homeomorphic to disks. From here, one can divide the surface in triangles, by choosing a point on each face and joining it to the vertices of the face. Using these triangles, one can then define a covering map from the surface onto the sphere, by dividing the sphere into two triangles (the hemispheres) and then mapping the triangles in the original surface to the ones in the sphere. For example, one could proceed like this:

\begin{figure}[h!]
\centering
\includegraphics[width=0.8\textwidth]{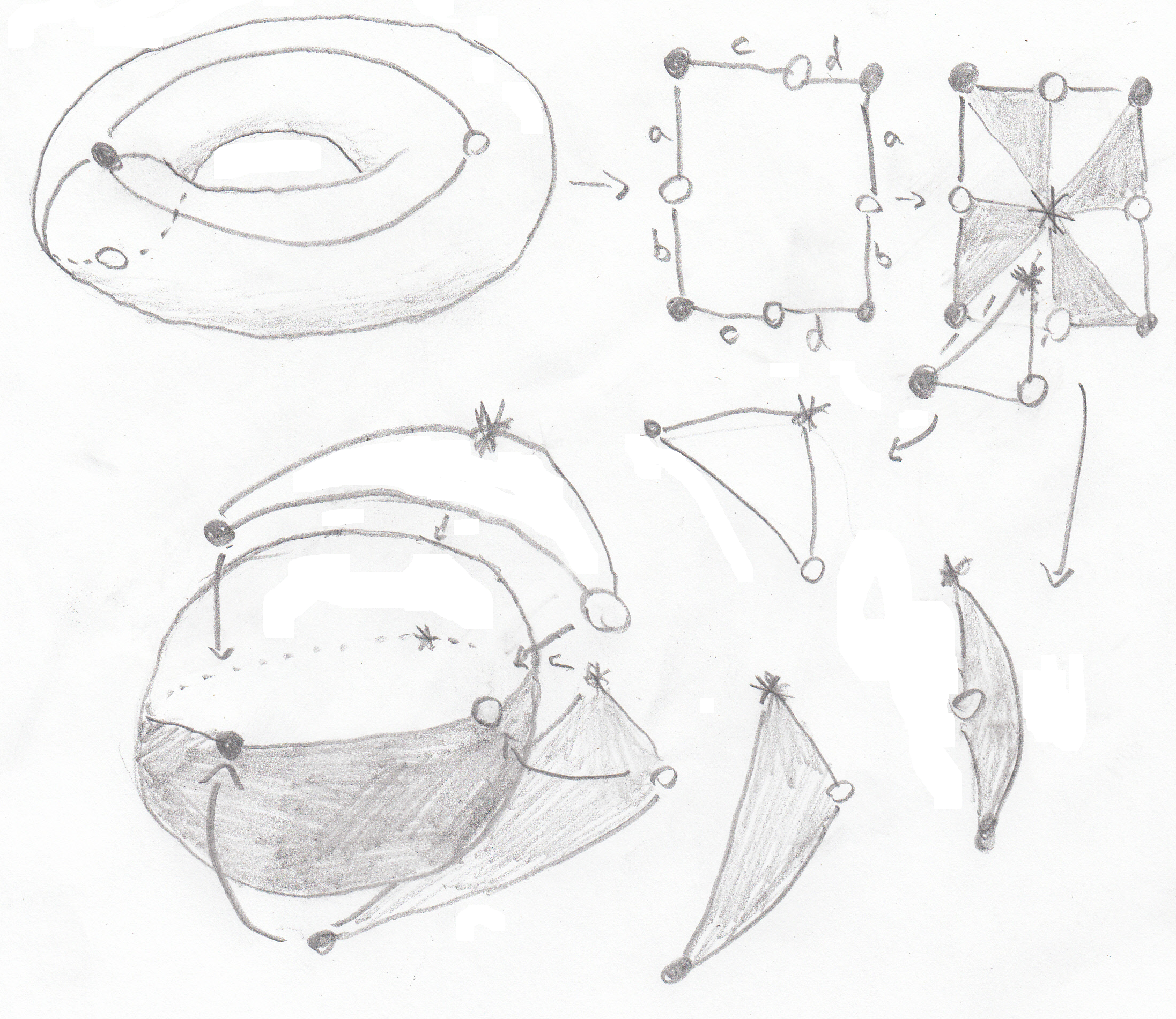}
\caption{The graph on the torus gives rise to a triangulation. We can then chop the torus up into 8 triangles and produce a 4-sheeted covering of the sphere.}\label{dibujito}
\end{figure}
The fun part starts when one takes the complex structure that comes from seeing the sphere as $\P^1(\C)$, and realizes that this structure determines a complex structure on the surface, i.e. it makes it into a Riemann surface. Also, every compact Riemann surface is isomorphic to some complex algebraic curve, so each dessin d'enfant can be interpreted as a covering, and this covering is actually a morphism of algebraic curves, so it can all be interpreted algebraically. What we end up obtaining is a bijection between dessins d'enfants and complex algebraic curves with a map to $\P^1(\C)$ that is ramified over $\{0,1,\infty\}$, which is called a Belyi map. For example, if we lift the conformal structure to the torus in figure \ref{dibujito}, we obtain the plane curve with equation $y^2=x^3-x$, and the covering map is the map $(x,y)\mapsto x^2$.

There's more: these equations that determine the algebraic curves and the functions can actually be written with coefficients not in $\C$, but in $\QQ$. This is one part of what is called Belyi's theorem. The other part, which Belyi proved in 1980 in \cite{belyi}, is that every curve defined over the algebraic numbers has maps that are ramified over three points. This theorem profoundly impressed Grothendieck, and led him to define dessins d'enfants, which he introduced in his Esquisse d'un Programme \cite{gro}.

The interest of dessins d'enfants lies in that we have simple combinatorial objects which are equivalent to curves with maps over $\QQ$. Given a curve and a Belyi map, we can take $\gal$ and make it act on the coefficients of their defining equations. What we obtain is sometimes a different dessin d'enfant. In fact, the action of $\gal$ on the set of dessins is faithful, that is, there are no Galois automorphisms that fix every dessin d'enfant. Thus, dessins d'enfants can be a tool for studying $\gal$. Actually, one can use dessins d'enfants to embed $\gal$ into other groups, like $\Out(\Fd)$ or $\Aut(\Fd)$.

In this work I have tried to present dessins d'enfants for someone with my own background. This means a really poor background in algebraic geometry, some algebra, but enough knowledge of group theory, covering spaces and Galois theory of finite extensions. The main criterion I have followed when choosing which proofs to include, and which approach to take, was based on my own background. From my point of view, a great part of the beauty of dessins d'enfants is that they are related to many different areas, such as algebraic geometry, complex geometry, topology, group theory, Galois theory and number theory. Thus, many results, especially in Part 1, can be proven using several of the points of view one can use in dessins d'enfants. I have no doubt that the reader with some knowledge of algebraic geometry will find many propositions trivial, had I not wandered around trying to prove it using groups. On the plus side, for me, and hopefully for readers with similar background as me, this thesis has meant that I have seen many theorems in algebraic geometry and I have been able to understand them and prove them using techniques that I already know and understand (although at the cost of proving them in a very specific setting). Despite the lack of generality, I think this has given me some great ``feel'' of algebraic geometry, especially for things like the étale fundamental group, and object which, at the time of writing, I do not know how to define.

With this in mind, in Part 1 there are the different definitions of dessins d'enfants, and the proof for the equivalence between them, in an order which I personally find natural. There are many expositions of this, so the reader can try Pierre Guillot's survey \cite{Guillot} for a combinatorics-based approach, or Girondo and González's book \cite{GG} for an approach from Riemann surfaces, where one can also find the uniformization approach, which is really fruitful and isn't included here. Another great reference is the book by Lando and Zvonkin \cite{Lando}. In the next part, the Galois action on dessins d'enfants is defined, and we prove Belyi's Theorem. We prove the ``obvious'' part like \cite{GG}, which helps us to avoid hardcore algebraic geometry concepts like the ones used in Weil's paper \cite{weil}. Then, we define the embedding of $\gal$ in $\Aut(\Fd)$ and $\Out(\Fd)$, and the field of moduli and the field of definition of a dessin.

The field of moduli of a dessin is the fixed field of the subgroup of $\gal$ that leaves it invariant, and a field of definition is a number field such that the dessin can be given by equations with coefficients in this number field. Our main problem is, what is the relationship between these fields? It is clear that the field of moduli will be contained in any field of definition. However, the field of moduli isn't always a field of definition. We prove some cases where the field of moduli is the field of definition. It follows from Weil's result, and it can be seen in \cite{abc} for dessins d'enfants, that regular dessins d'enfants, that is, dessins d'enfants whose automorphism group acts transitively on the edges, and dessins d'enfants without nontrivial automorphisms are both defined over their fields of moduli. Also, it is known (see \cite{Lando}) that dessins that are trees on the sphere are defined over their field of moduli. I present a proof of the slight generalization that dessins with one face, on any surface, can be defined over their field of moduli. This proof is due to my advisor, Andrei Jaikin.

In the last part, we turn our attention to the problem of finding a regular dessin whose field of moduli is not an abelian extension, which appears in \cite{cojowo}. We construct one example of such a dessin, and we use it to give examples of the points of view explained in parts 1 and 2. Also, we show that the curve we construct has itself the same field of moduli. Finally, we comment on another example of a regular dessin with non abelian field of moduli that appeared in \cite{voed} and \cite{malle}.

I am grateful to my advisor, Andrei Jaikin, for many insightful discussions, which I have struggled to follow but I have finally learned a lot from them. I also want to thank my friends, girlfriend and family for virtually everything.

My master's degree has been supported by the program Posgrado de Excelencia Internacional of the Autonomous University of Madrid.

\tableofcontents
\chapter{Various definitions of dessins d'enfants}
\section{Dessins as covering maps and as holomorphic maps}
There are many equivalent ways to define dessins d'enfants. We are going to give definitions closely related to the ones in \cite{Guillot}, which come from the theory of covering spaces and combinatorics. We are going to eventually see that dessins can also be defined as algebraic curves defined over the algebraic numbers with a choice of a meromorphic function with some conditions. To do this from the point of view of topology, we will always see the sphere as the projective complex line $\P^1=\P^1(\C)$.

\begin{definicion}
Let $\P^1\setminus \{0,1,\infty\}$ be the sphere with three points removed. A \textbf{dessin d'enfant} is a connected covering $\phi:S\lra \P^1\setminus \{0,1,\infty\}$ of finite degree.
\end{definicion}

\noindent For example, the identity map of the sphere minus three points is a dessin, which we call the trivial dessin. We can add more structure to a dessin d'enfant: we can give the sphere a complex structure, since we can see it as $\P^1$. Then, a complex structure on the sphere determines a complex structure on the covering surface. This complex structure in the sphere isn't affected by our choice of three points $\{0,1,\infty\}$, since the Möbius transformation $z\longmapsto \frac{z-z_0}{z-z_\infty}\frac{z_1-z_\infty}{z_1-z_0}$ sends any three points $\{z_0,z_1,z_\infty\}$ to $\{0,1,\infty\}$ and it is biholomorphic.

So a dessin is equivalent to choosing a covering of $\P^1\setminus \{0,1,\infty\}$, which will be holomorphic once we pick a suitable complex structure in the covering surface. Now, the key observation is that one can compactify the covering surface and extend the covering map $\phi$ to produce a holomorphic map from a compact Riemann surface into the Riemann sphere.



\begin{prop}
{cubierta}
Let $\phi:S\lra\P^1\setminus\{0,1,\infty\}$ be a dessin d'enfant. There exists a compact Riemann surface $\overline S$ and a holomorphic map $\overline \phi:\overline S\lra \P^1$, such that
\begin{enumerate}
\item $S$ is an open set of $\overline S$.
\item $\overline S\setminus S$ is a finite set.
\item $\overline \phi|_S=\phi$.
\item All the ramification points of $\overline \phi$ lie in the preimage of $\{0,1,\infty\}$.
\end{enumerate}
And any other compact Riemann surface $\o{S}'$ with a map $\o{\phi}':\o{S}'\lra \P^1$ satisfying the same conditions is biholomorphic to $\o{S}$, by a biholomorphism $\Psi$ such that $\o{\phi}'=\o{\phi}\circ\Psi$.
\end{prop}
\begin{proof}
There is an obvious complex structure on $S$, defined as follows: for each point $p\in S$, let $U$ be a simply connected neighborhood of $\phi(p)$. Now, since $U$ is simply connected and $\phi$ is a covering map, $\phi^{-1}(U)$ will be a finite disjoint union of open sets homeomorphic to $U$. Let $U_p$ be the one containing $p$. $\phi|_{U_p}$ is a homeomorphism, and therefore we have a chart $(U_p,\phi|_{U_p})$. Any two of these charts are compatible, since the transition maps between different charts are always the identity maps of some sets of $\P^1$.

We need to add some points to $S$ now in order to compactify it. Let us build the preimage of $0$. Let $U$ be the punctured disc of radius $1/2$ around $0$. If we restrict $\phi$ to the preimage of $U$, which is an open set of $\P^1\setminus \{0,1,\infty\}$, it will also be a covering. Therefore, we can split $\phi^{-1}(U)$ into its connected components, which are open, since $S$ is locally connected (it is a surface, after all). Let these connected components be $U_1,\ldots ,U_m$. For each $i$, $\phi_i=\phi|_{U_i}$ is a covering map of $U$. Now, $U$ has the homotopy type of a circle, and every covering of a circle is of the form $z\mapsto z^n$ for some $n\in \mathbb Z\setminus\{0\}$. Via the homotopy equivalence between the circle and $U$, the covering of the circle is equivalent to the map $z\mapsto 2^{n-1}z^n$. This means that there is a homeomorphism $\psi:U_i \lra U$ such that the following diagram commutes:
\begin{center}\hspace{0.5 cm}
\begindc{\commdiag}[40]
\obj(-1,1)[a]{$U_i$}
\obj(1,1)[b]{$U$}
\obj(0,0)[c]{$U$}
\mor{a}{b}{$\psi$}
\mor{a}{c}{$\phi_i$}
\mor{b}{c}{$\pi=z\mapsto 2^{n-1}z^n$}
\enddc
\end{center}
i.e. $\phi_i=\pi\circ \psi$. Now, we can take $\overline U=U\cup \{0\}$ and define a point $p_i=\psi^{-1}(0)\in U_i$, such that $\overline U$ and $\overline U_i=U_i\cup \{p_i\}$ are homeomorphic by a homeomorphism $\overline \psi$ extending $\psi$. Also, we can use $(\overline U_i,\overline \psi^{-1})$ as a chart around $p_i$. This chart is compatible with the atlas previously defined: If we have another point $q$, with a neighborhood $V_q$ which maps homeomorphically to $V$, the transition map is $\phi|_{U_p}\circ \psi^{-1}$. Now, since $\phi_i=\pi\circ \psi$, when we restrict ourselves to the intersection of the charts, we have that $\phi|_{U_p}\circ \psi^{-1}=\pi$, which is holomorphic.

Finally, if we define $\overline \phi(p_i)=0$ for all these new points, the resulting map is holomorphic, because we have defined it to be (if we take charts, the map will be $z\mapsto z^n$ for some $n$, on every chart). We can do the same thing to define preimages of $1$ and $\infty$ (using disjoint neighborhoods of $1$ and $\infty$), and we end up adding a finite number of points, obtaining a Riemann surface $\overline S$ and a holomorphic map to $\P^1$. Note that, in the topology we have defined every open subset of $S$ is an open subset of $\overline S$, so in particular $ S$ is an open subset of $\overline S$.

Also, $\overline S$ is compact: take an open cover $\mathcal U=\{U_i:i\in I\}$ of $\overline S$. Each point $p\in \P^1$, has a finite number of preimages. Now, for every point $p\in \P^1$, there is a neighborhood $V^p$ such that $\overline \phi^{-1}(V^p)=\sqcup_{j=1}^m V^p_j$, and $\overline \phi|_{V^p_j}$ is conjugate to the map $z\mapsto z^n$, in a neighborhood of $0$. For each preimage of $p$, we can pick a neighborhood contained in its corresponding $V_p^j$ and some set $U\in \mathcal U$; take the images of all these neighborhoods and intersect them. This gives a neighborhood $U^p$ of $p$ such that every connected component of its preimage is contained in an element of the cover. Since $\P^1$ is compact, the covering $\{V^p\}$ has a finite subcover $V_1,\ldots ,V_r$. Also, the preimage by $\overline \phi$ of each set $V_j$ is contained in the union of a finite number of $U_i$'s, and so the whole collection of these $U_i$'s is the finite subcover we are looking for. So $\overline S$ is compact.

For the uniqueness, note that the complex structure is completely determined on the points of $S$, so any two surfaces verifying the proposition will have open sets $S$ that will be biholomorphic, and whose complements will be finite sets. However, holomorphic functions can be extended to isolated points, so the surfaces will be isomorphic.
\end{proof}

\noindent Throughout this work, if we have a map $\phi:S\lra T$, and a point $p\in S$, such that $\phi'(p)=0$, we will call $p$ a \textbf{ramification point} and $\phi(p)$ a \textbf{ramification value}. Now, the previous proposition proves that dessins d'enfants can be viewed as holomorphic mappings of compact Riemann surfaces onto the Riemann sphere, such that their ramification values are contained in $\{0,1,\infty\}$. The converse is also true: if one restricts an holomorphic mapping to its unramified points, the result is a covering, because of the open mapping theorem for holomorphic maps. Therefore, from a surface with a map to the sphere ramified over at most three points, one obtains a dessin by removing the ramification points from the surface and $\{0,1,\infty\}$ from $\P^1$.

We can give dessins even more structure: it is a classical result that every compact Riemann surface is an algebraic curve, that is, it can be embedded (as a nonsingular curve) in $\P^n(\C)$ for some $n$ (actually $n=3$ is enough), such that its image is the zero set of some complex polynomials. This result can be found in \cite{GG}, assuming the uniformization theorem. The uniformization theorem and its proof can be found in \cite{farkas}. The result can also be found in \cite{Miranda}, if one is willing to assume that Riemann surfaces have enough meromorphic functions.

Every algebraic curve we talk about from now on, unless we say otherwise, is non-singular and projective.

For an algebraic curve $C$, we call $f:C\lra \P^1$ a \textbf{Belyi function} if it is ramified over some subset of $\{0,1,\infty\}$. If $f$ is a Belyi function for $C$, we call $(C,f)$ a \textbf{Belyi pair}. What we have established is that dessins d'enfants are equivalent to Belyi pairs. Since we will use equivalent definitions for dessins d'enfants, we might sometimes abuse notation and call a Belyi pair a dessin d'enfant.

There is one last bit of additional structure we can give dessins d'enfants. It is thanks to Belyi's theorem, which says the following: a curve has a Belyi function if and only if it is defined over the field of algebraic numbers. We will talk about dessins d'enfants for a while before proving it, but we are giving the motivation now. If one has a curve given by some set of equations and a Belyi function on it, all of whose coefficients are algebraic numbers one can take an element of $\gal$ and have it act on the coefficients of the equations and the function. This gives a Belyi pair which isn't necessarily isomorphic to the previous one. Thus, we have a non-trivial action of $\gal$. In fact, this action is faithful in many ways. We will talk a lot about the Galois action after we have finished introducing dessins d'enfants.

We are going to take a moment to talk about the relationship between different dessins. Namely, dessins form a category. When we see dessins as covers of $\P^1\setminus \{0,1,\infty\}$, a morphism between two Belyi pairs $(C_1,f_1)$ and $(C_2,f_2)$ is a morphism (of curves, or equivalently, a holomorphic map) $g:C_1\lra C_2$ such that $f_1=f_2\circ g$. We say that the dessin $(C_1,f_1)$ \textbf{covers} the dessin $(C_2,f_2)$. Two dessins are equivalent, or isomorphic, if they are isomorphic in this category.

\section{Dessins as algebraic curves and as field extensions}

We are going to talk about dessins d'enfants from the point of view of the fields of functions of the algebraic curves involved. We are going to talk about curves over $\C$, but we ask the reader to bear in mind that every statement we make about $\C$-algebras from now on is also valid for any algebraically closed field of characteristic $0$, in particular if one replaces $\C$ with $\QQ$. This is important, since dessins are defined over $\QQ$, and this will ultimately be the field we are interested in.

If we have a Belyi pair $(C,f)$, we can look at its field of functions $\C(C)$. To the morphism $f:C\lra \P^1$ corresponds a field homomorphism $f^*:\C(t)\lra \C(C)$, given by $P\longmapsto P(f)$. This is the same as the field inclusion $\C(f)\subset \C(C)$. Thus, to a dessin we can associate an extension of $\C(t)$.

Conversely, finite extensions of $\C(t)$ are always fields of functions of curves.

\begin{teor}
The following categories are equivalent:
\begin{enumerate}
\item Non-singular projective complex algebraic curves, with non-constant regular functions as morphisms.
\item Extensions of $\C$ of degree of transcendence 1, with $\C$-algebra homomorphisms as morphisms.
\end{enumerate}
The functor takes curves to their function fields.
\end{teor}
\begin{proof}
The proof can be found in \cite{GG}. We will just state how the functor acts on morphisms: given a regular map (a rational function) between two algebraic curves $\phi:C\lra C'$, one can consider the $\C$-algebra homomorphism
$$
\f{\phi^*}{\C(C')}{\C(C)}{f}{\phi^*f=f\circ \phi}
$$
Note that such morphisms are always injective, since both algebras are fields. Conversely, if one has an embedded algebraic curve $C'$ with homogeneous coordinates $(Y_0:\cdots :Y_n)$ (which can be seen as functions on the curve) and a homomorphism $\phi^*:\C(C')\lra \C(C)$, it is associated to the map
$$
\f{\phi}{C}{C'}{(X_0:\cdots :X_m)}{(Z_0:\cdots :Z_n)=(\phi^*(Y_0):\cdots :\phi^*(Y_n))}
$$
\end{proof}
\noindent From here follows that, since dessins d'enfants are certain morphisms from curves onto $\P^1$, via this equivalence, they must correspond to some extensions of the field of functions of $\P^1$, which is $\C(t)$. In order to speak about ramification of points in field extensions, we first need a way to speak about points in the context of a field. The way to do this is via the valuations of the field of functions.

\begin{definicion}
Let $K$ be a field. A (discrete) \textbf{valuation} on $K$ is a non-zero map $\nu:K^\times \lra \Z$ such that
\begin{itemize}
\item $\nu(ab)=\nu(a)+\nu(b)\ \forall a,b\in K^\times$
\item $\nu(a+b)\ge \min\{\nu(a),\nu(b)\}\ \forall a,b\in K^\times$
\item If $k$ is a subfield of $K$, and $\nu(k^\times)=0$, we say that $\nu$ is a $k$-valuation.
\end{itemize}
One usually extends a valuation to the whole field by setting $\nu(0)=+\infty$.
\end{definicion}

\noindent From the second property it follows that if $\nu(a)\neq \nu(b)$, then $\nu(a+b)=\min\{\nu(a),\nu(b)\}$, for if $\nu(a)<\nu(b)$ and $\nu(a)<\nu(a+b)$, then $\nu(a)=\nu(a+b-b)<\min\{\nu(a+b),\nu(-b)\}$.

It is easy to assign a $\C$-valuation to a point on a curve. One can take, for each function $f\in \C(C)$, the order at a point $P$ (which is either the order of the zero if $f(P)=0$, or minus the order of the pole if $f$ has a pole at $P$ or $0$ if $f$ neither has a zero nor a pole). It is straightforward to check that $f\mapsto \ord_P(f)$ is a $\C$-valuation. There are other valuations associated to $P$, since $m\cdot \ord_P$ for any $m\in \N$ is also a valuation. If two valuations differ by multiplication by a constant, we say that they are equivalent. If one asks for valuations to be surjective, every $\C$-valuation corresponds to a point:

\begin{prop}
Let $\C(C)$ be the field of functions of an algebraic curve. Then, the following sets are in bijective correspondence:
\begin{itemize}
\item The points in $C$.
\item The surjective $\C$-valuations of $\C(C)$.
\end{itemize}
And the correspondence is given by $P\longmapsto \ord_P$.
\end{prop}
\begin{proof}
The proof can be found in Proposition 3.17 of \cite{GG}.
\end{proof}

\noindent From now on, by ``valuation'' we will mean ``discrete $\C$-valuation''.

If we have a morphism between two curves $\phi:C\lra C'$, the morphism also maps valuations to valuations. By the chain law, for any function $f\in \C(C')$ and any point $P\in C$,
$$
\ord_P (\phi^*f) =\ord_P (f\circ \phi)=\ord_{\phi(P)}(f)e_P(\phi)
$$
Where $e_P(\phi)$ is the ramification index of $\phi$ at $P$, i.e. the order of the zero of $\phi$ at the point $P$, or the order with which the derivative of $\phi$ vanishes at $P$ plus 1.

This can give us the idea to define a map of valuations: for a valuation $\nu$ of $\C(C)$, one can define
$$
\phi(\nu)\sim\nu\circ \phi^*
$$
Where by $\sim$ we mean equal up to multiplication by a constant. Therefore, for every point $P$, we have
$$
\phi(\ord_P)=e_P(\phi)\ord_{\phi(P)}
$$
Since the valuation $\ord_{\phi(P)}$ is surjective, one can use this to define the ramification of a point by a $\C$-algebra homomorphism, and as we have just seen, it will match our previous definition of ramification.
\begin{definicion}
Let $\phi^*:\C(C')\lra \C(C)$ be a morphism between fields of functions. For every valuation $\nu$ of $\C(C)$, we define its image $\phi(\nu)$ as the unique valuation of $\C(C')$ such that there exists a natural number $e_\nu(\phi)$ with
$$
\nu\circ \phi^*=e_\nu(\phi)\phi(\nu)
$$
We call $e_\nu(\phi)$ the order of ramification of $\phi$ at the point $\nu$.
\end{definicion}
\noindent Note that, since $\phi^*$ is always injective, we can view $\C(C')$ as a subfield of $\C(C)$. If we do this, then the image of a valuation is just its restriction to the smaller field.

Since we know that valuations and points correspond to each other, we have some facts that we already know about points.
\begin{prop}
Let $\phi:C\lra C'$ be a morphism of curves. Then
\begin{enumerate}
\item The number of valuations $\nu$ of $\C(C)$ for which $e_\nu(\phi)>1$ is finite.
\item For every valuation $\nu$ of $\C(C')$, the number
$$
\sum_{\phi(\nu')=\nu} e_{\nu'}(\phi)
$$
is finite, constant and equal to $[\C(C):\phi^*\C(C')]$.
\end{enumerate}
\end{prop}

\noindent We will not prove or use the last part of the proposition, that the degree of the extension is the degree of the map, but we will prove it for dessins later on, in Proposition \ref{degree}.

Back to dessins d'enfants, a finite extension $\C(t)\subset K$ is the same as a morphism $i:\C(t)\lra K$, which corresponds to a morphism of curves $\phi:C\lra \P^1$. Thus, a dessin d'enfant can also be defined as an extension of $\C(t)$ that is ramified at most over $0$, $1$ and $\infty$ (that is, the valuations $\ord_0$, $\ord_1$ and $\ord_\infty$).
\begin{prop}
The category of dessins d'enfants, with coverings as morphisms, is equivalent to the category of finite extensions of $\C(t)$ unramified outside of $\{0,1,\infty\}$, with homomorphisms of extensions as morphisms.
\end{prop}
\begin{proof}
We have already seen the proof of everything, aside from identifying the morphisms. If we have a morphism between two dessins, which we see as Belyi functions $f:C\lra \P^1$ and $f:C':\lra \P^1$, a morphism is a rational function $\phi:C\lra C'$ such that $f'\circ \phi=f$. If we take the functor which maps curves to their fields of functions, this becomes $\phi^*\circ f'^*=f^*$, and if we see $f^*$ and $f'^*$ as inclusions of $\C(t)$ in the respective fields of functions, then this condition precisely means that $\phi^*$ preserves $\C(t)$. In other words, it is a homomorphism of extensions of $\C(t)$.
\end{proof}

\section{Monodromy and automorphisms}\label{monodromy}

In this section, we will see yet more characterizations of dessins d'enfants. It all boils down to the classification of covering spaces from topology. Results from this section can be found in Munkres' Topology \cite{Munkres}, in chapter 13. All coverings we will speak about will be connected.

Recall that, from a continuous map $f:X\lra Y$ such that $f(x_0)=y_0$, there is a homomorphism of the spaces' fundamental groups, called $f_*$, and it is given by $\gamma\longmapsto f\circ \gamma$.

\begin{prop}[Lifting lemma]
Let $p:C\lra X$ be a covering map, where $X$ is a locally path connected topological space. Let $Y$ be a locally path connected and path connected topological space and $f:Y\lra X$ be a continuous map. Choose points $c_0\in C$, $x_0\in X$ and $y_0\in Y$ such that $p(c_0)=x_0$ and $f(y_0)=x_0$. Then, we define a \textbf{lifting} $\widetilde f$ of $f$ to $C$ to be a map $\widetilde f:Y\lra C$ such that the following diagram commutes and $\widetilde f(y_0)=s_0$
\begin{center}\hspace{0.5 cm}
\begindc{\commdiag}[40]
\obj(0,0)[f]{$Y$}
\obj(1,0)[b]{$X$}
\obj(1,1)[d0]{$C$}
\mor{f}{b}{$f$}
\mor{d0}{b}{$p$}
\mor{f}{d0}{$\widetilde f$}[\atleft,\dashArrow]
\enddc
\end{center}
\noindent Such a lifting exists if and only if $f_*(\pi_1(Y,y_0))\subset p_*(\pi_1(C,c_0))$. If the lifting exists, it is unique.
\end{prop}
\begin{proof}
This is lemma 79.1 in Munkres' text \cite{Munkres}.
\end{proof}

\noindent Thus, for a covering with a base point $p:(C,c_0)\lra (X,x_0)$, the group $p_*(\pi_1(X,x_0))$ plays a very important role. This is why we call it \textbf{the subgroup associated to the covering} (note that it depends on the choice of a base point). The other main result is the following: for nice enough spaces, to every subgroup corresponds a covering.

\begin{prop}\label{cubiertapunto1}
Let $X$ be semi-locally simply connected (for example, let $X$ be a manifold), with a fixed base point $x_0$. Then, for every subgroup $H<\pi_1(X,x_0)$, there exists a covering $p:C\lra X$ such that $p_*(\pi_1(C,c_0))=H$
\end{prop}
\begin{proof}
This is theorem 82.1 in \cite{Munkres}.
\end{proof}

\noindent From the lifting lemma, it follows that the covering associated to a subgroup is unique, up to isomorphism. For two coverings $p_1:(C_1,c_1)\lra (X,x_0)$ and $p_2:(C_2,c_2)\lra (X,x_0)$, we define a morphism between them to be a map $\phi:C_1\lra C_2$ such that $p_1=p_2\circ \phi$, as we did with dessins. A morphism with base points is the same with the additional condition that $\phi(c_1)=c_2$.

\begin{prop}\label{uniquemorphism}
Let $p_1:(C_1,c_1)\lra (X,x_0)$ and $p_2:(C_2,c_2)\lra (X,x_0)$ be two coverings, with associated subgroups $H_1$ and $H_2$, respectively. Then, there exists a morphism with base points from $C_1$ to $C_2$ if and only if $H_1 <H_2$. If it exists, it is unique.
\end{prop}
\begin{proof}
A morphism with base points is the same as a lifting of the map $p_1$ to the covering $p_2$. With this in mind, it follows from the lifting lemma.
\end{proof}

\begin{corol}\label{cubiertapunto2}
Two covering spaces with the same associated subgroup are isomorphic with base points.
\end{corol}

\noindent We want to remove the base point, and to see when two coverings are isomorphic regardless of the base point. If two coverings are isomorphic, then they must be isomorphic with a suitable choice of base points. Now, if we have a covering $p:(C,c_0)\lra (X,x_0)$, and we choose another point $c_1$ such that $p(c_1)=x_0$, we know that $\pi_1(C,c_0)$ and $\pi_1(C,c_1)$ are related in the following way: if we choose a path $\eta$ from $c_0$ to $c_1$, then $\pi_1(C,c_1)=\eta * \pi_1(C,c_0)*\overset{\leftarrow}\eta$, where $*$ denotes path concatenation. We can apply $p_*$ to the previous relation, and $p_*(\eta)$ will be a loop, since $p(c_0)=p(c_1)$. Therefore, $H_1=p_*(\eta) H_2p_*(\eta)^{-1}$.

\begin{prop}\label{morcubierta}
Let $p_1:(C_1,c_1)\lra (X,x_0)$ and $p_2:(C_2,c_2)\lra (X,x_0)$ be two coverings, with associated subgroups $H_1$ and $H_2$. There exists a morphism of coverings $\phi:C_1\lra C_2$ that doesn't preserve the base points if and only if $H_1\subset gH_2g^{-1}$ for some $g\in \pi_1(X,x_0)$.
\end{prop}
\begin{proof}
Suppose there is such a morphism. Then, if $c_2'=\phi(c_1)$, $\phi$ is a morphism from $(C_1,c_1)$ to $(C_2,c_2')$ that does preserve the base points. Therefore, $H_1$ is contained in the subgroup associated to $(C_2,c_2')$. But changing the base point of a covering space gives a conjugate subgroup, so we have what we want: $H_1\subset gH_2g^{-1}$.

Conversely, suppose that $H_1\subset gH_2g^{-1}$. Pick a loop $\eta$ whose class in the fundamental group of $X$ is $g^{-1}$. We can lift this path to $C_2$, starting from $c_2$ (by the lifting lemma). Call $c_2'$ the endpoint of the lifted path. Then, by what we seen before the proposition, the subgroup associated to $(C_2,c_2')$ is $gHg^{-1}$. Therefore, there exists a morphism from $(C_1,c_1)$ to $(C_2,c_2')$, which can be seen as a morphism to $(C_2,c_2)$ that doesn't preserve the base points.
\end{proof}

\begin{corol}
Two covering spaces are isomorphic if and only if their associated subgroups are conjugate.
\end{corol}

\noindent We are seeing that the subgroups of the fundamental group and the coverings of a space are closely related. Another way of seeing this relation is the \textbf{monodromy action}. Given a covering space $p:C\lra (X,x)$, let $F=p^{-1}(\{x\})$. We can make $\pi_1(X,x)$ act on $F$ on the following way: take a loop $\eta$ in $\pi_1(X,x)$ and $c\in F$. As we did in the proof of the previous proposition, lift it to a path $\widetilde \eta$ in $C$ so that it starts on $c$. Then, call $\eta(c)=\eta(1)$. It is easy to check that this doesn't depend of the homotopy class of $\eta$. Also, it is indeed an action of $\pi_1(X,x)$ on $F$: if we have two loops $\eta$ and $\gamma$, the action of $\eta$ followed by the action of $\gamma$ is clearly the same as the action of $\eta *\gamma$. Thus, we have a right action. It is also transitive, since $C$ is path connected.

Also, note that $\eta(c)=c$ if and only if $\eta$ lifts to a loop $\widetilde \eta$ in $C$. This means in turn that $\eta=p_*(\widetilde \eta)\in p_*(\pi_1(C,c))$. Therefore, the stabilizer of the point $c$ by the monodromy action is the subgroup associated to $(C,c)$. Since we have a transitive action, the fiber $F$ of $x$ is in bijection with $H\backslash\pi_1(X,x)$, where $H$ is the subgroup associated to $(C,c)$. In particular, the degree of the covering, which is the cardinality of $F$, is equal to the index of $H$ in $\pi_1(X,x)$.

This gives two ways to look at a covering: as a conjugacy class of subgroups, and as a transitive $\pi_1(X,x)$-action. We can apply this to dessins: we are looking at covers of $\P^1\setminus \{0,1,\infty\}$, and the fundamental group of $\P^1\setminus \{0,1,\infty\}$ is $F_2$, the free group on two generators. A dessin can also be seen as the conjugacy class of a finite index subgroup of $F_2$. We will fix $1/2\in \P^1$ as the base point and we will pick two canonical generators for this group: $x$, which can be parametrized as $t\mapsto \frac{1}{2}e^{2\pi it}$, and $y$, which is $t\mapsto 1-\frac{1}{2}e^{2\pi it}$. In other words, $x$ goes around $0$ and $y$ goes around $1$.

Dessins are in correspondence with conjugacy classes of subgroups of $F_2$ of finite index. Also, if one takes the monodromy action, a dessin can also be seen as a transitive right action of $F_2$ on a finite set.

From the fact that dessins are equivalent to (right) $F_2$-actions, we can give yet another definition, which consists of, given an action, consider the induced morphism $f:F_2\lra S_F$ on the group of permutations of $F$ (which we see as also acting on the right on $F$). Then, a dessin can be seen as a homomorphism from $F_2$ to $S_n$ such that its image is a transitive subgroup of $S_n$. Dessins are equivalent if these morphisms are conjugate in $S_n$. The subgroup associated to the cover is then the preimage of the stabilizer of any point (since all the stabilizers are conjugate). Given a dessin, we call the image of $F_2$ in a permutation group the \textbf{cartographic group} of the dessin.

With the point of view of actions, morphisms between dessins are as follows: if one has a dessin $(C,f)$ and the transitive action of $F_2$ on the fiber $X$ given by the monodromy, and we have a morphism to another dessin $(C',f')$, we can consider a base point in $C$ and its image in $C'$, which gives an inclusion between a group $H$ associated to $(C,f)$ and a group $H'$ associated to $(C'.f')$. Then, $X\cong H\backslash F_2$, and $H$ is contained in $H'$, so $H'$ is the union of some cosets of $H$. Therefore, the cosets of $H'$ form a partition of $H\backslash F_2$. In the set $X$, we are just quotienting by the relation $a\sim b\Leftrightarrow \exists h\in H':a^h=b$. This quotient is simply $H'\backslash F_2$, so $F_2$ acts on it by the induced action, given by $[a]^g=[a^g]$. Thus, a morphism between actions is a map $\phi:X\lra X'$ that commutes with the action of $F_2$, that is, $\phi(x^g)=\phi(x)^g$. Since we are requiring all actions to be transitive, such a map will be surjective.

Finally, $F_2$ is a special group, since it satisfies a universal property: for every group $G$ and every two elements $\overline x,\overline y\in G$, there exists a unique homomorphism $\pi:F_2\lra G$ such that $\pi(x)=\overline x$ and $\pi(y)=\overline y$. For a subgroup $H<F_2$, we can consider the biggest normal subgroup contained in it, which is called the core, and it is $\cor_{F_2} H=\bigcap_{g\in F_2} g^{-1}Hg$ (which is of finite index, bounded by $[F_2:H]!$), we can consider the map $\pi:F_2\lra F_2/\cor_{F_2}H$. This gives a homomorphism from $F_2$ onto a group $G$ with a distinguished subgroup, $\pi(H)$, that satisfies that $\cor_{G}\pi(H)=1$. Conversely, given a group $G$ with two generators $\overline x,\overline y$, and a subgroup $\overline H$ such that $\cor_{\overline G}\overline H=1$, there is a subgroup $H$ of $F_2$ and a map $\pi$ taking $H$ to $\overline H$ and the generators of $F_2$ to $\overline x,\overline y$.

To sum up:

\begin{prop}
Dessins d'enfants are in correspondence with the following sets:
\begin{enumerate}
\item The finite index subgroups of $F_2$, up to conjugation. A dessin $C_1$ covers another one $C_2$ if and only if their corresponding subgroups $H_1$ and $H_2$ satisfy that $H_1\subset g^{-1}H_2g$ for some $g\in F_2$.
\item Finite groups $G$ with two distinguished generators $\overline x$ and $\overline y$ and a distinguished subgroup $\overline H$, such that $\cor_G \overline H=1$.
\end{enumerate}
\end{prop}
\begin{prop}
The category of dessins d'enfants is equivalent to the following categories:
\begin{enumerate}
\item Transitive right actions of $F_2$ on finite sets. A morphism between two sets with $F_2$ actions is a map $\phi:X\lra X'$ such that for all $g\in F_2$ and every $x\in X$, $g(\phi(x))=\phi(g(x))$. 
\item Pairs of permutations $\overline x,\overline y\in S_n$ that generate transitive subgroups or, equivalently, morphisms from $F_2$ to $S_n$ such that $x$ and $y$ are mapped to such generators. A morphism between $\overline x_1,\overline y_1\in S_n$ and $\overline x_2,\overline y_2\in S_m$ is a map $\phi:\{1,\ldots ,n\}\lra \{1,\ldots ,m\}$ such that $\phi\circ \overline x_1=\overline x_2 \circ \phi$, and $\phi\circ \overline y_1=\overline y_2 \circ \phi$.
\end{enumerate}
\end{prop}

\section{(Children's) Drawings on surfaces}

We can look at the monodromy in a much more geometric way. Let us consider the following triangulation of the sphere: Take the real line, which is a circle that passes through $0$, $1$ and $\infty$, and make these three points vertices and make the segments the real line is divided in by these points into edges. The hemispheres are two triangles with vertices $0$, $1$ and $\infty$. We will color the triangle with positive imaginary part white, and the other one black.

\begin{figure}[h!]
\centering
\includegraphics[scale=0.3]{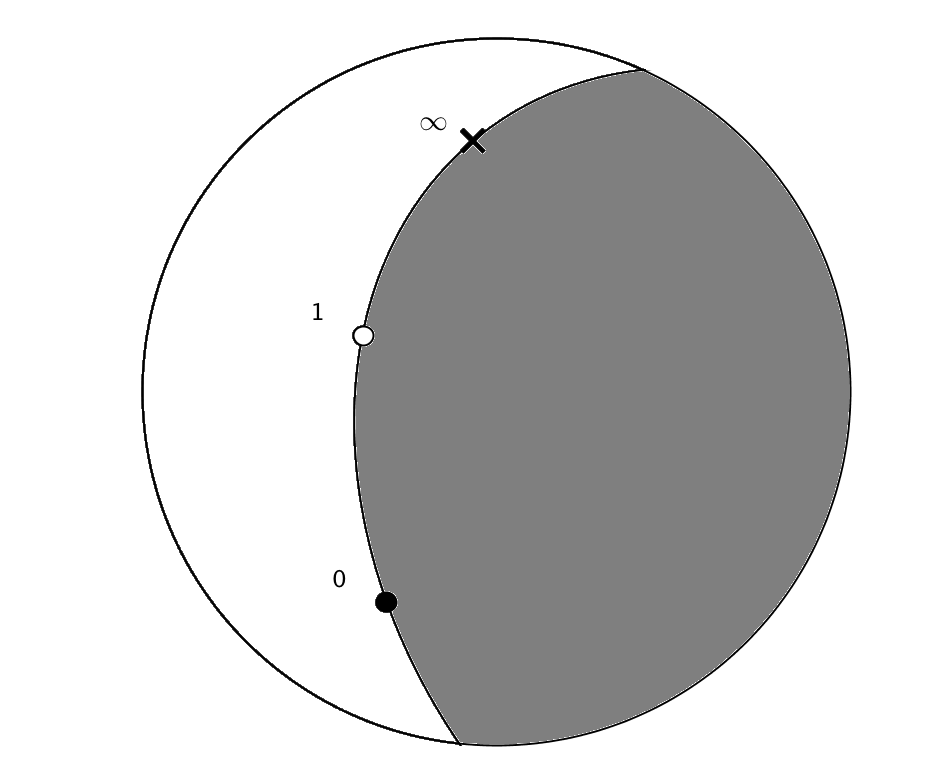}
\caption{The triangulation on the sphere}
\end{figure}

\noindent If we have a dessin d'enfant, the triangulation can be lifted to the covering surface, including the ramification points. We can do this as follows: let $p:C\lra \P^1$ be the covering map. If we take the white triangle and remove its vertices, we can call it $T$ and consider the inclusion $i:T\lra \P^1$. Now, for every point $P_j$ in $p^{-1}(1/2)$, we can lift the inclusion (by the lifting lemma) to a map $\widetilde i_j$ so that $\widetilde i_j(1/2)=P_j$. If we do this for the black triangle as well, we obtain the triangulation of $C$ except for the ramification points. Note that every triangle and every edge are lifted to as many triangles and edges as the degree of the covering. We take the ramification points as vertices, and, since locally the covering looks like $z\mapsto z^e$, each point will have $2e$ edges around it and $e$ triangles of each color.

\begin{figure}[h!]
\centering
\includegraphics[scale=0.3]{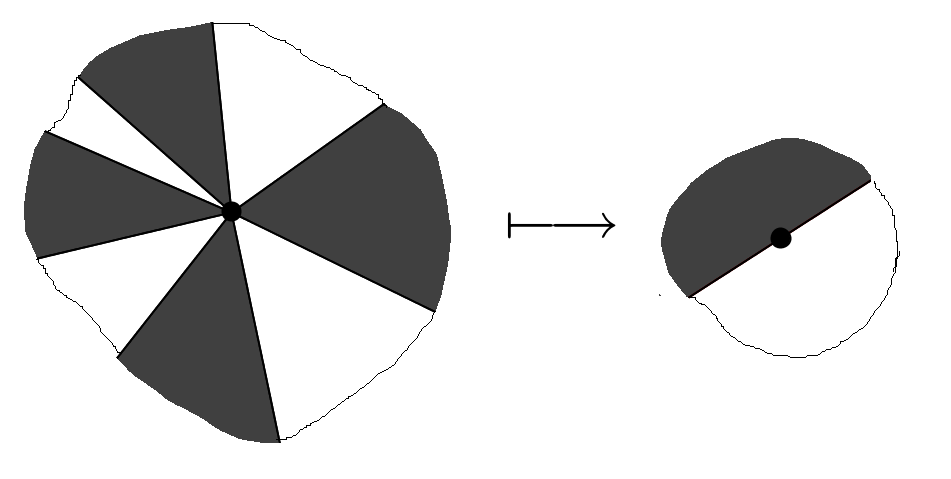}
\caption{A neighborhood of a point with ramification index 4}
\end{figure}

\noindent Thus there is a triangulation on the surface such that every face of it is mapped to a face of the triangulation on $\P^1$. We can color the points on $C$ depending on which image they have. We will say that the preimages of $0$ are black points, the preimages of $1$ are white points and the preimages of $\infty$ are stars. We have a surface with a triangulation with vertices of three kinds, and such that each triangle has a vertex of every kind.

We can go the other way round: Start from a triangulated oriented surface with black vertices, white vertices and stars such that each triangle has a vertex of each kind. Since the surface is oriented, for every triangle there is a notion of the counterclockwise order of its vertices. Color each triangle such that the vertices are ordered black-white-star in counterclockwise order white and every other triangle black. Note that this implies that every edge is shared by a triangle of each color. Now, we can map black vertices to 0, white vertices to 1, stars to $\infty$, white triangles to the white triangle on the sphere, black triangles to the black triangle on the sphere and edges of triangles to the corresponding segment: this gives a covering of the sphere.

We have proven that dessins are equivalent to triangulated orientable surfaces with three-colored vertices. There is yet another way to represent a dessin: From such a surface, if we delete every star vertex and the edges reaching it, we are left with a bicolored graph embedded in a surface. Each star vertex is a vertex of $e$ triangles of each color, and when we delete the vertex, these triangles turn into a face with $2e$ edges. We can also turn this the other way round. If we start with a  bicolored graph embedded in an orientable surface, such that the faces are homeomorphic to disks, we can add a star vertex to each face, join it with each vertex on the face and we have a triangulated surface as before. This is the reason why Grothendieck called them dessins d'enfants (children's drawings): because they can just be seen as a drawing on a surface.

From a Belyi function, one can just take the preimage of the segment $[0,1]$, and this will give the graph on the surface.

 As an example, we can take the cube dessin on the sphere (we are identifying $\C\subset \P^1$ with the plane as usual, so every picture has an outer face and a point at infinity):

\begin{figure}[h!]
\centering
\includegraphics[scale=0.12]{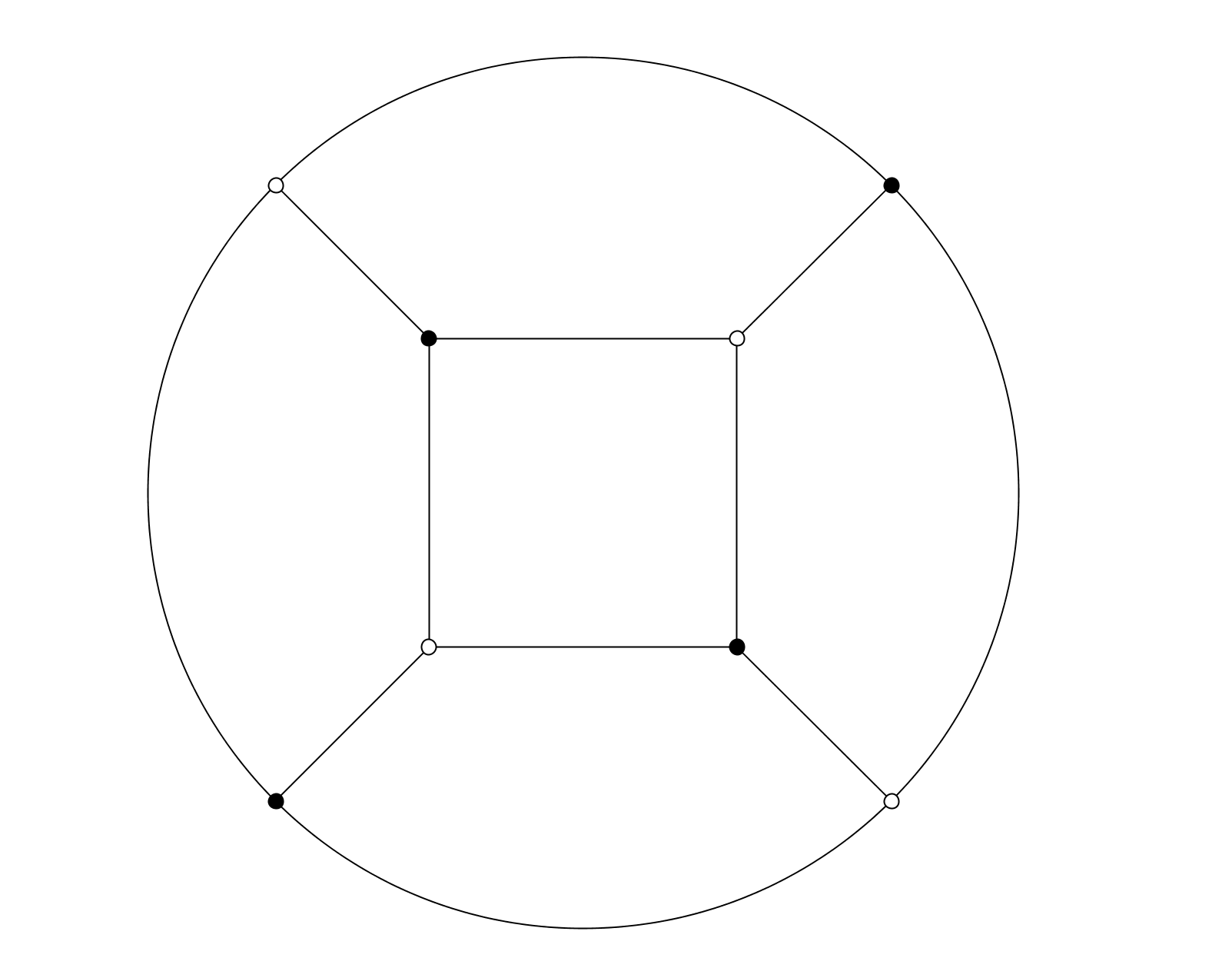}
\includegraphics[scale=0.12]{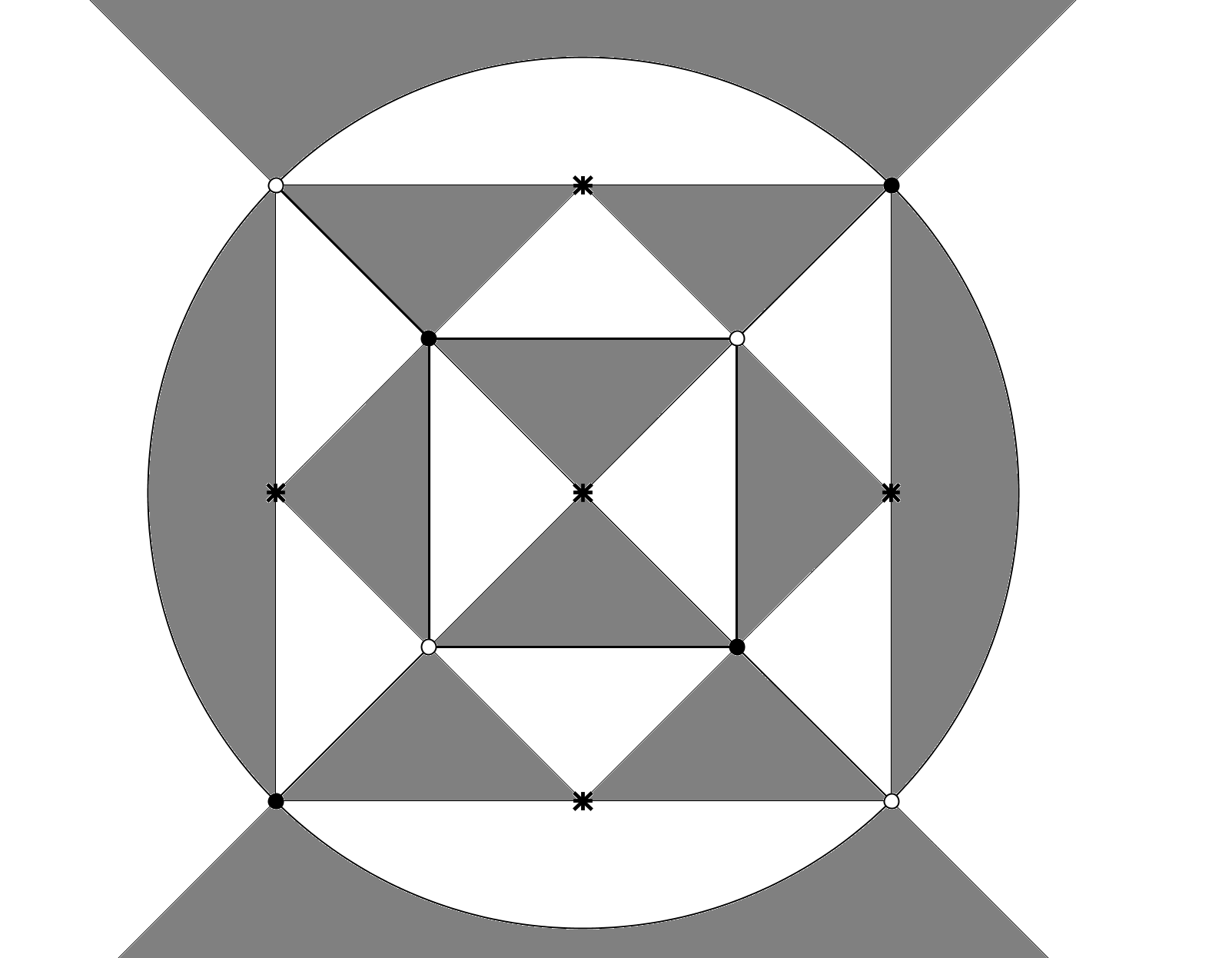}
\caption{The cube with and without the star vertices drawn}
\end{figure}
\noindent Let us take a moment to look at this example from all the points of view we have seen so far. From the triangulated surface, a covering of the sphere is determined. This, in turn, determines a complex structure on the sphere and a holomorphic map. This holomorphic map is given by the equation
$$
f(z)=-\frac{1}{12\sqrt 3}\frac{(z^4-2\sqrt 3 z^2-1)^3}{z^2(z^4+1)^2}=1-\frac{1}{12\sqrt 3}\frac{(z^4+2\sqrt 3 z^2-1)^3}{z^2(z^4+1)^2}
$$
We will not stop to see how the equation for the map is obtained. There are techniques for this that can be found in \cite{Guillot} and \cite{Lando}, chapter 2. We can, however, check that this map has degree 12, and that it is invariant under the Möbius transformations$$
z\mapsto -\frac{\sqrt 2 z+1+i}{(1-i) z-\sqrt 2},
z\mapsto -\frac{\sqrt 2 z+1-i}{(1+i) z-\sqrt 2}$$
These generate the symmetries of the cube, which is a group isomorphic to $A_4$. Since $f$ is invariant under the action of this group, it must factor through the quotient
$$
f:\P^1\lra \P^1/A_4 \lra \P^1
$$
However, since the first arrow has degree 12, and $f$ has degree 12, the second arrow must have degree 1, i.e. it must be an isomorphism of $\P^1$ with itself. So $f$ is the map we are looking for.

Given a Belyi map, the graph embedded in the surface is given by the preimage of $[0,1]$, with the preimage of 0 as black points and the preimage of 1 as white points. In this particular case, we can plot the preimage of $[0,1]$, to obtain the graph shown in Figure \ref{plot}.

\begin{figure}[h!]
\centering
\includegraphics[scale=0.2]{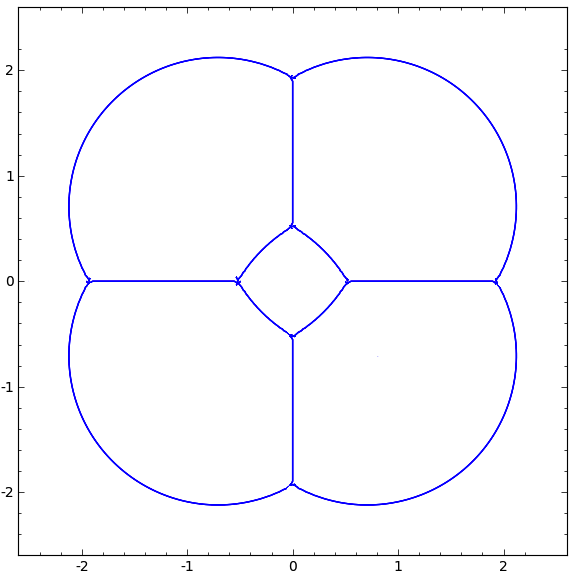}
\caption{The preimage of $[0,1]$ by the map $f$.} \label{plot}
\end{figure}

\noindent The field extension we are looking at is the homomorphism of $\C(t)$ into itself given by substitution of the variable by $f$. Thus, the extension is $\C(f)\subset \C(z)$.

We can also look at the monodromy. We want to know the action of the loops $x$ and $y$ which we described earlier on a fiber of the map. We can take any point in the interior of $[0,1]$. For instance, $1/2$. Now, for a point in the preimage of $1/2$, that is, a point in one of the edges of the dessin, lifting the loop around $0$ is just turning counterclockwise around the black vertex incident to that edge until we meet another edge.

\begin{figure}[h!]
\centering
\includegraphics[scale=0.18]{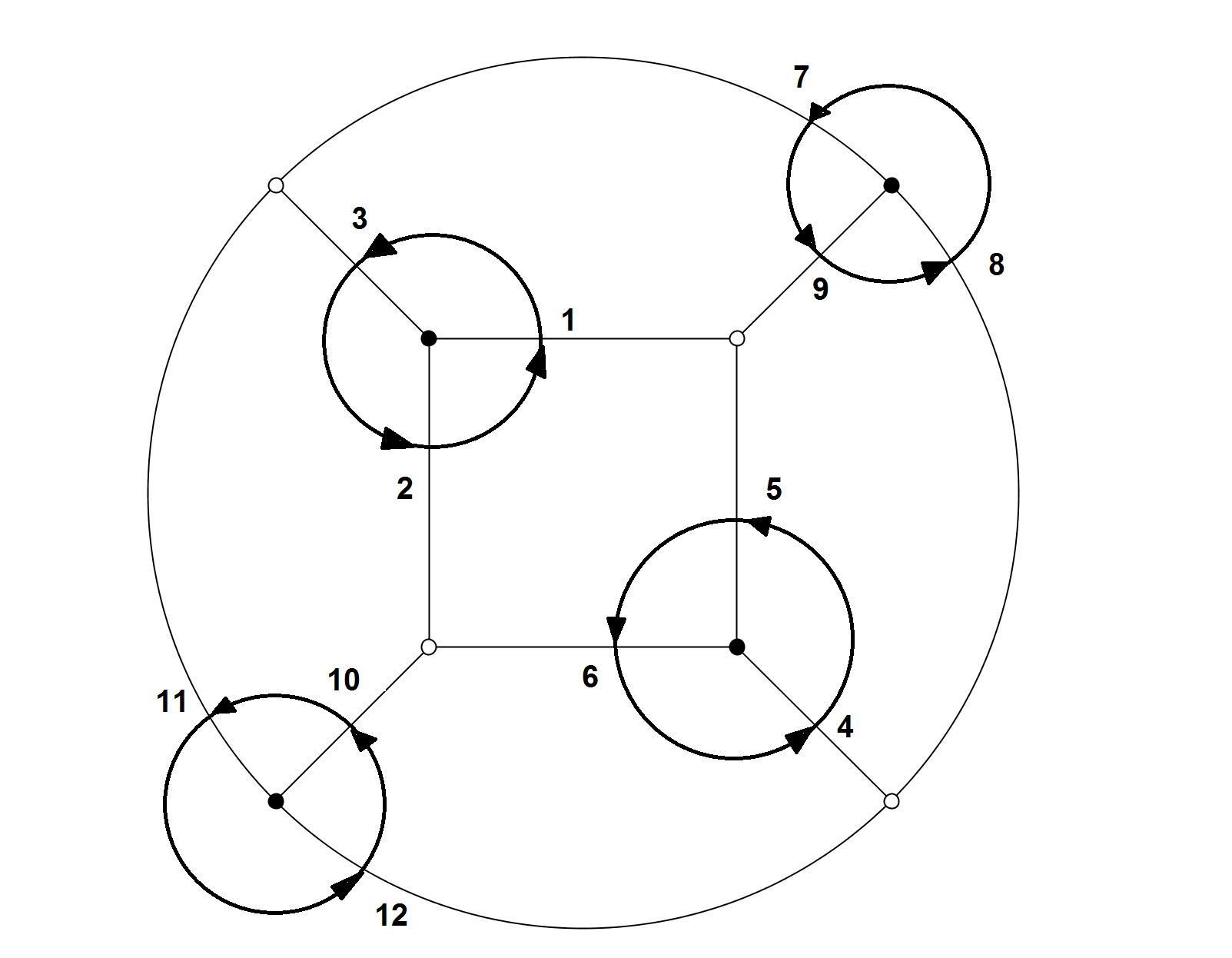}
\caption{The 12 liftings of the path $x$ to the cube.}
\end{figure}
\noindent Thus, every edge is mapped to another edge, and we get a permutation of $12$ elements. We can do the same with the loop around $1$. In this case, the map we obtain is
$$
x\longmapsto (1,3,2)(4,5,6)(7,9,8)(10,11,12)
$$
$$
y\longmapsto (1,5,9)(2,10,6)(3,7,11)(4,12,8)
$$
And also, the subgroup of $F_2$ which is the stabilizer of a point has index 12, since this is the size of the orbit (recall that these are equal!), so it is the preimage of the identity. Therefore, it is the normal subgroup generated by
$$
\{ x^3,y^3,(xy)^2\}
$$
Since these are relations that give $A_4$ with these generators.

Using the cube as an example, we have seen something important: we can know the monodromy action by just looking at the picture! This will turn out to be very useful for us.
\section{Regular dessins}\label{secregular}
The cube dessin was an example of a regular dessin. A regular dessin corresponds with many familiar concepts: a regular (or normal) cover, a normal subgroup, a proper action and a Galois extension. We will prove that all these notions coincide in a dessin d'enfant when we look at it from the various points of view.

\begin{definicion}
An \textbf{automorphism} of a dessin is an invertible morphism from the dessin to itself.
\end{definicion}

\noindent No surprises here. Recall what we mean by a morphism in the different contexts: if we have a Belyi pair $(C,f)$, it is a map $\phi:C\lra C$ such that $f\circ \phi=f$; if we have a field extension, it is a morphism of extensions, i.e., a homomorphism $\phi:\C(C)\lra \C(C)$ such that $\phi|_{\C(t)}=\Id$, and if we have an action, it is a map commuting with the group action. In any of the formulations, it is clear that any morphism from a dessin to itself is invertible, and it is therefore an automorphism. Also, we can speak of the automorphism group in any of these settings, since we have seen that they are all equivalent.
\begin{prop}\label{regular}
Let $(C,f)$ be a Belyi pair, let $\C(t)\subset\C(t)$ be the corresponding field extension, and let $F_2\lra X$ be the corresponding monodromy action. Let $H$ be the subgroup of $F_2$ associated with the action. The following are equivalent:
\begin{enumerate}
\item The order of the automorphism group of the dessin is equal to its degree.
\item $(C,f)$ is a regular cover. That is, there exists a subgroup $G<\Aut (C)$ such that the map is equivalent to $C\lra C/G$.
\item The extension $\C(C)/\C(t)$ is a Galois extension.
\item $H$ is a normal subgroup of $F_2$.
\item The action on $X$ of the cartographic group is proper, i.e. if for any $x\in X$ and $g\in F_2$, we have that $x^g=x$, then $g=1$.
\item The order of the cartographic group is equal to the degree of the dessin.
\end{enumerate}
\noindent When any of these is true, we say that the dessin is a \textbf{regular dessin}.
\end{prop}
\begin{proof}\ \\ \begin{itemize}

\item[1$\Leftrightarrow$ 2)] We have seen the idea for this when we talked about the cube. Let $G$ be the automorphism group of the cover. $G$ is a finite group, since its elements are determined by the image of an unramified point (therefore, its order is at most the degree of the dessin). Since $G$ is a finite group, its action is properly discontinuous, and the quotient by it is a holomorphic map to another Riemann surface\footnote{See, for instance, Proposition 2.21 in \cite{GG}, which proves it for Fuchsian groups, but the same proof applies. Nonetheless, the proof is essentially removing the fixed points and applying Proposition \ref{cubierta}}. $f$ must factor through this map, so there is a map $\overline f$ which fits in the diagram
\begin{center}\hspace{0.5 cm}
\begindc{\commdiag}[40]
\obj(0,1)[f]{$C$}
\obj(0,0)[b]{$C/G$}
\obj(1,1)[d0]{$\P^1$}
\mor{f}{b}{$\pi$}
\mor{b}{d0}{$\overline f$}[\atright, 0]
\mor{f}{d0}{$ f$}
\enddc
\end{center}
\noindent Now we look at the degrees of the maps: $G$ acts properly on the points on which $f$ is not ramified, since an automorphism that fixes one of them must be the identity. Therefore, the degree of $\pi$ equals $|G|$. however, we are assuming that $|G|$ is also the degree of $f$, so the degree of $\overline f$ must be 1. That is, $\overline f$ is an isomorphism and $\pi$ is equivalent to $f$.

Reciprocally, if the map is $C\lra C/G$, then the group $G$ is the group of automorphisms of the cover, since it obviously preserves the map, and therefore its order is at least the order of the fibers, which is the degree of the dessin. We have just seen that the order of the automorphism group is never bigger than the order of the fibers, so we are done.
\item[2 $\Leftrightarrow$ 3)] Suppose that the Belyi map is $C\lra C/G$, for some automorphism group $G$. $\C(C/G)$ is a subfield of $\C(C)$, since $\pi^*$ (with $\pi$ as in the previous part) is a field homomorphism from $\C(C/G)$ to $\C(C)$. The functions which are defined on $C/G$ are the functions on $C$ that factor through the quotient, i.e. the ones that $G$ preserves. Thus, if we take $G^*=\{\phi^*:\phi\in G\}$, which is a subgroup of the Galois group $\Gal(\C(C)/\C(t))$, the field $\C(C/H)$ is the fixed field of $G^*$. The premise is that $\pi:C\lra C/G$ is the Belyi map, or in other words, $\pi^*:\C(C/G)\lra \C(C)$ is the inclusion of $\C(t)$, so $\C(C/G)=\C(t)$. This means that the fixed field of $G^*$ is $\C(t)$, which in turn is the definition that the extension is Galois, and $G^*$ must be the whole Galois group.

Now, suppose the extension $\C(C)/\C(t)$ is Galois. Then, if we call $G^*=\Gal(\C(C)/\C(t))$, we have that the fixed field of $G^*$ is $\C(t)$. Every automorphism in $G^*$ corresponds to an automorphism of $C$ which preserves $t$, or in other words, an automorphism of $(C,f)$. Therefore, if we call $G<\Aut(C)$ the group corresponding to this Galois group, we have, as we had in the previous paragraph, that $\C(t)=\C(C/G)$, so the map is equivalent to $C\lra C/G$.

\item[1 $\Rightarrow$ 4)] We are going to prove that the automorphism group is isomorphic to $N_G(H)/H$.

Suppose we have an action of $F_2$ on a set $X$, and an automorphism $\phi:X\lra X$. Pick an element $x_0\in X$, and let $H$ be its stabilizer. First, we are going to prove that $\phi$ is determined by $\phi(x_0)$: indeed, since the action is transitive, every other $x\in X$ is $x_0^g$ for some $g\in F_2$. Then, (recall that we have defined the monodromy action on the right)
$$
\phi(x)=\phi(x_0^g)=\phi(x_0)^g
$$
If $\phi$ is an automorphism. Now, let $h\in F_2$ be such that $\phi(x_0)=x_0^h$. Then, $\phi$, as we have seen, must be $\phi(x_0^g)=\phi(x_0)^g=x_0^{hg}$, for every $g\in F_2$. In particular, this must be well-defined: if $g_1,g_2$ are such that $x_0^{g_1}=x_0^{g_2}$, we must have
$$
x_0^{hg_1}=\phi(x_0^{g_1})=\phi(x_0^{g_2})=x_0^{hg_2}\Longleftrightarrow x_0^{hg_2g_1^{-1}h^{-1}}=x_0\Llra hg_2g_1^{-1}h^{-1}\in H$$
Since $x_0^{g_1}=x_0^{g_2}$ is equivalent to $g_2g_1^{-1}\in H$, what we are saying is precisely that
$$
hHh^{-1}=H
$$
In other words, $h\in N_G(H)$. Also, $h$ induces the identity if and only if $x_0^h=x_0$, so the automorphism group is in correspondence $N_G(H)/H$. If its size is the degree of the dessin, which is the index of $H$, this means that $N_G(H)=F_2$, which is what we are trying to prove.
\item[4 $\Rightarrow$ 5)] Let $H$, as before, be the stabilizer of $x_0$. Then, the stabilizer of $x_0^g$ is $g^{-1}Hg$. If $H\lhd F_2$, all the stabilizers are the same, and the intersection of all of them is $H$. Then, the morphism $\phi:F_2\lra \mathcal C$ that has the cartographic group as its image has $H$ as its kernel, since it is the intersection of all the stabilizers. Therefore, a permutation stabilizes a point if and only if it is in the image of $H^g$ for some $g$, so it's in the kernel of the action.
\item[5 $\Rightarrow$ 6)] The order of the cartographic group is the product of the order of the orbit, which is the degree of the dessin, by the order of the stabilizer of a point, which is 1, by hypothesis (since an element that fixes a point fixes all of them). Therefore, the order of the cartographic group must be equal to the degree of the dessin.
\item[6 $\Rightarrow$ 1)] We know that the cartographic group is $F_2/\cor_{F_2}(H)$. If the order of this group is the degree of the dessin, which is $[F_2:H]$, this means that $H=\cor_{F_2}(H)$, or in other words, $H$ is normal in $F_2$. We have seen in the proof of $1\Longrightarrow 4$ that the automorphism group is isomorphic to $N_{F_2}(H)/H$. Then, it follows that its order is the degree of the dessin.
\end{itemize}
\end{proof}

\noindent There is one important conclusion in this proof: if a dessin is given by a subgroup $H$ of $F_2$, its automorphism group is the quotient $N_{F_2}(H)/H$. It acts on points of $X$ on the left: if we let $\phi_{h_1},\phi_{h_2}$ be the automorphisms that map $x_0$ to $x_0^{h_1}$ and $x_0^{h_2}$, for some $h_1,h_2\in N_{F_2}(H)$, then
$$
\phi_{h_1}(\phi_{h_2}(x_0))=\phi_{h_1}(x_0^{h_2})=\phi_{h_1}(x_0)^{h_2}=x_0^{h_1h_2}=\phi_{h_1h_2}(x_0)
$$
This is why we say that the automorphism group is $N_{F_2}(H)/H$.

We are going to make an important remark about notation: when we go from the automorphism group of a dessin, which is a subgroup of $F_2$ acting on $H\backslash F_2$ \textbf{on the left}, to a Galois group, we take a contravariant functor. Thus, for two elements $g_1,g_2\in F_2$, or in $N_{F_2}(H)/H$, we have that $(g_1g_2)^*$, that lies in some Galois group, equals
$$
(g_1\circ g_2)^*=g_2^*\circ g_1^*
$$
For this reason, we are going to adopt the convention that Galois groups always act \textbf{on the right}. This means that if $\sigma,\tau \in \Gal(F/K)$, the product $\sigma\tau$ is defined to be $\tau \circ \sigma$. If $f\in F$, we will write $\sigma$ acting on $f$ as $f^\sigma$, so we have that
$$
(f^\sigma)^\tau=f^{\sigma\tau}
$$
There is another conclusion to be drawn from the previous theorem: it follows that for a regular dessin, the degree of the field extension is equal to the degree of the map. In fact, the following is true:

\begin{prop}\label{degree2}
Let $C$ be a curve, and let $f\in \C(C)$. Then, the degree of $f$ equals $[\C(C):\C(f)]$.
\end{prop}

\noindent The proof for this fact can be found in \cite{GG}, in chapter 1. However, we will not use it. We will prove the weaker statement that this only applies to Belyi pairs, and it is the only version of this statement we will use. To do it, we can use the regular cover.

\begin{prop}\label{autmono}
Let $(C,f)$ be a Belyi pair, with corresponding group $H<F_2$. Then, the regular cover of $(C,f)$ corresponds to the Galois closure of $\C(C)/\C(f)$ and to the group $\cor_{F_2}H$.
\end{prop}
\begin{proof}
The most important part of the proof is the fact that the Galois closure of an extension isn't ramified at any new points. Take $\C(C')/\C(f)$ to be the Galois closure of $\C(C)/\C(f)$. Also, take the regular cover $(\widetilde C,\widetilde f)$. We wish to see that $\widetilde C= C'$. The field extension $\C(\widetilde C)/\C(f)$ is Galois, so it must contain $\C(C')$. However, if the extension $\C(C')/\C(f)$ was ramified at a point other than $0,1,\infty$, the extension $\C(\widetilde C)/\C(f)$ would be too, but this is not the case.

Therefore, the three objects exist and satisfy universal properties that correspond to each other. Take the Galois closure of $\C(C)/\C(f)$. It corresponds to a regular Belyi pair $(\widetilde C,\widetilde f)$, by Proposition \ref{regular}. If this Belyi pair wasn't minimal, the minimal belyi pair would correspond to a smaller Galois extension. Also, since covers correspond to subgroups of $F_2$ and regular covers to normal subgroups, the maximal normal subgroup, which is $\cor_{F_2} H$, must correspond to the minimal regular cover.
\end{proof}
\begin{prop}\label{degree}
Let $(C,f)$ be a Belyi pair. The degree of $f$ equals $[\C(C):\C(f)]$.
\end{prop}
\begin{proof}
Take its regular closure $(\widetilde C,\widetilde f)$, which we have just seen corresponds to the Galois closure of $\C(C)/\C(f)$. Now, by the Galois correspondence, take the subgroup $H^*<\Gal(\C(\widetilde C)/\C(f))$ such that its fixed field is $\C(C)$. In the proof of \ref{regular}. we saw that the fixed field of $H^*$ corresponded to $\C(\widetilde C/H)$. Now,
$$[\C(C):\C(f)]=\frac{[\C(\widetilde C):\C(f)]}{[\C(\widetilde C):\C(C)]}=\frac{[\C(\widetilde C):\C(f)]}{[\C(\widetilde C):\C(\widetilde C/H)]}$$
Since it's a regular dessin, $[\C(\widetilde C):\C(f)]=\deg \widetilde f$, and since the extension in the denominator is Galois, its degree is $|H|$, which is equal to the degree of the map $C\lra C/H$. Therefore,
$$
[\C(C):\C(f)]=\frac{[\C(\widetilde C):\C(f)]}{[\C(\widetilde C):\C(C)]}=\frac{\deg\left(
\widetilde C\lra \P^1
\right)}{\deg\left(
\widetilde C\lra C
\right)}=\deg\left(C\lra \P^1\right)
$$
\end{proof}
\noindent We have artfully avoided the pain of proving Proposition \ref{degree2} for general curves and maps.
\section{The field $\KK$}\label{thefieldk}
We are going to see dessins d'enfants, instead of as abstract extensions, as subextensions of one given extension $\KK/\C(t)$, very much in the same way as we see coverings of a space as subcoverings of its regular cover. For this section, we will also ask the reader to keep in mind that $\C$ plays no special role, and that it can be replaced by $\QQ$ when the moment comes.

We are going to construct the field that is the union of all the extensions of $\C(t)$ unramified outside of $\{0,1,\infty\}$. To do this, consider the set $\mathcal H=\{H<F_2:[F_2:H]<\infty \}$. As we know, we can make each of these subgroups correspond to a dessin d'enfant, in fact, a dessin d'enfant with a base point (since, recall from Proposition \ref{cubiertapunto1} and Corollary \ref{cubiertapunto2}, dessins d'enfants with a base point are in bijective correspondence with subgroups of the fundamental group). Therefore, we can make each subgroup correspond with an extension of $\C(t)$, although some of them may be isomorphic. The set $\mathcal H$ is also ordered, by inclusion. We actually know, by Proposition \ref{uniquemorphism}, that for every two subgroups such that $H_1\subset H_2$, there exists a unique morphism between the corresponding dessins preserving the base points. We can then define the following partially ordered set $\mathcal H^*$: it has one extension of $\C(t)$ for every subgroup $H$ of $F_2$, which we will call $\KK^H/\C(t)$, and, whenever $H_1< H_2$, it has the corresponding morphism between extensions $i_{H_2H_1}:\KK^{H_2}\lra \KK^{H_1}$ which, as we have just said, is unique.

The set $\mathcal H^*$ with these morphisms is actually a small category, since the compositions of the homomorphisms corresponding to inclusions $H_1\subset H_2$ and $H_2\subset H_3$ is the homomorphism corresponding to the inclusion $H_1\subset H_3$. It is also a direct system: for every two subgroups $H_1$ and $H_2$, we can take the subgroup $H_1\cap H_2$, and the morphisms $i_{H_1H_1\cap H_2}:\KK^{H_1}\lra \KK^{H_1\cap H_2}$.

Using these fields and homomorphisms, their union is well-defined. If one is familiar with the notion of a direct limit, then the union is just the direct limit of this direct system of fields, and direct limits of direct systems exists in the category of fields. We will nonetheless prove this explicitly.
\begin{prop}\label{KK}
There exists a field extension $\KK/\C(t)$ such that every extension corresponding to a dessin d'enfant is contained in it, and it is minimal in the sense that any other extension satisfying this property contains an extension isomorphic to $\KK/\C(t)$ as a subextension. 
\end{prop}
\begin{proof}
Consider the direct system $\mathcal H^*$. We will prove that there is a field $\KK$, with a morphism of extensions for each $H$, $i_H:\KK^H\lra \KK$, such that, for any two extensions $\KK^{H_1}$, $\KK^{H_2}$ such that one is contained in the other, the following diagram is commutative:
\begin{center}\hspace{0.5 cm}
\begindc{\commdiag}[40]
\obj(1,1)[a]{$\KK$}
\obj(1,0)[b]{$\KK^{H_1}$}
\obj(0,0)[c]{$\KK^{H_2}$}
\mor{b}{a}{$i_{H_1}$}[\atright,0]
\mor{c}{b}{$i_{H_2H_1}$}[\atright,0]
\mor{c}{a}{$i_{H_2}$}
\enddc
\end{center}
\noindent Also, we require $\KK$ to be the minimal field with this property (this is the definition of direct limit). The way to construct it is the following: consider
$$
\KK=\frac{\bigsqcup_{H<F_2}\KK^H}{\sim}
$$
Where $\bigsqcup$ stands for the disjoint union. In this union, we are going to write $x_H$ to express that $x_H\in H$. $\sim$ is the relation given by:
$$
x_{H_1}\sim y_{H_2} \Llra i_{H_1H_1\cap H_2}(x_{H_1})=i_{H_2H_1\cap H_2}(y_{H_2})
$$
It is easy to check that this is an equivalence relation, using the fact that the fields form a directed set\footnote{Constructing the direct limit when the objects are not fields requires modifications of this construction, but the one we are using works for this particular case with fields, and in which every morphism is injective.}. We will use square brackets to denote the equivalence classes for this relation. This quotient is a field, with the following operations:
$$
[x_{H_1}]+[x_{H_2}]=[i_{H_1H_1\cap H_2}(x_{H_1})+_{\KK^{H_1\cap H_2}}i_{H_2H_1\cap H_2}(x_{H_2})]
$$ $$
[x_{H_1}]\cdot[x_{H_2}]=[i_{H_1H_1\cap H_2}(x_{H_1})\cdot_{\KK^{H_1\cap H_2}}i_{H_2H_1\cap H_2}(x_{H_2})]
$$
Where $+_{\KK^H}$ and $\cdot_{\KK^H}$ denote the sum and product in the given field. It is easy to check that $\KK$ is a field, it is just a lot of properties to check (starting with the fact that the operations are well defined).

From any field $\KK^H$ there is a homomorphism to $\KK$, given by
$$
x\longmapsto i_H(x)=[x_H]
$$
And it is actually a homomorphism of $\C(t)$-algebras. The fact that these homomorphisms commute with $i_{H_2H_1}$ is as easy to check, since
$$
i_{H_1}(i_{H_2H_1}(x))=[i_{H_2H_1}(x)_{H_1}]=[x_{H_2}]=i_{H_2}(x)
$$
Where the middle equality comes just from the definition of the equivalence relation.

Suppose now that some other field $\KK'$ satisfied these properties with respect to the system $\mathcal H^*$: let $i'_H:\KK^H\lra \KK'$ be homomorphisms of extensions such that, for every $H_1\subset H_2$, $i'_{H_2}=i'_{H_1}\circ i_{H_2H_2}$. Define $i:\KK\lra \KK'$ to be $i([x_H])=i'_H(x)$. This is well defined thanks to the fact that the $i'_H$ commute with the inclusions between fields, and also $i \circ i_H=i'_H$, which is the property we wanted. Thus, $\KK$ is the minimal field extension satisfying this property.

Now, if a field extension $\KK'/\C(t)$ contains every dessin d'enfant as a subextension, let us see that it must contain the system $\mathcal H^*$. First of all, such a field must contain every extension corresponding to a regular dessin. Since these are Galois extensions, there is only one copy of each, and let us call them $\KK^{H'}$. We are going to prove that we can pick a base point (that is, a valuation), in every extension, in such a way that, whenever $\KK^{H_1'}\subset \KK^{H_2'}$ (which will happen whenever $H_2\subset H_1$, since these are Galois extensions, and there is only one isomorphic copy in each field), the restriction of the valuation in $\KK^{H_2'}$ to $\KK^{H_1'}$ is the base point of $\KK^{H_1'}$. This is clear using induction: if we have defined the valuation in a finite extension $F\subset \KK'$, we can extend it to another finite extension $F'\supset F$, and thus to every finite subextension of $F'$. Since this process can always continue, we can take the valuation defined in the union of all these fields $\KK^{H'}$, since they are a countable collection. Now, for every Galois extension, we can define $i'_H:\KK^H\lra \KK'$ so that its image is $\KK^{H'}$ and it maps the base point to the base point (since there is a unique homomorphism that does this). It is clear, using that there is only one homomorphism that preserves base points, that we will have $i'_{H_2}=i'_{H_1}\circ i_{H_2H_2}$ as before, and that we can extend these maps $i_H$ to extensions that are not Galois, since they are contained in Galois extensions, and we can obtain the maps $i_H$ by restriction.
\end{proof}

\noindent We have constructed a field that contains all dessins d'enfants, since these correspond to all the subgroups $H$. Actually, every finite subextension of this field is a dessin d'enfant.

\begin{prop}
Let $\KK/K/\C(t)$ be a subextension of $\KK/\C(t)$ such that $[K:\C(t)]$ is finite. Then, there is some $H<F_2$ of finite order such that $K=\KK^H$.
\end{prop}
\begin{proof}
Every element of $\KK$ belongs to some extension $\KK^H$, and in particular it is algebraic over $\C(t)$. Then, $K/\C(t)$ is an algebraic finite extension, so it is generated by one element $f$. This $f$ lies in some $\KK^H$, by the definition of $\KK$. There exists one maximal $H$ such that $f$ lies in $\KK^H$, since $\KK^{H}\cap \KK^{H'}=\KK^{HH'}$ (where $HH'$ stands for the subgroup generated by $H$ and $H'$ in case they are not normal). If we pick this $H$, it must follow that $\C(t)(f)=\KK^H$, since $\C(t)(f)$ is a subextension of $\KK^H$, but subextensions of the extension corresponding to a dessin correspond to subcovers. Therefore, $\C(t)(f)/\C(t)$ corresponds to a dessin d'enfant, and therefore it equals $\KK^H$, since this was the minimal one. 
\end{proof}

\noindent We are interested in the Galois group $\Gal(\KK/\C(t))$. In order to talk about this group, we are going to stop for a bit to talk about Galois theory of infinite extensions. The reader familiar with the Galois correspondence in infinite extensions may skip the next section.

\subsection{Galois theory of infinite extensions}

Suppose we have an algebraic Galois extension $E/K$, that is, one such that the Galois group $\Gal(E/K)$ fixes only $K$. We intend to look at the relation between $\Gal(E/K)$ and $\Gal(F/K)$ when $F$ ranges across the finite Galois subextensions of $\Gal(F/K)$.

First of all, the fields $K$ which are finite Galois extensions of $K$ form a direct system with the morphisms given by inclusion, and therefore, by the Galois correspondence, the Galois groups $\Gal(F/K)$ form an inverse system, by which we mean that, whenever $F\subset F'$, restriction gives an epimorphism
$$
\pi_{F'F}:\Gal(F'/K)\lra \Gal(F/K)
$$
And also, the composition of two restriction maps is another restriction. This system has an inverse limit. Let us define it just in case: if $G$ is the limit of this system, it means that for every finite Galois subextension $F$ there exists a homomorphism $\pi_F:G\lra \Gal(F/K)$, such that whenever $F\subset F'$, $\pi_{F}=\pi_{F'F}\circ \pi_{F'}$, and the group $G$ is universal with respect to this property: if for another group $H$, there exist homomorphisms $\rho_F:H\lra \Gal(F/K)$ such that $\rho_{F}=\pi_{F'F}\circ \rho_{F'}$, then there is a unique homomorphism $\rho:H\lra G$ such that $\rho_F=\pi_F\circ \rho$ for every $F$.

We can write this in a diagram to make it easier to remember:
\begin{center}\hspace{0.5 cm}
\begindc{\commdiag}[4]
\obj(0,25)[f]{$H$}
\obj(0,10)[b]{$\wh G$}
\obj(-10,0)[d]{$\Gal(F'/K)$}
\obj(10,0)[e]{$\Gal(F/K)$}
\mor{f}{b}{$\rho$}[\atleft,1]
\mor{f}{d}{$\rho_F$}[\atright,0]
\mor{f}{e}{$\rho_{F'}$}
\mor{b}{d}{$\pi_F$}
\mor{b}{e}{$\pi_{F'}$}[\atright,0]
\mor{d}{e}{$\pi_{F'F}$}[\atright,0]
\enddc
\end{center}

\noindent Note the analogy with our construction of the direct limit of fields (to obtain an equivalent diagram for direct limits of fields, one just needs to reverse all the arrows).

If we look at the Galois group $\Gal(E/K)$, it certainly has homomorphisms $\Gal(E/K)\lra \Gal(F/K)$, since these are induced by restriction. This suggests that it might be isomorphic to this inverse limit.

Let us prove first that inverse limits of groups exist. The proof will yield an explicit way to see an inverse limit.

\begin{prop}
Let $\left(\{G_i:i\in I\},\{\phi_{ij}:i,j\in I,i\le j\}\right)$ be an inverse system of groups, where $I$ is a directed set, $G_i$ are the groups and $\phi_{ij}$ denote the homomorphisms between them. Then, its inverse limit exists and it can be constructed as the subgroup of $\prod_i G_i$ made of the elements $(a_i)_I$ such that
$$
\pi_{ii'}(a_i)=a_{i'}\forall i\le i'$$
\end{prop}
\begin{proof}
The map from the direct limit to the groups is the projection on the coordinates, and the proof is analogous to the proof of Proposition \ref{KK}, so we will leave it to the reader.
\end{proof}

\noindent We will denote inverse limits by $\lim_{\leftarrow}$.

\begin{prop}
Let $E/K$ be an algebraic Galois extension. Then,
$$
\Gal(E/K)=\lim_{\underset{F/K\mathrm{\ finite\ and\ Galois}}\longleftarrow}\Gal(F/K)
$$
\end{prop}
\begin{proof}
Take the map induced by the restrictions, from $\Gal(E/K)$ to $\prod \Gal(F/K)$. The image of this map clearly lies inside the inverse limit, since the restriction to finite subextensions is compatible with the restrictions from one subextension to another. Also, it is injective: if $\sigma\in \Gal(E/K)$ is mapped to the identity, then it acts as the identity on every finite subextension. Therefore, since any element belongs to some finite Galois extension, it is the identity on $E$. For surjectivity, one only needs to check that an automorphism for each finite subextension that commutes with restrictions obviously defines an automorphism of $E/K$.
\end{proof}

\noindent The groups that arise this way are called profinite.

\begin{definicion}
A \textbf{profinite group} is a group that is an inverse limit of finite groups.
\end{definicion} 

\noindent Clearly, Galois groups of algebraic extensions fall in this category. Profinite groups can be given additional structure, namely a topology.

\begin{definicion}
Let $E/K$ be an algebraic Galois extension. The \textbf{Krull topology} on $\Gal(E/K)$ is the topology that has a basis consisting of the cosets of the groups $\Gal(E/F)$, where $F$ ranges across all finite Galois extensions of $F$.

Equivalently, it is the weakest topology that makes all the restriction homomorphisms $\Gal(E/K)\lra \Gal(F/K)$ continuous, when the finite groups $\Gal(F/K)$ are given the discrete topology.
\end{definicion}

\noindent It is clear that both definitions coincide: the preimages of points by a restriction homomorphism are open sets by definition, and also the cosets of any group of the form $\Gal(E/F)$ are preimages of subsets of $\Gal(E/\overline F)$, where $\overline F$ is the Galois closure of $F/K$.

It is also clear that the definition makes the Galois group a topological group, since multiplication by a group element maps the basis to itself.

Also, if we have a map $\pi_F:\Gal(E/K)\lra \Gal(F/K)$ for each Galois extension $F$, then every map is continuous if and only if the corresponding map$$\prod \pi_F:\Gal(E/K)\lra \prod_F\Gal(F/K)$$is continuous when the product on the right has the product topology. Since the latter map is inclusion, the Krull topology is the topology induced by the product topology\footnote{Recall that the product topology on $\prod X_i$ is the topology generated by sets of the form $\prod U_i$, where all the $U_i$'s are open sets of $X_i$ and all but finitely many of them are equal to $X_i$.}.

We can give nice characterizations of open and closed subgroups, which we will do now.

\begin{prop}
Let $G=\lim_{\leftarrow} G_i$ be a profinite group. A subgroup $H$ of $G$ is open if and only it is closed and of finite index.
\end{prop}
\begin{proof}
Let us see $G$ as a subspace of the product $\prod_i G_i$. If $H$ is open, it must contain some neighborhood of the identity, which is of the form $\prod_i U_i$, where $U_i=G_i$ for all but finitely many of the $i$'s. Therefore the index of $H$ is at most the product of the orders of the groups $G_i$ for which $U_i\subsetneq G_i$, which is finite. Also, every coset of $H$ is an open set, since $G$ is a topological group, and the union of the ones different from $H$ is open, and it is the complement of $H$, so $H$ is closed.

Reciprocally, if $H$ is closed and of finite index, its cosets are closed and the union of the ones different from $H$ is the complement of $H$ and a closed set.
\end{proof}

\begin{prop}\label{open}
Let $G$ be a profinite group as in the previous proposition. A subgroup $H$ of $G$ is closed if and only if it is an intersection of open subgroups.
\end{prop}
\begin{proof}
The ``if'' part is clear: open subgroups are also closed, so their intersection is also closed.

Suppose now we have a closed subgroup $H$. We have the obvious inclusion
$$
H\subset \bigcap_{U=\mathring U,H<U} U
$$
So what we have to prove is the reverse inclusion: that for every $g\notin H$, there is some open subgroup $U$ containing $H$ such that $g\notin U$. Let $g\notin H$. Since $H$ is closed, there is a neighborhood of $g$ that doesn't intersect $H$. As usual, it will be of the form $V=\prod V_i$, where almost all the $V_i$'s equal $G_i$. Let $J$ be the set of indices for which $V_i\neq G_i$. We can assume that $V_i=\{g_i\}$ when $i\notin J$, where $g_i$ is the projection of $g$ onto $G_i$. We can do this since this set will also be open and contained in the previous one.

Then, we can consider the open subgroup
$$U=G\cap \prod_{i\notin J} G_i \times \prod_{i\in J} \{1\}$$
It is the preimage of 1 in $\prod_{i\in J} G_i$, so if we take the minimum of $J$, and call it $i_0$ (we are using that the set of indices is a directed set), this group is the same as
$$
U=G\cap \prod_{i\neq i_0} G_i \times \{1_{G_{i_0}}\}
$$
Look at the subgroup $UH$. It is open, since it is a union of cosets of the open group $U$, and it is indeed a subgroup, since $U$ is normal. We claim that it doesn't contain $g$. If $g\in UH$, then $V\cap H$ would be non-empty: There would have to be some element $h\in H$ whose component in $G_{i_0}$ were equal to $g_{i_0}$, and this element would belong to $V$.
\end{proof}

\noindent The Galois correspondence holds then the same as for finite Galois extensions, only one needs to replace subgroups of a finite Galois group by closed subgroups.

\begin{prop}
Open subgroups of $\Gal(E/K)$ are in correspondence with finite subextensions $F/K$, via the mutually inverse maps
$$
H<\Gal(E/K) \longmapsto E^H=\{a\in E:\sigma(a)=a\forall \sigma\in H\}
$$
$$
F/K\longmapsto \Gal(E/F)<\Gal(E/K)
$$
\end{prop}
\begin{proof}
Suppose we have a finite Galois subextension $F/K$. An element of $\Gal(E/K)$ belongs to $\Gal(E/F)$ if and only if it maps to the identity in $\Gal(F/K)$, i.e. if it belongs to $\prod_{F'\neq F} \Gal(F'/K)\times \{1_{\Gal(F/K)}\}$. In particular, $\Gal(F/K)= \frac{\Gal(E/K)}{\Gal(E/F)}$, and $[F:K]=[\Gal(E/K):\Gal(E/F)]$.

Now, if $F/K$ is not a Galois extension, we can take its Galois closure $\overline F$ and do the same thing again. Then, $\Gal(E/F)$ will be the preimage of $\Gal(\overline F/F)<\Gal(\overline F/K)$ by the projection. In particular, the identity involving the index holds, by the Galois correspondence for finite extensions.

Suppose now we have an open subgroup $H$. Then, it contains some neighborhood of the identity in $\prod \Gal(F/K)$, of the form $\prod U_K$. Let $\overline K$ be the Galois closure of the fields for which $U_K\neq \Gal(F/K)$. It is a finite extension, and this neighborhood contains$$\Gal(E/K) \cap \prod_{F'\neq \overline F}\Gal(F'/K) \times \{1_{\Gal(\overline F/K)}\}=\Gal(E/\overline F)$$Therefore, the fixed field of $H$ is contained in $\overline F$, and $H$ maps onto a subgroup of $\Gal(\overline F/K)$. $E^H$ is the fixed field of this subgroup and its degree over $K$ is equal to $[\Gal(E/K):H]$, by the Galois correspondence for finite extensions.

The following are clear:
$$
H<\Gal(E/E^H);
F\subset F^{\Gal(E/F)}
$$
But they must be equal, since both groups in the first case have the same index in $\Gal(E/K)$, and both extensions have the same degree in the second case. 
\end{proof}

\begin{prop}
Closed subgroups of $\Gal(E/K)$ are in correspondence with subextensions $F/K$, via the mutually inverse maps
$$
H<\Gal(E/K) \longmapsto E^H=\{a\in E:\sigma(a)=a\forall \sigma\in H\}
$$
$$
F/K\longmapsto \Gal(E/F)<\Gal(E/K)
$$
\end{prop}
\begin{proof}
Note that closed subgroups are the intersection of the open subgroups containing them, by proposition \ref{open}, and that subextensions are the union of the finite subextensions contained in them. For a set of extensions $\{F_i/K\}$, let us denote the field they generate by $\left\langle\bigcup_i F_i\right\rangle$. It is clear that, for a set of finite extensions $\{F_i\}$,
$$
\bigcap_i \Gal\left(E/F_i\right)=\Gal\left(E/\left\langle \bigcup_i F_i\right\rangle\right)
$$
If some $\sigma\in \Gal(E/K)$ fixes every field $F_i$, it will fix the field they generate. Reciprocally, if it fixes the field they generate, it will fix each one.

Also, if we have some open subgroups $\{H_i\}$,
$$
\Gal\left(E/\left\langle \bigcup_i E^{H_i}\right\rangle\right)=\bigcap_i \Gal\left(E/E^{H_i}\right)=\bigcap_i H_i
$$
Let us prove that $\Gal(E/E^{\bigcap H_i})=\bigcap H_i$. The inclusion $\Gal(E/E^{\bigcap H_i})\supset\bigcap H_i$ is obvious. Now, take some element $\sigma \in \Gal(E/K)\setminus \bigcap H_i$. There must exist some $i_0$ such that $\sigma \notin H_{i_0}$, so $\sigma$ won't fix $E^{H_{i_0}}\subset E^{\bigcap H_i}$. Therefore, $\sigma \notin \Gal(E/E^{\bigcap H_i})$. Using this, we have that
$$
\left\langle \bigcup_i E^{H_i}\right\rangle=E^{\Gal(E/\langle \bigcup_i E^{H_i}\rangle)}=E^{\bigcap_i H_i}
$$
Where the first equality comes from the fact that $E/K$ is Galois. Using these equalities, we can prove that the maps in the Galois correspondence are mutually inverse. Let $H$ be a closed subgroup of $\Gal(E/K)$. It is then equal to the intersection of the open subgroups containing them, $H_i$. It follows that
$$
\Gal(E/E^H)=\Gal(E/E^{\bigcap_i H_i})=\bigcap_i H_i=H
$$
Also, if we have some extension $F/K$, it is the union of the finite extensions $F_i$ it contains. Therefore,
$$
E^{\Gal(E/\langle\bigcup_i  F_i\rangle)}=E^{\bigcap_i \Gal(E/F_i)}=\left\langle 
\bigcup_i E^{\Gal(E/F_i)}
\right\rangle=\left\langle\bigcup_i  F_i\right\rangle=F
$$
\end{proof}

\noindent Thus we have the Galois correspondence for any algebraic extension. We are just going to add one remark, which is that, the same as for finite extensions, if $F/K$ is a subextension and $\sigma\in \Gal(E/K)$, then $\Gal(E/\sigma(F))=\sigma\Gal(E/F)\sigma^{-1}$.

This theorem on Galois theory is explained in Brian Osserman's notes \cite{ucdavis}, and for more information on profinite groups, one can look in \cite{profinite}.

\subsection{Back to $\KK$}

We are going to use the results in the previous section to prove that the Galois group of $\KK/\C(t)$ is the inverse limit of Galois groups of the finite Galois subextensions, which correspond to the finite quotients $F_2/N$. This inverse limit is called the profinite completion.

\begin{definicion}
Let $G$ be a group. The set of its finite quotients $G/N$ with the projections amongst them form an inverse system, and the inverse limit of this system is a profinite group, which we denote $\wh G$, and it is the \textbf{profinite completion} of $G$.
\end{definicion}

\noindent Since the inverse limit can be seen as embedded in a product, we can see its elements as sequences, whose elements belong to the quotients of $G$.

As an example, what is the profinite completion of $\Z$? It is made up of sequences of the form
$$
(a_2,a_3,a_4,\ldots )
$$
Where $a_n\in \zn{n}$, and, whenever $m|n$, $a_n\equiv a_m (\mod m)$. So it has elements like
$$
(1_2,1_3,1_4,\ldots)
$$
But also many other elements, like
$$
(0_2,2_3,0_4,3_5,2_6,4_7,0_8,5_9,\ldots)
$$
That may never become constant. The profinite completion of a group is, in general, a much larger group (for example, $\wh \Z$ is uncountable\footnote{There are no countably infinite profinite groups!}). The group $\wh \Z$ arises for example as the Galois group of $\overline{\mathbb F_p}/\mathbb F_p$.

Back to $\KK$: we are going to see that $\Gal(\KK/\C(t))=\Fd$. We know that this Galois group is the inverse limit of the Galois groups of finite Galois subextensions. These Galois groups are the finite quotients of the form $F_2/N$, and also, let's see that the restriction maps are equal to the projections in $F_2$.

Suppose we have some finite index $N\tl F_2$, In Proposition \ref{regular} we said that $\Gal(\KK^N,\C(t))$ is (canonically equivalent to) the automorphism group of the dessin corresponding to $H$, and that this in turn is the group $N_{F_2}(N)/N=F_2/N$. Therefore, the maps $\pi_N:\Gal(\KK,\C(t))	\lra \Gal(\KK^N,\C(t))$ map onto $F_2/N$.

Let us now prove that, if $N_2<N_1$ and both are (finite index) normal subgroups of $F_2$,
$$\overline \pi_{N_1}=\pi_{N_2N_1}\circ \o\pi_{N_2}$$
Where $\pi_{N_1N_2}$ is the canonical projection from $F_2/N_2$ onto $F_2/N_1$, and $\o \pi_{N_i}$ is the projection from $\Gal(\KK,\C(t))$ onto $\Gal(\KK^{N_i},\C(t))$. Let's see this. Let $(C_i,f_i)$ be the dessin corresponding to $N_i$, and let $p:C_2\lra C_1$ be the cover between them, that is, the map that maps the base point to the base point and such that $f_2=f_1\circ p$. We are going to prove that the projection $F_2/N_2\lra F_2/N_1$ corresponds to restricting maps. Indeed, let $g\in F_2$, and call $g_i$ the unique automorphism of $(C_i,f_i)$ that maps the base point $x_i$ to $x_i^g$. Then,
$$
p\circ g_2=g_1\circ p
$$
Both are covering maps from $C_2$ to $C_1$, and to prove that they are equal, we just have to see where the base point maps to, by the unique lifting property. Since the monodromy action commutes with the covering maps,
$$
p(g_2(x_2))=p(x_2^g)=p(x_2)^g=g_1(p(x_2))
$$
If in the equality $p\circ g_2=g_1\circ p$ we take the corresponding maps in function fields, what we get is
$$
g_2^*\circ i=i\circ g_1^*
$$
In other words, $g_1^*$ is just the restriction of $g_2^*$ to the smaller extension $\KK^{N_1}/\C(t)$. Now, if we have some $\sigma\in \Gal(\KK/\C(t))$, we can find some $g_2$ such that $g_2^*=\pi_{N_2}(\sigma)$. Then, $g_1^*=\pi_{H_1H_2}(g_2^*)$, and the equality is just saying that
$$
\pi_{N_2N_1}(\o \pi_{N_2}(\sigma))=\pi_{N_1N_2}(g_2^*)=g_1^*=g_2^*\circ i=\sigma|_{\KK^{N_1}}=\o\pi_{N_1}(\sigma)
$$
So we have proven,what we wanted, that the maps $\pi_N:\Gal(\KK/\C(t))\lra  F_2/N$ given by restriction commute with the projections between these quotients. Therefore, we have proven that $\Gal(\KK/\C(t))=\Fd$.

We can now use the Galois correspondence in $\KK/\C(t)$ to map extensions, or dessins d'enfants, to open subgroups of $\Fd$. Suppose we have some open subgroup $U\tl \Fd$. The fixed field of $U$ is $\KK^H$, for some $H<F_2$. If we take some finite index normal subgroup $N$ contained in $H$, then $U=\Gal(\KK/\KK^H)$ equals the preimage by the restriction of $\Gal(\KK^N/\KK^H)=H/N$. Also, this preimage does not depend on the choice of $N$. Therefore, we can use $\wh H$ to name open subgroups of $\Fd$, since every open subgroup arises in this way (as the preimage of $H/N$, where $H<F_2$ and $N$ is any finite index normal subgroup of $F_2$ contained in $H$).
\begin{corol}
Dessins d'enfants with a base point are in correspondence with open subgroups of $\Fd$.

An open subgroup $\wh H$ has fixed field $\KK^H$ (and hence the notation). 
\end{corol}

\noindent $F_2$ is embedded in $\Fd$, by means of the projections from $F_2$ to its quotients. Therefore, we can consider the image of finite index subgroups $H<F_2$. We are going to see that their closure in the Krull topology is $\wh H$. Take the map $i:F_2\lra \Fd$, and consider $i(H)$. Since closed subgroups correspond to subextensions, $\o{i(H)}=\Gal(\KK/\KK^{i(H)})$. Now, the fixed field of $i(H)$ is precisely $\KK^H$, so its closure is $\wh H$.

From now on, we will omit the hat from the groups $\wh H$ and call them $H$, since they are both strongly related, and it will be clear from the context whether we are talking about a subgroup of $F_2$ or an open subgroup of $\Fd$. Most of the time from now on, we will refer to subgroups of $\Fd$.

Let us consider dessins d'enfants without a base point: they are in bijection with conjugacy classes of subgroups of $F_2$. Since two subgroups $H_1$ and $H_2$ are conjugate in $F_2$ if and only if they are conjugate in $F_2/\cor_{F_2}(H_1\cap H_2)$, which is a finite quotient, they will be conjugate if and only if they are also conjugate in $\Fd$. Also, it is immediate to check, as it is for finite extensions, that for some $\sigma\in \Gal$ and some field $F\subset \K$, $\Gal(\KK/F^\sigma)=\Gal(\KK/F)^\sigma$. Therefore, we have the following.

\begin{corol}
Dessins d'enfants without a base point are in correspondence with open subgroups of $\Fd$, up to conjugation.
\end{corol}

\chapter{The Galois action on dessins d'enfants and Belyi's theorem}
\section{The Galois action}\label{galaction}

We are going to see how the Galois group of a field can act on a curve. We are going to use it first for the Galois group of the complex numbers $\Gal(\C/\Q)$, but in the next section we will prove Belyi's theorem, that states that every dessin d'enfant is defined over $\QQ$, so we will look at $\gal$. Throughout this section, $K$ will be any algebraically closed field of characteristic 0.

Recall that we said in section \ref{secregular} that Galois groups would always act on the right. For coherence, we are going to keep this convention for $\Gal(K/\Q)$, so the action of some $\sigma \in \Gal(K/\Q)$ on an element $a\in K$ is given by $a^\sigma$, and it follows that $(a^\sigma)^\tau=a^{\sigma\tau}$.

Suppose we have a curve $C$ defined over $K$, with field of rational functions $K(C)$, and some $\sigma \in \Gal(K/\Q)$. We are first going to define the action on the field of functions. We define $K(C)^\sigma$ to be the $K$-algebra given by $K\otimes_{\sigma^{-1}} K(C)$.

This tensor product, as usual, is generated by all the elements of the form $\{k\otimes_{\sigma^{-1}} f:k\in K,f\in K(C)\}$, but we have the relation
$$
k\otimes_{\sigma^{-1}} f=1\otimes_{\sigma^{-1}} k^{\sigma^{-1}}f
$$
Since we also have that $1\otimes_{\sigma^{-1}}f_1+1\otimes_{\sigma^{-1}}f_2=1\otimes_{\sigma^{-1}}(f_1+f_2)$, all the elements in the tensor product can be written in the form $1\otimes_{\sigma^{-1}} f$.

In this tensor product, the product by elements of $K$ is given by
$$
k(1\otimes_{\sigma^{-1}} f)=k\otimes_{\sigma^{-1}}f=1\otimes_{\sigma^{-1}} k^{\sigma^{-1}}f
$$
In other words, if $i$ stands for the embedding of $K$ in $K(C)$, the embedding of $K$ in $K(C)^\sigma$ is $i\circ \sigma^{-1}$. Note that $\sigma$ induces a field isomorphism, from $K(C)$ to $K(C)^\sigma$, which we will also call $\sigma$, given by $f^\sigma=1\otimes_{\sigma^{-1}} f$. However, this field isomorphism might not be a $K$-algebra isomorphism: if $k\in K$ and $f\in K(C)$, then $(kf)^\sigma=1\otimes_{\sigma^{-1}} kf=k^\sigma\otimes_{\sigma^{-1}} f=k^\sigma f^\sigma$.

Suppose $K(C)=K[x](y_1,\ldots, y_n)/(f_1,\ldots ,f_m)$. We are going to prove that$$K(C)^\sigma\cong K[x](y_1,\ldots, y_n)/(f_1^\sigma,\ldots ,f_m^\sigma)$$Here, if $f=\sum a_nx^n$, then we define $f^\sigma=\sum a_n^\sigma x^n$. Also, the sign $\cong$ stands for an isomorphism of $K$-algebras. We will prove it in the case where $K(C)=K[x](y)/(f)$, since the proof is the same, but the notation is clearer. If we call $x^\sigma=1\otimes_{\sigma^{-1}} x$, and $y^\sigma=1\otimes_{\sigma^{-1}} y$,
$$
f^\sigma(x^\sigma,y^\sigma)=\sum_{m,n} a_{mn}^\sigma (1\otimes_{\sigma^{-1}} x)^m(1\otimes_{\sigma^{-1}} y)^n=\sum_{m,n} a_{mn}^\sigma (1\otimes_{\sigma^{-1}} x^ny^m)=\sum_{m,n} 1\otimes_{\sigma^{-1}} a_{mn}x^my^n=0
$$
Therefore, $K(C)^\sigma\cong K[x](y)/(f^\sigma)$. In particular, $K(C)^\sigma$ has transcendence degree 1 over $K$.

This gives a right action of $\Gal(K/\Q)$ on $K$-algebras, by which we mean that for two Galois elements $\sigma,\tau\in \Gal(K/\Q)$, there is a natural isomorphism of algebras between $(K(C)^\sigma)^\tau$ and $K(C)^{\sigma\tau}$, given by
$$
\f{\cong}
{(K(C)^\sigma)^\tau=K\otimes_{\tau^{-1}} K \otimes_{\sigma^{-1}} K(C)}
{K(C)^{\sigma\tau}=K\otimes_{\tau^{-1}\sigma^{-1}}K(C)}
{1\otimes_{\tau^{-1}}1 \otimes_{\sigma^{-1}} a}{1\otimes_{\tau^{-1}\sigma^{-1}}a}
$$
It is obviously a field isomorphism, and also, if $k\in K$,
$$
k(1\otimes_{\tau^{-1}}1 \otimes_{\sigma^{-1}} a)=k\otimes_{\tau^{-1}}1 \otimes_{\sigma^{-1}} a=1\otimes_{\tau^{-1}}k^{\tau^{-1}} \otimes_{\sigma^{-1}} a=$$ $$=1\otimes_{\tau^{-1}}1 \otimes_{\sigma^{-1}} k^{\tau^{-1}\sigma^{-1}}a\cong 1\otimes_{\tau^{-1}\sigma^{-1}}k^{\tau^{-1}\sigma^{-1}}a=k(1\otimes_{\tau^{-1}\sigma^{-1}}a)
$$
If we call $\sigma$ the field isomorphism from $K(C)$ to $K(C)^\sigma$, and we do the same for $\tau$, we get that $\tau \circ \sigma=\sigma\tau$, modulo this natural isomorphism. This means that for an element $f\in K(C)$,
$$
(f^\sigma)^\tau=1\otimes_{\tau^{-1}}1 \otimes_{\sigma^{-1}} f\mapsto 1\otimes_{\tau^{-1}\sigma^{-1}}f=f^{\sigma\tau}
$$
In other words, $(f^\sigma)^\tau=f^{\sigma\tau}$, by means of this natural isomorphism.
\begin{definicion}
Let $C$ be a curve defined over $K$. For $\sigma\in \Gal(K/\Q)$, we define $C^\sigma$ to be the curve that has $K(C)^\sigma$ as its field of rational functions. Or, equivalently, if $C=V(f_1,\ldots ,f_m)$, then $C^\sigma=V(f_1^\sigma,\ldots f_m^\sigma)$.
\end{definicion}

\noindent This definition also gives a right action of $\Gal(K/\Q)$ on the curves, since for two elements $\sigma,\tau\in \Gal(K/\Q)$, and a curve $C$,$$K((C^\sigma)^\tau)=K(C^\sigma)^\tau=(K(C)^\sigma)^\tau=K(C)^{\sigma\tau}=K(C^{\sigma\tau})$$If a curve is transformed into another one by the action of the Galois group, we say that they are Galois conjugate.

The action of the Galois group is not trivial, and a curve can be Galois conjugate to a non-isomorphic one. As an example, take the plane curve given by the equation $y^2=x(x-1)(x-\sqrt 2)$. It is Galois conjugate to the curve with equation $y^2=x(x-1)(x+\sqrt 2)$, and these are not isomorphic (since the theory of elliptic curves says that the only curves isomorphic to $y^2=x(x-1)(x-\lambda)$ of the form $y^2=x(x-1)(x-\lambda')$ are the ones with $\lambda'\in \left\{\lambda,\frac{1}{\lambda},1-\lambda,1-\frac{1}{\lambda},\frac{1}{1-\lambda},\frac{\lambda}{1-\lambda} \right\}$).

We want to define this action also on dessins d'enfants. In order to do this, we will first define the action on points and morphisms, which will allow us to see that the conjugate of a dessin is another dessin, and to find some invariants of the action.

We have a field isomorphism between $K(C)$ and $K(C)^\sigma$, so we can use it to map valuations from one field to the other, and then use the correspondence between valuations and points. Take a valuation $\nu$ of $K(C)$. This valuation is mapped to a valuation $\nu^\sigma$ of $K(C^\sigma)$, by$$\nu^\sigma(f^\sigma)=\nu(f)$$This is the Galois action on the points of the curves. It is straightforward to check that this defines a $K$-valuation, using that $\sigma$ is a field isomorphism. Note that, for two Galois automorphisms $\sigma $ and $\tau$,
$$
(\nu^\sigma)^\tau((f^\sigma)^\tau)=\nu^\sigma(f^\sigma)=\nu(f)=\nu^{\sigma\tau}(f^{\sigma\tau})=\nu^{\sigma\tau}((f^\sigma)^\tau)
$$
Therefore, $(\nu^\sigma)^\tau=\nu^{\sigma\tau}$. Note that this definition is also $\nu^\sigma=(\sigma^{-1})^*(\nu)$.

If we have a morphism between two curves, $\phi:C\lra C'$, we can define $ \phi^\sigma $ as well. We proceed like this: we take the morphism between the function fields $\phi^*:K(C)\lra K(C')$, and we define $\phi^{*\sigma}=\sigma \circ \phi^*\circ \sigma^{-1}$, that is,
$$\phi^{*\sigma}(f^\sigma)=(\phi^*(f))^\sigma$$
From the definition, it is clear both that $({\phi}^{*\sigma})^\tau=\phi^{*\sigma\tau}$, and that the map $\phi^{*\sigma}$ is a $K$-algebra homomorphism.

We can now define the morphism $\phi^\sigma$ on points of $C^\sigma$, as $(\phi^{*\sigma})^*$. Let us write this out explicitly: If we have a point in $C$ corresponding to a valuation $\nu$ of $K(C)$ and a function $f\in K(C')$,
$$
(\phi^\sigma(\nu^\sigma))(f^\sigma)=\nu^\sigma(\phi^{*\sigma}(f^\sigma))=\nu^\sigma((\phi^*(f))^\sigma)=\nu(\phi^*(f))=(\phi(\nu))(f)
$$
Since $(\phi(\nu))^\sigma(f^\sigma)=(\phi(\nu))(f)$, it follows that
$$
\phi^\sigma(\nu^\sigma)=(\phi(\nu))^\sigma
$$
And, of course, the same thing happens if we change points for valuations.

So, if a point $P$, which corresponds to some valuation $\nu$, is mapped by $\sigma$ to another point $P^\sigma$, with valuation $\nu^\sigma$, then
$$
\phi^\sigma(P^\sigma)=(\phi(P))^\sigma
$$
Of course, from the definition $\phi^\sigma=(\phi^{*\sigma})^*$ follows that $(\phi^\sigma)^\tau=\phi^{\sigma\tau}$.

\noindent Suppose we have $C=\P^1$. Its field of functions is $K(t)$, and if we take an automorphism $\sigma\in \Gal(K/\Q)$, the field $K(t)^\sigma$ is canonically isomorphic as a $K$-algebra to $K(t)$, by
$$
\f{\psi}{K(t)^\sigma}{K(t)}{1\otimes_{\sigma^{-1}}\frac{f(t)}{g(t)}}{\frac{f^\sigma(t)}{g^\sigma(t)}}
$$
It is straightforward to check that this is indeed a $K$-algebra isomorphism. We can then identify $K(t)^\sigma$ with $\psi(K(t)^\sigma)=K(t)$, so that $\psi\circ\sigma$ is just applying $\sigma$ to the coefficients of functions (note that this action is the same as the one we have defined for polynomials). Since we can identify $K(t)^\sigma$ with $K(t)$ via this isomorphism, we can also identify $\P^1$ with $(\P^1)^\sigma$. Now, let us see where the points are mapped. Suppose we have the point $P$ with coordinates $(1:p)$. This is the unique point where the function $x_1/x_0-p=t-p$ vanishes. By the identity
$$
1=\ord_P \left(t-p\right)=\ord_{P^\sigma}(\left(t-p\right)^\sigma)=\ord_{P^\sigma}\left(t-p^\sigma\right) 
$$
Where the last equality is the identification $\psi$ we have just defined. Since $t-p^\sigma$ has order $1$ at $P^\sigma$, this must mean that the coordinates of $P^\sigma$ are $(1:p^\sigma)$. Thus, $\sigma$ acts on $\P^1$ by acting on the coordinates.

Now, suppose we have a curve $C\subset \P^n$, with coordinates $(x_0:\cdots x_n)$, and a point $P=(p_0:\cdots :p_n)\in C$. Then,
$$
\left(\frac{x_i}{x_j}\right)^\sigma(P^\sigma)=\left(\left(\frac{x_i}{x_j}\right)(P)\right)^\sigma=\left(\frac{p_i}{p_j}\right)^\sigma=\frac{p_i^\sigma}{p_j^\sigma}
$$
Which means that $\sigma$ acts on $C$ by acting on the coordinates.

We could really just forget about the previous definitions and take the action on the coordinates as definition, but this way, we have clearly proven that the action doesn't depend on how a curve is embedded in $\P^n$.

Let us sum up all the definitions, the conclusions and the notation.

\begin{prop}[The Galois action] Suppose we have two curves $C$ and $C'$ defined over $K$, a point $P\in C$, a morphism $\phi:C\lra C'$ and some $\sigma\in \Gal(K/\Q)$.
\begin{itemize}
\item The field $K(C)^\sigma$ is $K\otimes_{\sigma^{-1}} K(C)$.
\item There is a field isomorphism $\sigma:K(C)\lra K(C)^\sigma$
\item To a morphism of algebras $\phi^*:K(C)\lra K(C')$ corresponds a morphism $\phi^{*\sigma}:K(C)^\sigma\lra K(C')^\sigma$, given by $\phi^{*\sigma}(f^\sigma)=(\phi^*(f))^\sigma$.
\item The curve $C^\sigma$ is such that $K(C^\sigma)=K(C)^\sigma$.
\item The points of $C^\sigma$ can be defined by $\ord_{P^\sigma}(f^\sigma)=\ord_P(f)$.
\item To a morphism of curves $\phi:C\lra C'$ corresponds a morphism $\phi^\sigma:C^\sigma\lra C'^\sigma$ such that $\phi^\sigma(P^\sigma)=\phi(P)^\sigma$.
\item $\phi^{\sigma*}=\phi^{*\sigma}$.
\end{itemize}
\end{prop}

\noindent We can now define the Galois action on dessins d'enfants. We see them as extensions $K(f)\subset K(C)$. The Galois action maps $K(C)$ to another algebra $K(C)^\sigma$, and the subfield $K(f)$ is mapped to some other subfield $K(f)^\sigma$, which is isomorphic to $K(t)$ (it is the subfield generated by $K^\sigma$ and $f^\sigma$). The dessin d'enfant conjugate  by $\sigma$ to $(C,f)$ is defined to be $(C^\sigma,f^\sigma)$, i.e. the extension $K(f)^\sigma\subset K(C)^\sigma$.

If the curve $C$ is embedded, then we know the Galois action on the curve is given by applying $\sigma$ to the coefficients of the equations. Also, since, for any function $f=\frac{\sum a_{ \alpha} \overline x^{ \alpha}}{\sum b_{ \beta} \overline x^{ \beta}}$ (where $\alpha$ and $\beta$ are multiindices), and a point with coordinates $\overline p$,
$$
\frac{\sum a_{ \alpha}^\sigma  (\overline p^\sigma )^{ \alpha}}{\sum b_{ \beta}^\sigma ( \overline p^\sigma )^{ \beta}}=\left(\frac{\sum a_{ \alpha} \overline p^{ \alpha}}{\sum b_{ \beta} \overline p^{ \beta}}\right)^\sigma
=f(P)^\sigma=
f^\sigma(P^\sigma)=
\frac{\sum c_{ \alpha} (\overline p^\sigma )^{ \alpha}}{\sum d_{ \beta} (\overline p^\sigma )^{ \beta}}
$$
We must have that $c_\alpha=a_\alpha^\sigma $ and $d_\beta=b_\beta^\sigma $, so applying $\sigma$ to a function is applying $\sigma$ to its coefficients.

Therefore, a definition for the Galois action on a Belyi pair can be just applying the automorphism to the coefficients of the equations of the curve and the Belyi function.

We must prove that the function $f^\sigma$ is unramified outside of $0$, $1$ and $\infty$. Take a point $P^\sigma\in C^\sigma$. If $f(P)=0$, then $f^\sigma(P^\sigma)=0^\sigma=0$, so the fiber of $0$ is mapped to the fiber of $0$. Also, $\ord_{P^\sigma} f^\sigma=\ord_P f$, so the ramification index is preserved. The same thing happens to the points over 1 and $\infty$. Also, if $f^\sigma(P^\sigma)\notin \{0,1,\infty\}$, then $f(P)\notin \{0,1,\infty\}$, so $f$ is unramified at $P$. Therefore,
$$\ord_{P^\sigma} (f^\sigma-f^\sigma(P^\sigma))=\ord_{P^\sigma} (f^\sigma-(f(P))^\sigma)=\ord_P(f-f(P))=1$$
So the function $f^\sigma$ is unramified over every point outside of $\{0,1,\infty\}$. Therefore, it is a Belyi function.

If $(C,f)$ is a Belyi pair, then $\ord_{P^\sigma}f^\sigma=\ord_P f$, so not only the ramification values are preserved, but also the set of ramification indices over each point. This is an example of a \textbf{Galois invariant}, a property of a dessin that is preserved under the Galois action. If we look at the bicolored graph, the index of ramification at a black or white point is the number of edges it has attached, and for a star vertex, it is half the number of edges on a face. Therefore, this means that a dessin d'enfant has the same number of points of each order and of faces of each number of sides as any of its Galois conjugates.

Since, by the Euler formula, the Euler characteristic is equal the number of vertices minus the number of edges plus the number of faces in any triangulation, this means that the genus of the underlying curve is another Galois invariant.

There is yet another invariant preserved by the Galois action, which is the automorphism group: the automorphism group of a dessin given by an extension $K(C)/K(t)$ is $\Gal(K(C)/K(t))$, and $\sigma\in \Gal(K/\Q)$ acts on the homomorphisms in this Galois group. It is clear that the following map is an isomorphism:
$$
\f{\cdot^\sigma}{\Gal(K(C)/K(t))}{\Gal(K(C)^\sigma/K(t))}{\phi}{g^\sigma=\sigma\circ g \circ \sigma^{-1}}
$$
Since its inverse is given by the corresponding map for $\sigma^{-1}$.

Also, a Galois automorphism $\sigma$ will map Galois extensions of $K(t)$ to Galois extensions (since they preserve the degree and the Galois group). It is clear then that it maps the Galois closure of an extension to the Galois closure of its image: in the language of Belyi covers, the Galois group preserves the regular cover of another cover. Now, since we know that the monodromy group of a dessin given by a finite index subgroup $H$ of $F_2$ (or an open subgroup of $\Fd$) is given by the quotient of $F_2$ by the kernel of the action, which is $\cor_{F_2} H$. Therefore, it is also the automorphism group of the regular dessin given by $\cor_{F_2} H$, its regular cover. Since regular covers are preserved, their automorphism groups are also preserved, and these are the cartographic groups, we conclude that the cartographic group of a dessin is isomorphic to the cartographic group of a conjugate dessin. However, the dessins need not be isomorphic! This is because the Galois group doesn't necessarily map the canonical generators to the canonical generators, and the choice of these generators is what determines the dessin.

Let us sum up the Galois invariants so far:

\begin{prop}
The Galois action on dessins preserves the following:
\begin{itemize}
\item The degree of the covers.
\item The ramification indices of points.
\item The genus of the underlying curves.
\item The automorphism group, up to isomorphism.
\item The monodromy group, up to isomorphism.
\end{itemize}

\end{prop}

\noindent This action is what makes dessins d'enfants interesting, since it can be used to study the group $\gal$ (since dessins d'enfants are defined over the algebraic field), for example by embedding it into other groups, as in proposition \ref{out}. This can be done because the action is faithful, i.e. for any $\sigma\in \gal$, there is a dessin $(C,f)$ that isn't fixed by $\sigma$. In fact, there is a dessin of a given genus \cite{genus} and also a regular dessin \cite{jule03} such that they are not fixed by the action. For now, we will just prove that the action is not trivial.
\begin{prop}
Take the curve $C$ with equation $y^2=x(x-1)(x-\sqrt 2)$, and the function $f(x,y)=x^2(2-x^2)$. This function is a Belyi function, and if $\sigma\in \gal$ is such that $\sigma(\sqrt 2)=-\sqrt  2$, then the Belyi pair $(C^\sigma,f^\sigma)$ is not equivalent to $(C,f)$.
\end{prop}
\begin{proof}
It is straightforward to see that $f$ is a Belyi function: its ramification points are $\{(0,0),(\sqrt 2,0),(1,0),\\\left(-1,\pm \sqrt{-2-2\sqrt 2} \right),(\sqrt 2,0),(0:1:0)\}$, and their images are contained in $\{0,1,\infty\}$.

The curve $C^\sigma$ is not isomorphic to $C$, since this is the example we have seen before, so the dessins cannot be isomorphic. In fact, if we draw them, we obtain:

\begin{center}
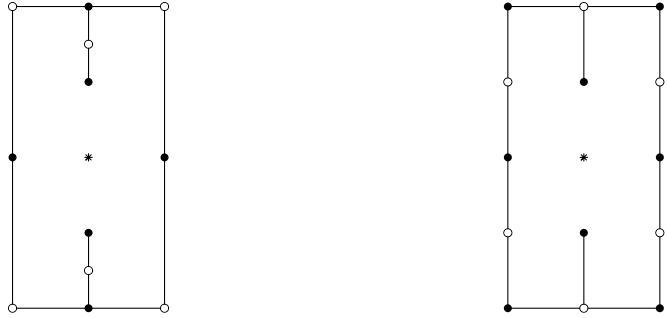

\begin{tikzpicture}[line cap=round,line join=round,>=triangle 45,x=2cm,y=2cm]
\clip(-1.7,-1.5) rectangle (1.5,1.5);
\draw (0.5,1)--(0.5,-1)--(-0.5,-1)--(-0.5,1)--(0.5,1);
\draw (0,-1)--(0,-0.5);
\draw (0,1)--(0,0.5);
\begin{scriptsize}
\draw (0,0)-- ++(-1.0pt,-1.0pt) -- ++(2.0pt,2.0pt) ++(-2.0pt,0) -- ++(2.0pt,-2.0pt) ++ (-1.0pt,-0.4pt) -- ++(0,2.8pt) ++(-1.4pt,-1.4pt)-- ++(2.8pt,0);
\puntoblanco{(-0.5,1)}
\puntoblanco{(-0.5,-1)}
\puntoblanco{(0.5,1)}
\puntoblanco{(0.5,-1)}
\puntonegro{(0.5,0)}
\puntonegro{(-0.5,0)}
\puntonegro{(0,1)}
\puntonegro{(0,-1)}
\puntonegro{(0,0.5)}
\puntoblanco{(0,0.75)}
\puntonegro{(0,-0.5)}
\puntoblanco{(0,-0.75)}
\end{scriptsize}
\end{tikzpicture}
\begin{tikzpicture}[line cap=round,line join=round,>=triangle 45,x=2cm,y=2cm]
\clip(-1.7,-1.5) rectangle (1.5,1.5);
\draw (0.5,1)--(0.5,-1)--(-0.5,-1)--(-0.5,1)--(0.5,1);
\draw (0,-1)--(0,-0.5);
\draw (0,1)--(0,0.5);
\begin{scriptsize}
\draw (0,0)-- ++(-1.0pt,-1.0pt) -- ++(2.0pt,2.0pt) ++(-2.0pt,0) -- ++(2.0pt,-2.0pt) ++ (-1.0pt,-0.4pt) -- ++(0,2.8pt) ++(-1.4pt,-1.4pt)-- ++(2.8pt,0);
\puntonegro{(-0.5,1)}
\puntonegro{(-0.5,-1)}
\puntonegro{(0.5,1)}
\puntonegro{(0.5,-1)}
\puntonegro{(0.5,0)}
\puntonegro{(-0.5,0)}
\puntoblanco{(0,1)}
\puntoblanco{(0,-1)}
\puntonegro{(0,0.5)}
\puntonegro{(0,-0.5)}
\puntoblanco{(0.5,0.5)}
\puntoblanco{(0.5,-0.5)}
\puntoblanco{(-0.5,0.5)}
\puntoblanco{(-0.5,-0.5)}
\end{scriptsize}
\end{tikzpicture}
\captionof{figure}{To the left, the dessin $(C,f)$, and to the right, the dessin $(C^\sigma,f^\sigma)$. Opposing edges of the rectangles are identified.}
\end{center}
\noindent (Recall that an elliptic curve is a torus!). These drawings are clearly not isomorphic. For details on how to obtain the drawings, the reader can look at the method used in section \ref{dibujitos} for a similar example.
\end{proof}

\section{Belyi's theorem}
This is the main theorem about dessins d'enfants.

\begin{teor}[Belyi's Theorem]
Let $C$ be a complex algebraic curve. The following are equivalent:

\begin{enumerate}
\item $C$ is defined over $\QQ$, i.e. $C$ can be given by equations with coefficients in $\QQ$.
\item $C$ has a Belyi map.
\end{enumerate}
\end{teor}

\noindent We will devote the rest of this section to proving Belyi's Theorem. We will follow closely the proof in \cite{GG}.

\subsection{Curves defined over the algebraic numbers have Belyi maps}\label{belyi1}
Let's prove the first implication. Suppose we have a curve defined over $\QQ$. The process we are going to follow is: take a function on the curve $f$, look at its ramification values, and then compose it with functions $g_i$ from $\P^1$ to $\P^1$ that will make the function $g_k\circ g_{k-1} \circ \cdots \circ f$ have less and less ramification values, until we reach our objective, which is only three of them.

First of all, we will prove that we can make the ramification values rational.

\begin{lema}
Let a curve $C\subset \P^n$ be given by equations $f_1,\ldots f_m$. Take coordinates $(x_0:\cdots :x_n)$, and the affine part of the curve given by $x_0=1$. Fix a point $P=(p_1,\ldots ,p_n)\in C$, and consider the matrix $Df=\left(\frac{\partial f_i}{x_j}|_P\right)_{ij}$. Since the curve is non-singular, the matrix has rank $n-1$.

A function of the form $a_1x_1+\cdots +a_nx_n$ will be ramified at $P$ if and only if $(a_1,\ldots ,a_n)$ lies in the span of the rows of $Df$.
\end{lema}
\begin{proof}
Taking a translation of affine space, we can assume that the point is the origin. Suppose $(a_1,\ldots ,a_n)=(\lambda_1,\ldots ,\lambda_m)Df$. Then, the polynomial $f=\lambda_1f_1+\cdots +\lambda_mf_m$ is of the form $a_1x_1+\cdots a_nx_n$ plus parts of degree greater than 1. Now, if $\ord_P (a_1x_1+\cdots a_nx_n)=1$, then $\ord_P(f)$ would be $1$, since the rest of the terms in $f$ have order at least 2. However, $f$ is identically $0$ on the curve, and therefore its order cannot be 1. So our function ramifies at $P$.

Now, take a function $a_1x_1+\cdots a_nx_n$ that doesn't lie in the span of the rows of $Df$. This function along with the rows of $Df$ span the whole space. Therefore, if this function were ramified, the curve wouldn't be non-singular, since a non-singular curve has an unramified function at every point.
\end{proof}

\noindent Therefore, if we have a curve that is defined over the algebraic numbers, we can take a coordinate function, and it will be ramified at the points where it lies in the span of $Df$. This can be expressed by saying that the ramified points are the points where some polynomials vanish, namely the minors of size $n$ in the matrix made up of $Df$ plus a row at the bottom for the coordinate function, plus the polynomials that define the function. Therefore, the coordinates of the ramified points will be algebraic, i.e. the coordinate function will be ramified only over the algebraic numbers and $\infty$, since it might be ramified also at the points that don't lie in the affine part.

\begin{lema}
Let $C$ be a curve defined over the algebraic numbers. Then, there is a map $f\in \C(C)$ that ramifies only over values in $\Q\cup \{\infty\}$.
\end{lema}
\begin{proof}
Take a function $f$, as we have just done, that is ramified only over algebraic values. Let $\{b_1,\ldots ,b_n\}$ be the ramification values of the function. Now, let $g$ be the smallest degree polynomial in $\Q[T]$ such that $g(b_i)=0$ for all the $b_i$'s, which means that $P$ will be the product of the minimal polynomials of the $b_i$'s, removing repetitions.

Let us look at the function $g\circ f$. By the chain rule, for a point $P$, $\ord_P(g\circ f)=\ord_{f(P)}(g)e_P(f)$, and therefore the ramification values of $g\circ f$ are the ramification values of $g$ plus the image by $g$ of the ramification points of $f$. The latter set is contained in $\{0,\infty\}$, since we have constructed it this way. Let $\{b_1',\ldots ,b_m'\}$ be the ramification values of $g$, which are $\{g(c_i):g'(c_i)=0\}$. We can take their minimal polynomial again, call it $g_1$, and consider the function $g_1\circ g\circ f$. We are going to see that the degree of $g_1$ is strictly smaller than the degree of $g$.

Let $g'=h_1^{\alpha_1}\cdots h_k^{\alpha_k}$ be the decomposition of $g'$ into irreducible polynomials. The $c_i$'s are then the roots of these polynomials, in fact, we can number them $c_1^1,\ldots,c_{d_1}^1,c_1^2,\ldots ,c_{d_k}^k$, so that $c_{j}^i$ are the roots of $h_i$. Now, for a fixed $i$, the $g(c_j^i)=b_{j}^i$'s must be the roots of the same polynomial $\widetilde h_k$ of degree at most $\deg h_i$. Let us check this: $\Q(g(c_j^i))\subset \Q(c_j^i)$, and the degree of the minimal polynomial is $[\Q(g(c_j^i)):\Q]$, so $\deg \widetilde h_i\le \deg h_i$. Also, the minimal polynomial is the same for all of them, since we can take elements of the Galois group $\sigma_j$, such that $c_j^i=(c_1^i)^{\sigma_j}$ (since they have the same minimal polynomial). Therefore, if $\widetilde h_i(c_1^i)=0$, then for every $j$ we will have that $\widetilde h_i(c_j^i)=\widetilde h_i((c_1^i)^{\sigma_j})=(\widetilde h_i(c_1^i))^{\sigma_j}=0$.

Therefore, $g_2=\widetilde h_1\cdots \widetilde h_k$, and $\deg \widetilde h_i\le \deg h_i$, so $\deg g_2\le \deg g'<\deg g$. We can iterate this process (take now the minimal polynomial of the ramification values of $g_2$), and we will eventually reach a point where the degree is $0$, and all the ramification points will be rational (there might be many of them, since we haven't kept track of $g_1(0)$, although we know it is rational, and the same for $g_2(0)$, and so on).
\end{proof}

\noindent We can now finish the proof of this part of the theorem.

\begin{lema}
Let $\{q_1,\ldots,  q_m\}\subset \Q$. There is a polynomial $f$ that maps all these points to $0$ and ramifies only over $\{0,1,\infty\}$.
\end{lema}
\begin{proof}
We proceed by induction. If $m=2$, then the polynomial $-\frac{4}{(q_1-q_2)^2}(x-q_1)(x-q_2)$ ramifies only at $\frac{q_1+q_2}{2}$ and its value there is $1$, so we are done.

Suppose now that $q_1<q_2<\cdots <q_m$. Take an affine transformation that maps $q_1$ to $0$ and $q_3$ to $1$. Let the image of $q_2$ be $\frac{m}{m+n}$. If we take the polynomial $P(x)=\frac{(m+n)^{m+n}}{m^mn^n}x^m(1-x)^n$, one can check that its only ramification points are $0,1,\frac{m}{m+n}$ and $\infty$. Their images are, respectively, $0,0,1,\infty$, so the affine transformation composed with this polynomial maps $\{q_1,\ldots ,q_m\}$ to a set of $m-1$ points, which contains its ramification values. We can iterate this process until we reach $2$ points, and we are done.
\end{proof}

\noindent This ends the proof of the first part of Belyi's theorem. We have an easy corollary now: the Galois group acts faithfully on dessins.

\begin{prop}\label{fiel}
The Galois group acts faithfully on dessins of genus 1.
\end{prop}
\begin{proof}
Recall from the classification of elliptic curves that every elliptic curve can be written in the form $y^2=x(x-1)(x-\lambda)$ for $\lambda\in \C\setminus \{0,1\}$, and two such curves, for $\lambda$ and $\lambda'$, are isomorphic if and only if $\lambda'\in \left\{\lambda,\frac{1}{\lambda},1-\lambda,1-\frac{1}{\lambda},\frac{1}{1-\lambda},\frac{\lambda}{1-\lambda} \right\}$, or, equivalently, if their $j$-invariants are equal, where
$$
j(\lambda)=256\frac{(1-\lambda(1-\lambda))^3}{\lambda^2(1-\lambda)^2}
$$
Using this and Belyi's theorem, it is easy to produce dessins that change under the Galois action. Suppose we are given some $\sigma\in \gal\setminus\{\Id\}$. Pick some $j_0\in \QQ$ such that $\sigma(j_0)\neq j_0$. Since $j$ is surjective, there must be some $\lambda \in \QQ$ such that $j(\lambda)=j_0$ (and it will be different from $0$ and 1).

Now, in virtue of Belyi's Theorem, we can pick a Belyi function $f$ on the curve $C$ with equation $y^2=x(x-1)(x-\lambda)$. The conjugate Belyi pair will be defined on the curve $C^\sigma$ with equation $y^2=x(x-1)(x-\lambda^\sigma)$, which is not isomorphic to the first curve, since $j(\lambda^\sigma)=(j(\lambda))^\sigma=j_0^\sigma\neq j_0$. Therefore, the dessins must be non-isomorphic.
\end{proof}
\subsection{Curves with Belyi functions are defined over $\QQ$}

The key property we need to prove the reverse implication is the fact that curves with Belyi functions have finitely many Galois conjugates.

Suppose we have a Belyi pair $(C,f)$. The action of $\Gal(\C/\Q)$ on it gives other Belyi pairs $(C^\sigma,f^\sigma)$, which have the same degree as the original one. The number of Belyi pairs of a given degree is finite, since the number of possible monodromy actions on a fixed number of points is clearly finite. Therefore, if a curve has a Belyi pair, it has finitely many distinct Galois conjugates.

It is in fact true that if a curve has finitely many conjugates, then it is defined over $\QQ$, that is, it can be given by equations with coefficients in $\QQ$. The proof can be found in chapter 3 of \cite{GG}.

We are going to prove this in the special case of Belyi pairs. Now, in order to do this, we need to talk about specializations.

Suppose  we have a finitely generated field $K\subset \C$ over the rational numbers. We can pick a maximal transcendental set within $K$, which we can call $\pi_1,\ldots ,\pi_n$, and then, by the primitive element theorem, the extension will be of the form $\Q(\pi_1,\ldots ,\pi_n,u)/\Q$, where  $u$ is algebraic over $\Q(\pi_1,\ldots ,\pi_n)$. Let $m_u$ be the minimal polynomial of $u$ over $\Q(\pi_1,,\ldots ,\pi_n)$. We can clear denominators in $m_u$ so that $m_u\in \Q[\pi_1,\ldots ,\pi_n][X]$. If the resulting polynomial isn't monic, we can multiply $u$ by some element of $\C(\pi_1,\ldots ,\pi_n)$ to obtain a different $u$ for which $m_u$ is monic. From now on, we will assume that $m_u$ is monic.

A \textbf{specialization} of $(\pi_1,\ldots ,\pi_n)$ is just a set of complex numbers $(\eta_1,\ldots ,\eta_n)$. Note that, since the numbers are algebraically independent, there is a $\Q$-algebra homomorphism
$$
s:\Q[\pi_1,\ldots ,\pi_n]\lra\C
$$
Such that $s(\pi_i)=\eta_i$. Now, to extend this homomorphism to $\Q[\pi_1,\ldots ,\pi_n,u]$, it is needed that if we take the minimal polynomial $m_u$ of $u$, and we apply $s$ to its coefficients, to obtain a polynomial $m_u^s$, that $m_u^s(s(u))=0$.

In fact, this is sufficient: since $\Q[\pi_1,\ldots ,\pi_n,u]\cong \Q[\pi_1,\ldots ,\pi_n][X]/(m_u)$, a necessary and sufficient condition for a homomorphism from $\Q[\pi_1,\ldots ,\pi_n,X]$ into $\C$ to factor through this quotient is for the kernel to contain $m_u$.

Therefore, any specialization $s$ of $(\pi_1,\ldots ,\pi_n)$ gives some homomorphism into $\C$ provided that we can find a root of $m_u^s$. To guarantee this, we choose the numbers $\eta_i$ close to $\pi_i$.

We define the \textbf{distance} of a specialization $(\eta_1,\ldots ,\eta_n)$ of $(\pi_1,\ldots ,\pi_n)$ to be the maximum of $|\pi_i-\eta_i|$ (recall that every number here is a complex number!).

\begin{lema}\label{special}
Let $(\pi_1,\ldots ,\pi_n;u)$ be complex numbers such that $\pi_1,\ldots ,\pi_n$ are algebraically independent over $\Q$ and $u$ is algebraic over $\Q(\pi_1,\ldots ,\pi_n)$. Let $m_u$ be the minimal polynomial of $u$ over $\Q[\pi_1,\ldots ,\pi_n]$.

Let $\delta>0$ be a number such that it is smaller than $|u_i-u_j|/2$, where $u_i,u_j$ are any distinct roots of $m_u$.

Then, there exists an $\eps$ such that for every specialization $(\eta_1,\ldots ,\eta_n)$ of $(\pi_1,\ldots ,\pi_n)$ of distance smaller than $\eps$, the polynomial $m_u^s$ has exactly one root $u_s$ such that $|u-u_s|<\delta$.
\end{lema}
\begin{proof}
We just need to note that the coefficients of $m_u$ are polynomials on $\pi_i$, and therefore continuous functions. Also, the roots of a polynomial depend continuously of its coefficients, so choosing close enough coefficients will yield roots of the polynomial $m_u^s$ that are close enough to the roots of $m_u$.
\end{proof}

\noindent Specializations are the way we will change curves defined over transcendental fields by curves defined over $\QQ$, by using the fact that $\QQ$ is dense in $\C$.

We want to prove that our curve $C$ that has a Belyi function is isomorphic to another one whose coefficients are algebraic numbers. We are going to see how an isomorphism can be reduced to some polynomial equalities.

Take two Belyi pairs $(C,f)$ and $(C',f')$. An isomorphism between them is equivalent, as we know, to an isomorphism of $\C(t)$-algebras
$$
\Phi:\C(C)\lra \C(C')
$$
By the primitive element theorem, $\C(C)$ is generated over $\C(t)$ by some $x$. Let $F\in \C(t)[X]$ be its minimal polynomial. Then,
$$
\C(C)\cong \frac{\C(t)[X]}{(F)}
$$
And, analogously, there is some $x'$ that generates $\C(C')$ and some $F'\in \C(t)[X]$ such that
$$
\C(C')\cong \frac{\C(t)[X]}{(F')}
$$
A $\C(t)$-algebra homomorphism from $\C(C)$ to $\C(C')$ is then determined by the image of $x$. The image of $x$ will be some polynomial $\Phi(x)=P_1(x')$, where $P_1\in \C(t)$. The map then takes $P(x)$ to $P(P_1(x'))$, for any $P\in \C(t)[X]$. For the map to be well-defined, the necessary and sufficient condition is that $F(x)$ maps to $0$, that is, that $F(P_1(X))$ equals $0$ in $\C(t)[X]/(F')$. We can write this only in terms of polynomials, by saying that there exists $H_1\in \C(t)[X]$ such that
$$
F(P_1(X))=H_1(X)F'(X)
$$
For an isomorphism, we need two mutually inverse morphisms, so there must be another morphism $\Phi'$ given by another polynomial $P_2$ so that $x'$ will map to $P_2(x)$, and now there must exist a polynomial $H_2$ such that
$$
F'(P_2(X))=H_2(X)F(X)
$$
The final requirement is that the maps are mutually inverse. This means that $\Phi'(\Phi(x))=x$. In other words, there exists some $G_1\in \C(t)[X]$ such that
$$
P_1(P_2(X))=X+G_1(X)F(X)
$$
And also, $\Phi\circ \Phi'=\Id$ is equivalent to there existing a polynomial $G_2$ such that
$$
P_2(P_1(X))=X+G_2(X)F'(X)
$$
These 4 equations are the way to express an isomorphism between covers, as we have just seen.
\begin{lema}\label{isomor}
Let $\C(C)/\C(t)\cong \C(t)[X]/(F)$ and $\C(C)/\C(t)\cong \C(t)[X]/(F')$ be the extensions corresponding to two covers. There exists an isomorphism between both covers if and only if there exist $P_1,P_2,H_1,H_2,G_1,G_2\in \C(t)[X]$ such that the following identities hold:
\begin{align}\label{isomoreq1}
F(P_1(X))&=H_1(X)F'(X)\\
F'(P_2(X))&=H_2(X)F(X)\\
P_1(P_2(X))&=X+G_1(X)F(X)\\
P_2(P_1(X))&=X+G_2(X)F'(X)\label{isomoreq2}
\end{align}
\end{lema}
\noindent Let us now prove Belyi's theorem. Suppose be are given a Belyi pair $(C,f)$ with its corresponding extension $\C(C)/\C(t)\cong \C(t)[X]/(F)$. Let $K=\Q(\pi_1,\ldots ,\pi_n,u)$ be the field generated by the coefficients of $F$, where $\pi_1,\ldots ,\pi_n$ are algebraically independent and $u$ is algebraic over the field generated by the rest of them. The Belyi pair has finitely many Galois conjugates. Therefore, there are many Galois automorphisms $\sigma$ that fix the cover and map $(\pi_1,\ldots ,\pi_n)$ to some complex numbers $\pi_i^\sigma=\pi_{n+i}$ such that the set $(\pi_1,\ldots ,\pi_{2n})$ is an algebraically independent set.

Since the cover is fixed by $\sigma$, there exists an isomorphism
$$
\Phi:\C(C)\lra \C(C^\sigma)\cong \C(t)[X]/(F^\sigma)
$$
By lemma \ref{isomor}, this means that there exist polynomials such that
$$F(P_1(X))=H_1(X)F^\sigma(X)$$
$$F^\sigma(P_2(X))=H_2(X)F(X)$$
$$P_1(P_2(X))=X+G_1(X)F(X)$$
$$P_2(P_1(X))=X+G_2(X)F^\sigma(X)$$
Take the field $K_2$ generated by the coefficients of all of these polynomials. We can add some elements $\pi_{2n+1},\ldots ,\pi_{d}$ to our list $\pi_1,\ldots ,\pi_{2n}$ so that they are a maximal algebraically independent set within $K_2$. Then, there exists some $v$ such that $K_2=\Q(\pi_1,\ldots ,\pi_d,v)$.

Now we are going to take a specialization of the above formulas, in order to define an isomorphism between $\C(C)$ and some other curve defined over $\QQ$. We have seen in lemma \ref{special} that there exists some $\eps$ such that every specialization of $(\pi_1,\ldots ,\pi_d)$ of distance smaller than $\eps$ can be extended to a $\Q$-algebra homomorphism
$$
s:\Q[\pi_1,\ldots ,\pi_d,v] \lra \C
$$
Since $\QQ$ is dense in $\C$, we can take a specialization with distance smaller than $\eps$ given by
$$
\left(
\eta_1=\pi_1,\ldots ,\eta_n=\pi_n,q_{n+1},\ldots ,q_d\right)
$$\nopagebreak With $q_i\in \QQ$.

Now, $s$ is a homomorphism of $\Q$-algebras, and it can of course be extended to $\Q[\pi_1,\ldots ,\pi_d,v][t,X]$, with image in $\C[t,X]$.

We want to extend $s$ to the polynomials $F^\sigma,H_1,H_2,G_1,G_2$, and we can do this if we are careful: their coefficients, if we see them as rational functions in $t$ and $X$, are elements of $K_2$, and thus rational functions in $\pi_1,\ldots ,\pi_d,v$, whose denominators might vanish in the specialization. However, the numbers$\left(
\eta_1=\pi_1,\ldots ,\eta_n=\pi_n,\right.$ $\left.q_{n+1},\ldots ,q_d\right)$ that will make one of this denominators vanish will lie in some closed set of $\C$. Therefore, since we have a whole ball around each $\pi_i$ to choose from, we can still pick elements that won't make these denominators vanish.

Also, if we see $F^\sigma,H_1,H_2,G_1,G_2$ as polynomials in $X$, their coefficients are rational functions in $\C(t)$, whose denominators might also vanish. However, we can do the same thing, since the specializations that make this happen lie in a closed set of $\C^d$.

Therefore, there exists some distance $\eps$, such that for any specialization of this distance, the image of the polynomials $F^\sigma,H_1,H_2,G_1,G_2$ will be well-defined, since it will lie in the subring of
$$
\Q(\pi_1,\ldots ,\pi_d,v,t,x)
$$
In which $s$ can be defined (rational functions whose denominators don't vanish). We can also take the specialization so that $\pi_{n+1},\ldots,\pi_d$ are algebraic. Also, $F^s$ will equal $s$, for $u$ can only be mapped to one possible $u$, provided $\eps$ is small enough, so $s(u)=u$. Since $s$ is a $\Q$-algebra homomorphism, and $F^s=F$, we will have that
$$F(P_1^s(X))=H_1^s(X)(F^\sigma)^s(X)$$
$$(F^\sigma)^s(P_2^s(X))=H_2^s(X)F(X)$$
$$P_1^s(P_2^s(X))=X+G_1^s(X)F(X)$$
$$P_2^s(P_1^s(X))=X+G_2^s(X()F^\sigma)^s(X)$$
In other words, if we use proposition \ref{isomor}, we obtain that the extension $\C(C)/\C(t)$ is isomorphic to the extension $\C(t)[X]/((F^\sigma)^s)$, and the coefficients of $(F^\sigma)^s$ are algebraic.

We can finally prove that we can define dessins over $\QQ$ by using the fact that every function field over $\QQ$, which is algebraically closed, corresponds to a smooth projective curve (this can be proven, for example, using the Riemann-Roch Theorem, as in \cite{Miranda}).

Also, we can prove that if we have two dessins $\QQ(C_1)/\QQ(t)$ and $\QQ(C_2)/\QQ(t)$ such that, when their scalars are extended to $\C$, they give isomorphic dessins, then the initial dessins are also isomorphic. The way to do this is to prove, also using specialization, that morphisms between covers can also be defined over $\QQ$. 

\begin{lema}
Suppose we have two extensions corresponding to dessins d'enfants defined over $\QQ$, which we will call $\C(C_1)=\C(t,x_1)=\C(t)[X]/(F_1)$, and $\C(C_2)=\C(t,x_2)=\C(t)[X]/(F_2)$, where $F_1,F_2\in \QQ(t)[X]$. Every morphism from $\C(C_1)$ to $\C(C_2)$ can be defined over $\QQ$, that is, it maps $x_1$ to $P(x_2)$, where $P\in \QQ(t)[X]$. 
\end{lema}

\begin{proof}

We have seen that a morphism $\Phi$ from $\C(C_1)$ to $\C(C_2)$ is equivalent to a polynomial $P\in \C(t)[X]$, such that $\Phi(x_1)=P(X)$, for which there exists another polynomial $G\in \C(t)[X]$ such that
$$x
F_1(P(X))=G(X)F_2(X)
$$
We are going to specialize this identity  so that $P\in \QQ(t)[X]$. As usual, let $(\pi_1,\ldots ,\pi_n,u)$ generate the field that contains the coefficients of all the polynomials involved. We can choose a specialization of small enough distance so that $u$ is mapped by the $\Q$-algebra homomorphism to one unique root of the image of its minimal polynomial. Also, if the distance is small enough, as before, no denominators will become $0$ in the coefficients of the polynomials, and also the coefficients of $F_1$ and $F_2$ will remain unchanged. If we apply the resulting homomorphism $s$, we are left with
$$
F_1(P^s(X))=G^s(X)F_2(X)
$$
So $P^s(x_2)=\Phi(x_1)$ will be a root of $F_1$ in $\C(C_2)$. However, there are only finitely many such roots, (since $\C(C_2)$ is a field). And if we make the distance of the specialization smaller, $P^s(x_2)$ will lie in some neighborhood of $P(x_2)$\footnote{Since $P^s(x_2)$ can be given by a finite set of complex numbers (the coefficients in the rational functions in $t$ that are the coefficients of $P^s$), we can give these values the topology of $\C^N$. Anyway, it is clear that the function mapping the specialization to $P^s(x_2)$ is continuous, and that finite sets in $\C^N$ are closed.}, that is a root of $F_1$, that will contain no other roots of $F_1$. Therefore, $P^s(x_2)$ will equal $P(x_2)$, and the morphism will be defined over the algebraic numbers. $G$ will also be defined over the algebraic numbers, since it is $F_1(P)/F_2$.
\end{proof}

\noindent Therefore, morphisms of the dessins d'enfants over $\QQ$ correspond bijectively to morphisms of dessins d'enfants over $\C$. In particular, if we have two extensions $\QQ(C_1)/\QQ(t)$ and $\QQ(C_2)/\QQ(t)$ such that, when the scalars are extended to $\C$, they give the same dessin d'enfant, they must be isomorphic.

Also, we have our directed set of extensions of $\QQ(t)$, that enables us to construct the field $\KK$, which is the extension that contains all the dessins d'enfants, just as in the complex case (like we did in section \ref{thefieldk}). Since morphisms between dessins are defined over $\QQ$, we have that taking $\QQ$ as the base field doesn't change the Galois group of the extensions corresponding to dessins d'enfants. Therefore, since $\Gal(\KK/\QQ(t)$ is the limit of the Galois groups of Galois subextensions, we have that
$$
\Gal(\KK/\QQ(t))\cong \lim_{\substack{(C,t)\text{ Belyi pairs}\\ \longleftarrow}} \Gal(\QQ(C)/\QQ(t))\cong \lim_{\substack{(C,t)\text{ Belyi pairs}\\ \longleftarrow}} \Gal(\C(C)/\C(t))=\Fd
$$
\section{The action of $\mathrm{Gal}(\overline{\mathbb Q}/\mathbb Q)$ on $\Fd$}
Now that we know that dessins are defined over $\QQ$, we can talk about the action of $\gal$ (recall section \ref{galaction}). Recall that, for every dessin $\QQ(C)/\QQ(t)$, and every $\sigma\in \gal$, there is a conjugate dessin $\QQ(C^\sigma)/\QQ(t)$, and furthermore, the Galois group also acts on morphisms between dessins, and it does so in a functorial way: if we have two morphisms $\phi,\psi$, then
$$
(\phi\circ\psi)^\sigma=\phi^\sigma \circ \psi^\sigma
$$
If these morphisms are automorphisms of an extension, then , it is clear that the Galois action will give an isomorphism from $\Aut(C,f)$ to $\Aut(C^\sigma,f^\sigma)$ (from now on, we will avoid calling these automorphism groups $\Gal(\QQ(C)/\QQ(t))$, to avoid confusion with $\gal$).

Since the groups $\Aut(C,f)$ are quotients of $\Fd$, it is interesting to consider whether we can define an action of $\gal$ not just on its quotients, but on the whole group $\Fd$ (this should have been easy to guess, given this section's title). Our objective now will be to prove this fact.

There is something more that we want from this section: if we look at the fields $\Q(t)\subset \QQ(t)\subset \KK$, we see that the extension $\QQ(t)/\Q(t)$ is clearly Galois (since it is basically the same as the extension $\QQ/\Q$), and its Galois group clearly is isomorphic to $\gal$. Therefore,
$$
\gal=\Gal(\QQ(t)/\Q(t))\cong\frac{\Gal(\KK/\Q(t))}{\Gal(\KK/\QQ(t))}=\frac{\Gal(\KK/\Q(t))}{\Fd}
$$
In other words, we have the following exact sequence:
$$
1\lra \Fd \lra \Gal(\KK/\Q(t))\lra \gal \lra 1
$$
For every exact sequence $1\lra N\lra G \lra H\lra 1$, there is an induced map $H\lra \Out (N)$. Recall that, for a group $G$, $\Inn(G)$ is the set of automorphisms given by conjugation in the group, and $\Out(G)=\frac{\Aut(G)}{\Inn(G)}$. Thus there is a map $\gal \lra \Out(\Fd)$.

An element of $\Out(\Fd)$ permutes the conjugacy classes of subgroups of $\Fd$, since it a well-defined automorphism up to conjugation. This permutation is actually the Galois action we have defined on dessins: Suppose we have some $\wh \sigma\in \Gal(\KK/\Q(t))$ such that its class modulo $\Fd$ is $\sigma \in \gal$. A dessin $\QQ(C)/\QQ(t)$ is mapped to $\QQ(C)^{\wh\sigma}/\QQ(t)$. If $\QQ(C)=\QQ(t,x)\cong\QQ(t)[X]/(F)$, then $F^{\sigma}(x^{\wh\sigma})=(F(x))^{\wh\sigma}=0$, so its image, $\QQ(t,x^{\wh\sigma})$, is isomorphic, as an extension, to $\QQ(t)[X]/(F^{\sigma})$, which is the Galois action we have already defined. Therefore, the subgroup ${\wh\sigma}^{-1}H {\wh\sigma}<\Fd<\Gal(\KK/\Q(t))$ is the subgroup $H^\sigma$, up to conjugation.

\begin{prop}\label{out}
The homomorphism $\gal \lra \Out(\Fd)$ we have just defined is injective.
\end{prop}
\begin{proof}
Every $\sigma\in \gal$ is mapped to an element different from the identity. If some element weren't, its image would be an inner automorphism of $\Fd$, and in particular, it would leave every conjugacy class of subgroups unchanged. However, we have seen that the Galois group acts faithfully on dessins (proposition \ref{fiel}), so $\sigma$ can't be mapped to the identity of $\Out(\Fd)$.
\end{proof}

\noindent We are going to prove now that we can find a homomorphism $\gal\lra \Aut(\Fd)$, and that actually, the extension we wrote before is split, so
$$
\Gal(\KK/\Q(t))=\gal \ltimes \Fd
$$
Here's some basic facts from split extensions: an extension $N\lra G\overset{f}\lra H$ is \textbf{split} if there exists a homomorphism $s:H\lra G$ such that it is a section, i.e. $f\circ s=\Id_H$. This happens if and only if $G$ is isomorphic to the semidirect product of $N$ and $H$, where $H$ is identified with its image in $G$ by $s$.

\begin{teor}\label{split}
The sequence
$$
1\lra \Fd \lra \gall \lra \gal\lra 1
$$
splits.
\end{teor}
\begin{proof}
Let's try to define the map to $\Aut(\Fd)$ now. Take some $\sigma\in \gal$. Clearly, $\sigma$ gives an isomorphism between $\Gal(\KK/\QQ(t))$ and $\Gal(\KK^\sigma/\QQ(t))$. However, both extensions are isomorphic, so what we need is a canonical way to identify them, so that $\sigma$ is an automorphism, and the resulting function $\gal \times \Gal(\KK/\QQ(t))\lra \Gal(\KK/\QQ(t))$ is indeed an action.

The canonical way is given by using the base points in dessins. In $\KK$, every subextension corresponds to a dessin with a base point. If we have some $\KK^H\subset \KK$, there is a base point in the corresponding curve $C$, which we will call $P_H$. We could call it $P_C$, but one curve can have different dessins defined on it, or the same one with different base points, so it is nicer to associate it with the subgroup.

Recall that we defined the inclusions between subfields of $\KK$ so that they would preserve the base point (since morphisms correspond to inclusions between subgroups of $F_2$, and there is an inclusion if and only if there is a covering preserving base points, see proposition \ref{uniquemorphism}), so, if $i:\KK^{H_1}\lra \KK^{H_2}$ is one such inclusion, then $i^*(P_{H_2})=P_{H_1}$. Also, this means that we can define a valuation $\wh P$ in $\KK$, such that its restriction to each $\KK^H$ is $P_H$.

Take some open subgroup $H<\Fd$ and $\KK^H$. if we consider the dessin given by $\left(\KK^H\right)^\sigma$, with the base point $(P_H)^\sigma$, then there is one unique subfield of $\KK$ such that the extensions are isomorphic taking the base points into account. That is, there exists a unique extension, which we will call $\KK^{H^\sigma}$, and a unique isomorphism $\phi_{\sigma,H}:(\KK^H)^\sigma\lra \KK^{H^\sigma}$ such that $\phi_{\sigma,H}^*(P_{H^\sigma})=P_H^\sigma$. When we have defined the action of $\gal$ on $\Fd$, we will have that $\K^{H^\sigma}$ is indeed the fixed field of $H^\sigma$, but for now consider it just as notation.

Thus we have, for each open subgroup $H<\Fd$, an isomorphism (of extensions) from $(\KK^H)^\sigma$ to some $\KK^{H^\sigma}$. We have to prove that this induces a morphism
$$
\phi_\sigma:\KK^\sigma \lra \KK
$$
Let us check that it is. Suppose we have two open subgroups $H_1<H_2$, and the corresponding homomorphism $i_{H_2H_1}:\KK^{H_2}\lra \KK^{H_1}$. Since $\KK^{H_i^\sigma}$ are isomorphic extensions to $(\KK^{H_i})^\sigma$, and their base points are carried over by these isomorphisms, there will be a unique inclusion $i_{H_2^\sigma H_1^\sigma}$ between them preserving the base point. What we need to check for $\phi_\sigma$ to be defined that the following diagram (of homomorphisms of $\C(t)$-algebras) is commutative:
\begin{center}\hspace{0.5 cm}
\begindc{\commdiag}[60]
\obj(0,0)[a]{$\left(\KK^{H_2}\right)^\sigma$}
\obj(0,1)[b]{$\left(\KK^{H_1}\right)^\sigma$}
\obj(1,0)[c]{$\KK^{H_2^\sigma}$}
\obj(1,1)[d]{$\KK^{H_1^\sigma}$}
\mor{a}{b}{$i_{H_2H_2}^\sigma$}
\mor{c}{d}{$i_{H_2^\sigma H_1^\sigma}$}[\atright,0]
\mor{a}{c}{$\phi_{\sigma,H_2}$}[\atright,0]
\mor{b}{d}{$\phi_{\sigma,H_1}$}
\enddc
\end{center}
\noindent It is indeed: there is at most one morphism from $(\KK^{H_2})^\sigma$ to $\KK^{H_1^\sigma}$ such that the restriction of the valuation corresponding to $P_{H_1^\sigma}$ is $P_{H_2}^\sigma$. However,
$$
(\phi_{\sigma,H_1}\circ i^\sigma_{H_2H_1})^*(P_{H_1^\sigma})=i^{\sigma*}_{H_2H_1}\left(\phi_{\sigma,H_1}^*(P_{H_1^\sigma})\right)=i^{\sigma*}_{H_2H_1}(P_{H_1}^\sigma)=\left(i^*_{H_1H_2}(P_{H_1})\right)^\sigma=P_{H_2}^\sigma
$$
And also,
$$
(i_{H_2^\sigma H_1^\sigma}\circ \phi_{\sigma,H_2})^*(P_{H_1^\sigma})=\phi_{\sigma,H_2}^*\left(
(i_{H_2^\sigma H_1^\sigma}^*(P_{H_1^\sigma})
\right)=
\phi_{\sigma,H_2}^*\left(
P_{H_2^\sigma}
\right)=P_{H_2}^\sigma
$$
So the diagram commutes.

Therefore, we have a morphism
$$
\phi_\sigma:\KK^\sigma \lra \KK
$$
And, for each $\sigma\in \gal$, there is some automorphism $\wt \sigma\in \Gal(\KK/\Q(t))$ (it doesn't fix $\QQ(t)$, because $\sigma$ doesn't give homomorphisms of extensions of $\QQ(t)$) defined by
$$
\wt\sigma=\phi_\sigma \circ \sigma
$$
What we want to prove now is that
$$
\wt{}:\gal\lra \Gal(\KK/\QQ(t))
$$
is a group homomorphism. It is clearly a section of the quotient $\Gal(\KK/\Q(t))\lra \gal$, since it maps extensions $\KK^H/\QQ(t)$ to extensions that are isomorphic to $(\KK^H)^\sigma/\QQ(t)$, because $\phi_\sigma$ is an isomorphism of $\QQ(t)$-algebras.

Let $\sigma,\tau\in \gal$. Let us prove the following: (recall that the Galois group acts \textbf{on the right})
$$\wt{\tau\sigma}=\wt{\sigma\circ\tau}=\wt\sigma \circ\wt\tau=\wt \tau\wt \sigma
$$
We only need to check this in finite extensions, since their union is $\KK$. Let $\KK^H$ be one such extension. What we are trying to prove is
$$
\phi_{\tau\sigma} \circ \sigma \circ \tau=\phi_{\sigma}\circ \sigma \circ \phi_\tau \circ \tau
$$
As we said, it suffices to prove this when restricted to $\KK^H$.

Let us sketch what we are trying to prove:
\begin{center}\hspace{0.5 cm}
\begindc{\commdiag}[60]
\obj(0,1)[a]{$\KK^H$}
\obj(1,1)[b]{$\left(\KK^H\right)^\tau$}
\obj(2,1)[c]{$\left(\KK^H\right)^{\tau\sigma}$}
\obj(3,1)[d]{$\KK^{H^{\tau\sigma}}$}
\obj(1,0)[e]{$\KK^{H^\tau}$}
\obj(2,0)[f]{$\left(\KK^{H^\tau}\right)^\sigma$}
\obj(3,0)[g]{$\KK^{\left(H^\tau\right)^\sigma}$}
\mor{a}{b}{$\tau$}
\mor{b}{c}{$\sigma$}
\mor{c}{d}{$\phi_{\tau\sigma}$}
\mor{b}{e}{$\phi_\tau$}
\mor{e}{f}{$\sigma$}
\mor{f}{g}{$\phi_{\sigma}$}
\enddc
\end{center}
\noindent Recall that for any $\QQ$-algebra $A$, $(A^\tau)^\sigma=A^{\tau\sigma}$, so $\sigma\left((\KK^H)^\tau\right)$ actually equals $(\KK^H)^{\tau\sigma}$, and $\phi_{\tau\sigma}$ is defined on it. We need to prove that the objects at the far right of the diagram are the same, and that the diagram commutes when we see them as the same object.

First of all, $\phi_\tau$ is an isomorphism of $\QQ(t)$-algebras, and $(\KK^H)^\tau$ is isomorphic as a $\QQ(t)$-algebra to $\KK^{H^\tau}$. Also, $\phi_\tau^*(P_{H^\tau})=P_H^\tau$.

Since both algebras are isomorphic, their images by $\sigma$ also are isomorphic, and the isomorphism between them is actually $\phi_\tau^\sigma$. Let us trace back the point $P_H$: we have that $\sigma^*((P_H^\tau)^\sigma)=\sigma^*(P_H^{\tau\sigma})=P_H^{\tau}$, by definition, and, on the lower part of the diagram, $\sigma^*(P_{H^\tau}^\sigma)=P_{H^\tau}$. Also,
$$
\phi_\tau^{\sigma*}(P_{H^\tau}^\sigma)=\phi_\tau^{*\sigma}(P_{H^\tau}^\sigma)=\left(\phi_\tau(P_{H^\tau})\right)^\sigma=(P_H^\tau)^\sigma=P_H^{\tau\sigma}
$$
So $\phi_\tau^\sigma$ is an isomorphism between the objects in the third column which carries one base point to the other.

We have that $(\KK^H)^{\tau\sigma}$ and $\left(\KK^{H^\tau}\right)^\sigma$ are extensions isomorphic by an isomorphism carrying $P_{H^\tau}^\sigma$ to $P_H^{\tau\sigma}$. Now, by definition of $\phi_\sigma$, $$\phi_\sigma^*(P_{(H^\tau)^\sigma})=P_{H^\tau}^\sigma$$Also, by definition of $\phi_{\tau\sigma}$,
$$
\phi_{\tau\sigma}^*(P_{H^{\tau\sigma}})=P_H^{\tau\sigma}
$$
The whole situation can be summed up in the following diagram, which states what the images of points are:
\begin{center}\hspace{0.5 cm}
\begindc{\commdiag}[60]
\obj(0,1)[a]{$P_H$}
\obj(1,1)[b]{$P_H^\tau$}
\obj(2,1)[c]{$P_H^{\tau\sigma}$}
\obj(3,1)[d]{$P_{H^{\tau\sigma}}$}
\obj(1,0)[e]{$P_{H^\tau}$}
\obj(2,0)[f]{$P_{H^\tau}^\sigma$}
\obj(3,0)[g]{$P_{\left(H^\tau\right)^\sigma}$}
\mor{b}{a}{$\tau^*$}[\atleft,\aplicationarrow]
\mor{c}{b}{$\sigma^*$}[\atleft,7]
\mor{d}{c}{$\phi_{\tau\sigma}^*$}[\atleft,\aplicationarrow]
\mor{e}{b}{$\phi_\tau^*$}[\atleft,\aplicationarrow]
\mor{f}{e}{$\sigma^*$}[\atleft,\aplicationarrow]
\mor{g}{f}{$\phi_{\sigma}^*$}[\atleft,\aplicationarrow]
\mor{f}{c}{$\phi_\tau^{\sigma*}$}[\atleft,\aplicationarrow]
\enddc
\end{center}
\noindent We can see that $\KK^{H^{\tau\sigma}}$ and $\KK^{(H^\tau)^\sigma}$ are isomorphic (as extensions), by an isomorphism preserving their base points, namely
$$
\phi_\sigma \circ \phi_\tau^\sigma \circ \phi_{\tau\sigma}^{-1}
$$
Therefore, they are equal, and $H^{\tau\sigma}=(H^\tau)^\sigma$. This gives an action of $\gal$ on the open subgroups of $\Fd$, which is a lifting of the one we already have on their conjugacy classes.

We also have that the square
\begin{center}\hspace{0.5 cm}
\begindc{\commdiag}[60]
\obj(1,1)[b]{$\left(\KK^H\right)^\tau$}
\obj(2,1)[c]{$\left(\KK^H\right)^{\tau\sigma}$}
\obj(1,0)[e]{$\KK^{H^\tau}$}
\obj(2,0)[f]{$\left(\KK^{H^\tau}\right)^\sigma$}
\mor{b}{c}{$\sigma$}
\mor{b}{e}{$\phi_\tau$}
\mor{e}{f}{$\sigma$}
\mor{c}{f}{$\phi_\tau^\sigma$}
\enddc
\end{center}
\noindent Commutes by definition.

Also, the diagram
\begin{center}\hspace{0.5 cm}
\begindc{\commdiag}[30]
\obj(4,2)[c]{$\left(\KK^H\right)^{\tau\sigma}$}
\obj(6,1)[d]{$\KK^{H^{\tau\sigma}}$}
\obj(4,0)[f]{$\left(\KK^{H^\tau}\right)^\sigma$}
\mor{c}{d}{$\phi_{\tau\sigma}$}
\mor{c}{f}{$\phi_\tau^\sigma$}
\mor{f}{d}{$\phi_{\sigma}$}
\enddc
\end{center}
\noindent Commutes because all the morphisms involved are isomorphisms of $\QQ(t)$-algebras, and they preserve the base points (this is written two diagrams above).

So finally, we have proven that
$$
\wt{\tau\sigma}=\tilde\tau\tilde \sigma
$$
Which is what we wanted to prove.
\end{proof}

\noindent Now that we have the splitting that makes $\gall=\Fd\rtimes \gal$, the lifting to $\Aut(\Fd)$ is just
$$
\sigma \longmapsto \left(
g \mapsto g^\sigma=\wt\sigma^{-1} \tau \wt\sigma
\right)
$$
If we take an open subgroup $H<\Fd$, and we apply this automorphism to it, we get that
$$
h\in H\Llra h|_{\KK^H}=\Id \Llra \wt\sigma h \wt\sigma^{-1}|_{\wt\sigma(\KK^H)}=\Id
$$
Since $\wt\sigma(\KK^H)=\KK^{H^\sigma}$, it follows that the subgroup that fixes $\KK^{H^\sigma}$ is $H^\sigma$, which is the reason why we called it that way in the first place.
\subsection{The field of moduli and fields of definition}
We are going to describe two special fields related to a dessin. Take a Belyi pair $(C,f)$, such that the curve is $V(g_1,\ldots ,g_m)\subset \P^n$. Take the field generated by the coefficients of the $g_i$'s and $f$. These are all algebraic, so this field is a number field $K$. We call it  field of definition.
\begin{definicion}
A \textbf{field of definition} of a Belyi pair $(C,f)$, or a dessin d'enfant, is a number field $K$ such that both the curve $C$ and the Belyi function $f$ can be defined with coefficients in $K$.
\end{definicion}
\noindent A dessin can have many fields of definition: first of all, if some $K$ is a field of definition, every field containing it is also a field of definition. Nonetheless, it is not even true that there must exist a smallest field of definition.
\begin{definicion}
Let $(C,f)$ be a Belyi pair, with a corresponding extension $\KK^H/\QQ(t)$. A \textbf{model} for $(C,f)$ over a number field $K$ is a set of equations $g_1,\ldots ,g_m$ over $K$ such that $C\cong V(g_1,\ldots g_m)$ and such that in these coordinates, $f$ has equations in $K$.
\end{definicion}

\noindent Suppose we are given a finite extension $A/\Q(t)$ of $\Q(t)$ contained in $\KK$. Then, this extension can be split into two parts: if we consider $\Q(t)\subset A\cap \QQ(t) \subset A$, we see that $A\cap \QQ(t)$ is $K(t)$ for some number field $K$. The second part is an finite extension of $K(t)$ that doesn't intersect $\QQ(t)$, so it is of the form
$$
K(t,x)\cong K(t)[X]/(F)
$$
Where $F\in K(t)[X]$. Now, if we take the field $\QQ A$, that is, the largest field containing both, we are going to see that we get the field
$$
\QQ(t,x)\cong\QQ(t)[X]/(F)
$$
The extensions $\QQ(t,x)/\QQ(t)$ and $K(t,x)/K(t)$ must have the same degree, for
$$
[\QQ(t,x):\QQ(t)]=
\left[
\Gal(\KK/\QQ(t)):\Gal(\KK/\QQ(t,x))
\right]=
\left[
\Gal(\KK/\QQ(t)):\Gal(\KK/\QQ(t))\cap \Gal(\KK/K(t,x))
\right]
$$
And by the second isomorphism theorem, this equals
$$
[\QQ(t,x):\QQ(t)]=
\left[
\Gal(\KK/\QQ(t))\Gal(\KK/K(t,x)):\Gal(\KK/K(t,x))
\right]=$$ $$=
\left[
\Gal(\KK/K(t)):\Gal(\KK/K(t,x))
\right]=[K(t,x):K(t)]
$$
Therefore, the polynomial $F$ is also irreducible in $\QQ(t)[X]$, so it is also the minimal polynomial of $x$ over $\QQ(t)$. Therefore, as we wanted, $\QQ(t,x)\cong \QQ(t)[X]/(F)$.

Suppose now that $(C,f)$ is given by some equations defined over $K$, so that the extension is $\QQ(t,x_1,\ldots ,x_n)\cong\QQ(t)[X_1,\ldots X_n]/(g_1,\ldots ,g_m)$. We can consider the field $K(t)[X_1,\ldots ,X_n]/(g_1,\ldots ,g_m)$. For the same reasons as before, the intersection of this field with $\QQ(t)$ will be $K(t)$, and the field generated by it and $\QQ(t)$ will be the previous field. Therefore, we can see a model also as a field extension.

\begin{lema}
A model for a dessin given by an extension $\KK/F/\QQ(t)$ over a number field $K$ is equivalent to a finite subextension $\KK/\wt F/\Q(t)$ such that $\wt F\cap \Q(t)=K(t)$ and $\QQ(t)\wt F=F$.
\end{lema}

\noindent Let us use the Galois correspondence to swap extensions for groups. Recall that $\Fd$ is naturally included in $\gall$.

\begin{lema}\label{fielddef}
A model for a dessin given by an open subgroup $H<\Fd$ is equivalent to an open subgroup $\wt H$ of $\gall=\gal \ltimes \Fd$ such that $\Fd\wt H=\Gal(\QQ/K) \ltimes \Fd$ and $\wt H\cap \Fd=H$.
\end{lema}
\noindent This comes from the equivalence of these diagrams:
\begin{center}\hspace{0.5 cm}
\begindc{\commdiag}[30]
\obj(0,3)[e]{$\Q(t)$}
\obj(4,2)[a1]{$\Gal(\QQ/K)\ltimes \Fd$}
\obj(0,2)[a]{$K(t)$}
\obj(-1,0)[b]{$\QQ(t)$}
\obj(1,1)[c]{$\wt F$}
\obj(0,-1)[d]{$F$}
\mor{a}{b}{}[\atleft,\solidline]
\mor{a}{c}{}[\atleft,\solidline]
\mor{b}{d}{}[\atleft,\solidline]
\mor{c}{d}{}[\atleft,\solidline]
\obj(4,3)[e1]{$\gall$}
\obj(3,0)[b1]{$\Fd$}
\obj(5,1)[c1]{$\wt H$}
\obj(4,-1)[d1]{$H$}
\mor{a1}{b1}{}[\atleft,\solidline]
\mor{a1}{c1}{}[\atleft,\solidline]
\mor{b1}{d1}{}[\atleft,\solidline]
\mor{c1}{d1}{}[\atleft,\solidline]
\mor{a1}{e1}{}[\atleft,\solidline]
\mor{a}{e}{}[\atleft,\solidline]
\enddc
\end{center}
\noindent Note that
$$
\Gal(\QQ/K)=\frac{\wt H \Fd}{\Fd}
$$
Is equivalent to the fact that the image of $\wt H$ by the projection onto $\gal$ is $\Gal(\QQ/K)$. Also, by the second isomorphism theorem,
$$
\frac{\wt H}{\wt H\cap\Fd}\cong \frac{\wt H \Fd}{\Fd}=\Gal(\QQ/K)
$$
Let us look now at another field related to a dessin: the field of moduli. Take a dessin, given by a subgroup $H$. Consider the subgroup $\mathrm{St} H<\gal$ made up of the automorphisms that fix the dessin (up to isomorphism, that is, to conjugation in $\Fd$). If $K$ is any field of definition of $H$, this subgroup obviously contains $\Gal(\QQ/K)$. It is therefore a closed subgroup of $\gal$, and it has a fixed field defined by it. This field is called the field of moduli of the dessin.

\begin{definicion}
The \textbf{field of moduli} of a dessin is the fixed field of the subgroup of $\gal$ that leaves it invariant.

If a dessin is given by a Belyi pair $(C,f)$, we denote it by $\M(C,f)$.
\end{definicion}

\noindent It is clear that any field of definition contains the field of moduli: If $K$ is a field of definition, any element of $\Gal(\QQ/K)$ will fix the dessin, and therefore the subgroup that fixes the dessin contains $\Gal(\QQ/K)$. Since inclusions are reversed by the Galois correspondence, the field of moduli is contained in any such $K$.

Suppose a dessin's field of moduli was indeed a field of definition. Then, it would be the minimal field of definition. Now, the question is: when is a dessin defined over its field of moduli? Actually, not every dessin is defined over its field of moduli. There is an example of this phenomenon in \cite{rabbit} and on section \ref{notdefined}. Nonetheless, some dessins are defined over their field of moduli, such as regular dessins and dessins with no automorphisms. We will see this in short time.

\noindent Let us give an interpretation of the field of moduli in terms of subgroups of $\gall$, as we did with the field of definition.

As with any semidirect product, we can write the elements of $\gall=\gal \ltimes \Fd$ in a unique way as $\wt \sigma g$, which we can write as $(\sigma,g)\in \gal\times \Fd$ (we are not saying that the lifting of $\gal$ is unique, only that given this lifting, the way of writing the elements as pairs is unique), and the product is given by

$$
(\sigma_1,g_1)(\sigma_2,g_2)=\wt\sigma_1g_1\wt\sigma_2g_2=\wt \sigma_1 \wt \sigma_2 \wt \sigma_2^{-1} g_1\wt\sigma_2g_2=\left(\sigma_1\sigma_2,g_1^{\sigma_2}g_2\right)
$$
Now, suppose some $\sigma$ fixes a dessin given by an open subgroup $H$. This means that $H^\sigma=gHg^{-1}$ for some $g\in \Fd$. If we see $H$ as a subgroup of $\gall$, this means that, for every $h\in H$,
$$
(\sigma,g)^{-1}(1,h)(\sigma,g)=\left(\sigma^{-1},(g^{-1})^{\sigma^{-1}}\right)(1,h)(\sigma,g)=
(1,g^{-1}h^{\sigma}g)
\in 1\ltimes H
$$
So we see that $\sigma\in \mathrm{St} H$ if and only if there is some $g\in \Fd$ such that $(\sigma,g)\in N_{\gall}(H)$.

To shorten the notation, we will call $G=\gall$.

Therefore, the field of moduli can be seen as the fixed field of the projection onto $\gal$ of the group $N_{G}(H)$.

\begin{lema}
The field of moduli of a dessin corresponding to an open subgroup $H<\Fd$ is the projection onto $\gal$ of $N_{G}(H)$.
\end{lema}

\noindent Take any model $\wt H$ for a dessin given by a subgroup $H$. Since $\Fd$ is normal in $G$, for any $h\in \wt H$, $h Hh^{-1}\subset \Fd$, but also, it is contained in $\wt H$. Therefore, $h H \wt h^{-1}\subset  \Fd \cap \wt H=H$, so $\wt H\subset N_G(H)$.

If we have a dessin given by a Belyi pair $(C,f)$ and a subgroup $H$, let us call $\M(H)=\Gal(\QQ/\M(C,f))$. By the second isomorphism theorem,
$$
\M(H)=\frac{N_G(H)}{N_G(H)\cap \Fd}=\frac{N_G(H)}{N_{\Fd}(H)}\cong \frac{N_G(H)\Fd}{\Fd}
$$
\begin{prop}\label{model}
A dessin given by a subgroup $H$ is defined over its field of moduli if and only if the following exact sequence splits:
$$
1\lra \frac{N_{\Fd}(H)}{H} \lra \frac{N_G(H)}{H}\lra \M(H)\lra 1
$$
\end{prop}
\begin{proof}
Recall that a sequence $1\lra N\lra G \lra H\lra 1$ splits if and only if there is a subgroup $\wt H<G$ such that $\wt H\cap N=1$ and $\wt HN=G$. A sketch of the proof goes as follows: if there is such a group, it maps bijectively onto the quotient $H$, so the inverse of this map is the desired section. Conversely, if there is a section, its image will be the subgroup $\wt H$.

Therefore, we are saying that a field is defined over its field of moduli if and only if there is a subgroup $\overline{\wt H}<N_G(H)/H$ such that
$$
\o{\wt H}\cap \frac{N_{\Fd}(H)}{H}=1;\o{\wt H}\frac{N_{\Fd}(H)}{H}= \frac{N_G(H)}{H}
$$
If we call $\wt H$ the preimage of this group by the quotient by $H$, then $\wt H<N_G(\Fd)$, and
$$
\wt H \cap N_{\Fd}(H)=H;\wt HN_{\Fd}(H)=N_G(H)
$$
Therefore, if we look back to lemma \ref{fielddef}, we have that $\wt H$ is a model for $H$. The field of definition of this model is the one with Galois group $\wt H/(\wt H\cap \Fd)$. Since $\wt H<N_G(H)$, then $\wt H\cap \Fd\subset N_{\Fd}(H)$, so $\wt H \cap \Fd=H$, and therefore the Galois group of the model's field of definition is $\wt H/H\cong \M(H)$ (since $\wt H$ maps bijectively onto the fourth term in the exact sequence). In conclusion, if the sequence splits, there is a model defined over the field of moduli.

Reciprocally, if we have a model $\wt H$ such that $\wt H/(\Fd\cap \wt H)\cong \M(H)$, this means that $\wt H \Fd=N_G(H)\Fd$. Since $\wt H\subset N_G(H)$ for any model, this means that $\wt H N_{\Fd}(H)=N_G(\wt H)$. Also, since $\wt H$ is a model,
$$
\wt H\cap N_{\Fd}(H)=\wt H \cap \Fd=H
$$
So a model defined over the field of moduli splits the exact sequence.
\end{proof}

\noindent As a corollary, we see that both regular dessins and dessins without automorphisms are defined over their fields of moduli.

\begin{corol}
Regular dessins and dessins without (nontrivial) automorphisms are defined over their fields of moduli.
\end{corol}
\begin{proof}
Take a dessin given by $H<\Fd$. Recall, from the proof of proposition \ref{regular}, that a dessin's automorphism group is $N_{\Fd}(H)$. In the case where a dessin has no automorphisms, this group is 1, so the exact sequence in the last lemma becomes
$$
1\lra 1\lra N_G(H)/H\lra \M(H)\lra 1
$$
So it is automatically split.

On the other hand, if a dessin is regular, $N_{\Fd}(H)=H$, so the sequence is
$$
1\lra \Fd/H\lra N_G(H)/H \lra \M(H)\lra 1
$$
In this case, $N_G(H)$ contains $\Fd$, so the sequence will be split because $\gall=\gal \ltimes  \Fd$. If we take $\wt H=\M(H)\ltimes H$, it will be the section we are looking for.
\end{proof}
\noindent These two results can be seen in Wolfart's survey \cite{abc}, where he uses Theorem 1 in Weil's paper \cite{weil}. This Theorem, in our case, states the following.
\begin{prop}
A Belyi pair $(C,f)$ is defined over its field of moduli $\M(C,f)$ if and only if for every $\sigma\in \Gal(\QQ/\M(C,f))$, there is an isomorphism
$$
\phi_\sigma:(C,f)\lra (C^\sigma,f^\sigma)
$$
Such that, for every $\sigma, \tau\in \Gal(\QQ/\M(C,f))$, the following holds:
$$
\phi_{\tau\sigma}=\phi_\sigma^\tau\circ \phi_\tau
$$
\end{prop}

\noindent If we look at this proposition closely, we can see that the condition is the same as the splitting of our exact sequence. Suppose we have a splitting for a subgroup $H$, with a model $\wt H$. Then, for every $\sigma\in \M(H)$, there is some $g_\sigma\in \Fd$ such that $(\sigma,g_\sigma)\in \wt H$. Recall that this $g_\sigma$ had to verify that $H^\sigma=g_\sigma Hg_\sigma^{-1}$. Therefore,
$$g_\sigma\left(\KK^{H^\sigma}\right)=
g_\sigma\left(\KK^{g_\sigma H g_\sigma^{-1}}\right)=\KK^H$$
So $g_\sigma$ gives an isomorphism from $\KK^{H^\sigma}$ to $\KK^H$. Also, the choice of $g_\sigma$ doesn't change this isomorphism: take another $g_\sigma'$. Then, $(g_\sigma',\sigma)^{-1}(g_\sigma,\sigma)\in \wt H\cap \Fd=H$, so, for some $h\in \Fd$,
$$
(1,h)=\left(\sigma^{-1},({g_{\sigma}'}^{-1})^{\sigma^{-1}}\right)(\sigma,g_\sigma)=(1,g_\sigma'^{-1}g_\sigma )
$$
So $g_\sigma\in g_\sigma' H$, or equivalently, $g_\sigma\in g_\sigma'Hg_\sigma'^{-1}g_\sigma'=H^{\sigma}g_\sigma'$. Therefore, $g_\sigma g_\sigma'^{-1}\in \Gal(\KK/\KK^{H^\sigma})$, so their actions on the field $\KK^{H^\sigma}$ are the same, i.e. for any $f\in\KK^{H^\sigma}$, $f^{g_\sigma}=f^{g_\sigma'}$.

If we define $\phi_\sigma$ to be the morphism from $\QQ(C)$ to $\QQ(C^\sigma)$ such that $\phi_\sigma^*=g_\sigma$, then the condition is satisfied, as we can see. $\wt H$ is a subgroup, so
$$
(\tau\sigma,g_{\tau\sigma})=(\tau,g_\tau)(\sigma,g_\sigma)=(\tau\sigma,g_\tau^\sigma g_\sigma)
$$
So $g_{\tau\sigma}= g_\tau^\sigma g_\sigma=g_\sigma\circ g_\tau^\sigma $. If we make $\phi_\sigma=g_\sigma^*$, and the same for $\tau$, we have that
$$
\phi_{\tau\sigma}=\phi_\sigma^\tau\circ \phi_\tau
$$
Reciprocally, suppose such $\phi_{\sigma}$ exist. Then, we can follow this reasoning backwards to pick every $g_\sigma$ for each $\sigma\in \M(H)$, such that $\phi_\sigma^*=g_\sigma|_{\QQ(C^\sigma)}$. Then, the set
$$
\{(\sigma,g_\sigma):\sigma\in \M(H), \phi_\sigma^*=g_\sigma|_{\QQ(C^\sigma)}\}\subset G
$$
Will be a subgroup, because of the compatibility condition, and thus there will be a splitting of the sequence. 

\section{Dessins with one face}

In this section we are going to prove the following.

\begin{teor}
Dessins d'enfants with one face are defined over their field of moduli.
\end{teor}

\noindent This is a slight generalization of the theorem that dessins that are trees embedded in $\P^1$, i.e. dessins with one face and genus 0, are defined over their field of moduli. This result can be found in \cite{Lando}, with a different proof than the one we are giving here. In the book, it says that the theorem was first published in \cite{couv}, but it appeared earlier in some unpublished notes in Russian by Shabat for a seminar on dessins d'enfants. I have not read these notes, but apparently they used Galois cohomology, so his ideas are probably similar to the ones presented here.

Before we start with this, we are going to look at how the Galois group acts on ramified points.

Let us call $\wh P$ the valuation of $\KK$ which restricts to $P_H$ on every subfield $\KK^H$. On each dessin, $P_H$ lies on an edge (since we chose it to be a preimage of $1/2$, but it really doesn't matter). This edge lies in some white triangle, that is, we take the black and the white vertices that are its endpoints, and out of the two faces that contain the edge, we choose the one that has a star vertex that makes the vertices of the triangle be in counterclockwise order black-white-star, so that the triangle will be white. If, for each dessin, we choose these points as canonical preimages of $0$, $1$ and $\infty$, it is clear that, for example, for 0, we will obtain points $\{Q_H\}_H$ compatible with the covers between dessins, so the family of points gives rise to a valuation in $\KK$, which we can call $\wh P_0$. Similarly, we obtain $\wh P_1$ and $\wh P_\infty$.

We are going to look at the stabilizer of $\wh P_\infty$ in $\Fd$, by the action of $\Fd$ on regular dessins. Note that, in every dessin, the monodromy action of $z=(xy)^{-1}$ is a rotation around the face which contains $P_\infty$. Therefore, this stabilizer contains $\langle z\rangle$, the (closed) subgroup generated by $z$. Let us prove that this is the whole stabilizer.

Suppose that $g\in \Fd$ fixes $\wh P_\infty$. Take a regular dessin given by $H\tl \Fd$. The automorphism of $D$ given by $g$ maps $P_H$ to some point $g(P_H)$, which must lie on one of the edges of the face that contains $P_\infty$ for it to fix $P_\infty$. Thus, $g(P_H)=z^m(P_H)$ for some $m$. Since the automorphism group of a dessin acts faithfully, it follows that $g \equiv z^m \mod H$. Therefore, $g$ is a power of $z$ in every quotient $\Fd/H$, so it belongs to the closed group that $z$ generates.

Note that this group is isomorphic to $\wh \Z$, so we can write things like $z^n$, where $n\in \wh \Z$. Also, $\wh \Z$ has a natural ring structure.

\begin{prop}\label{accionrara}
Let $\sigma \in \gal$, and let it act on $\Fd$ on the way we have defined in theorem \ref{split}. Then, there exist $g_x,g_y,g_z\in \Fd$ such that
$$
x^\sigma=g_x^{-1} x^\alpha g_x;y^\sigma=g_y^{-1} y^\alpha g_y;z^\sigma=g_z^{-1} z^\alpha g_z
$$
Where $\alpha\in \wh \Z$ is such that, for every root of unity $\xi\in \QQ$, $\xi^\sigma=\xi^\alpha$.
\end{prop}
\begin{proof}
Take some open $H\tl \Fd$. We can assume, by passing to a smaller group that is contained in $H$, that it is invariant under the Galois action, by taking
$$
\bigcap_{g\in G} H^g
$$
This is clearly a normal subgroup and it is Galois invariant. There are also a finite number of different groups in the intersection, so it is of finite index. We are going to call this subgroup $H$, and note that its field of moduli is $\Q$, and, since it is regular, its field of definition is also $\Q$.

Call $P_0$ the restriction of $\wh P_0$ to the corresponding Galois extension, and $P$ for the restriction of $\wh P$. We know that $x(P_0)=P_0$, where $x$ is understood as an automorphism of the dessin. We have that
$$
x^\sigma(P_0^{\wt \sigma})=(x(P_0))^{\wt \sigma}=P_0^{\wt \sigma}
$$
So $x^\sigma$ fixes $P_0^{\wt \sigma}$. Let $g$ be such that $P_0^{\wt \sigma}=g(P_0)$, for some $g\in \Fd/H$. This $g$ exists because the automorphism group acts transitively. Then, $(x^\sigma)^g$ fixes $P_0$, so $(x^\sigma)^g$ lies in the projection of the stabilizer of $P_0$, which is $\langle x\rangle$. Therefore, for some $\alpha$, $x^\sigma=(x^\alpha)^{g^{-1}}$.

The same thing happens for $y$ and $z$, taking the points corresponding to $y$ and $z$. Now, let see what $\alpha$ is, and that is indeed the same for $x$, $y$ and $z$.

Consider the point $P_0$ and its valuation ring $\O_{P_0}$. Since the curve is smooth, it is a DVR, so the space $\m_{P_0}/\m_{P_0}^2$ has dimension 1. Since the Belyi map on a neighborhood of $P_0$ looks like $z\longmapsto z^m$, where $m$ is the order of $x$ in $\Fd/H$, it follows that the map that $x$ induces in the cotangent space $\m_{P_0}/\m_{P_0}^2$ is given by $f \longmapsto \xi f$, where $\xi$ is a primitive $m$-th root of unity.

\noindent Take a function $f$ that has valuation $1$ in $P_0$, that is, if we call $\m_{P_0}$ the ideal of functions vanishing in $P_0$, $f$ belongs to $\m_{P_0}$ but not to $\m_{P_0}^2$. If we apply $x$, we get another function $f^x$, that is also in $\m_{P_0}$, because $x$ fixes $P_0$, and $f^x=\xi f \mod \m_P^2$.

If we apply $\sigma$ to everything, we have that $f^{\wt \sigma} \in \m_{P_0^{\wt \sigma}}\setminus \m_{P_0^{\wt \sigma}}^2$, and, if $\xi'$ is such that $(f^{\wt \sigma})^{x^\sigma}\equiv \xi' f^{\wt \sigma} \mod \m_{P_0^{\wt \sigma}}^2$, then
$$(f^{\wt \sigma})^{x^\sigma}
 +\m_{P_0^{\wt \sigma}}^2=
 \left(f^x +\m_{P_0}^2\right)^{\wt \sigma}=\left({\xi f} +\m_{P_0}^2\right)^{\wt \sigma}=\xi^\sigma f^{\wt \sigma}+\m_{P_0^{\wt \sigma}}^2
$$
So $\xi'=\xi^\sigma$. Therefore, the map that $x^\sigma$ induces on the cotangent space at $P_0^{\wt \sigma}$ is $f\mapsto \xi^\sigma f$.

Now, we know that there is some $g$ such that $x^\sigma=(x^\alpha)^{g^{-1}}$, and $P_0^{\wt \sigma}=g(P_0)$. Then,
$$
\xi^\sigma f^{\wt \sigma}+\m_{P_0^{\wt \sigma}}^2=(f^{\wt \sigma})^{x^\sigma}
 +\m_{P_0^{\wt \sigma}}^2=(f^{\wt \sigma})^{(x^\alpha)^{g^{-1}}}
 +\m_{g(P_0)}^2=(f^{\wt \sigma})^{g(x^\alpha){g^{-1}}}
 +\m_{g(P_0)}^2=(f^{\wt \sigma})^{gx^\alpha g^{-1}}
 +(\m_{P_0}^2)^{g^{-1}}=$$
$$=\left(((f^{\wt \sigma})^g)^{x^\alpha}+\m_{P_0}^2\right)^{g^{-1}} =\left(\xi^\alpha(f^{\wt \sigma})^g+\m_{P_0}^2 \right)^{g^{-1}}=\xi^\alpha f^{\wt \sigma}+\m_{g(P_0)}^2
$$
Therefore, $\xi^\sigma=\xi^\alpha$, so $\alpha$ has the value we claimed.

If we take every $H$ that is normal in $\Fd$ and is Galois invariant, we get two families $\{g_H\}_H$ and $\{\alpha_H\}$ such that $x^\sigma\equiv (x^{\alpha_H})^{g_H} \mod H$. It is clear that these families will give rise to an element of $\Fd$ and an element of $\wh \Z$, since it will follow from the fact that $x^\sigma$ is an element of $\Fd$.
\end{proof}

\noindent In view of the previous proposition, let us define the (right) $\gal$-module $\wh\Z(1)$. As a group, it is $\wh \Z$, and the action of some $\sigma\in \gal$ is multiplication by $\alpha_\sigma\in \wh\Z$, where $\alpha_\sigma$ is such that if we take a root of unity $\xi\in \QQ$, $\xi^\sigma=\xi^{\alpha_\sigma}$.

Let us look at the group $S_\infty=\{g\in G:g(\wh P_\infty)=\wh P_\infty\}$. This is similar to what we did in theorem \ref{split}. There, we considered the subgroup that fixed one point, and it turned out to be a section of $\gall\lra \gal$. Now, this group will turn out to be the semidirect product $\gal \ltimes \wh \Z(1)$, where $\wh \Z(1)$ is the subgroup of $\Fd$ generated by $z$. Since the stabilizer of $P_\infty$ in $\Fd$ is $\langle z\rangle$, $S_\infty\cap\Fd=\langle z\rangle$. Note that its image in $\gal$ is the whole group: if $g\in G$ maps $z$ to $(z^\alpha)^{g_z}$, then $gg_z^{-1}$ will belong to $S_\infty$.

Let us prove that this group is a semidirect product. Consider the group
$$
A=\{g\in G:z^g\in \langle z\rangle,\exists u\in [\Fd,\Fd]\text{ such that } x^g\in \langle x\rangle^u \}
$$
We are going to prove that $A\cap \Fd=1$ and $A\Fd=G$. In particular, it will follow that $A\cap \langle z\rangle=1$ and $A\langle z\rangle=S_\infty$.

$A$ is clearly a subgroup of $G$, so no comments there. Suppose $g\in A\cap \Fd$. Then, $z^g$ must equal $z$, since $z$ is not conjugate to any of its powers: it is not conjugate in the quotient $\Fd/\langle \langle x\rangle \rangle \cong \wh \Z$, so it can't be conjugate in $\Fd$. Therefore, $g$ commutes with $z$. Since the centralizer of $z$ in $\Fd$ is $\zz$ (see \cite{freepro}), it follows that $g\in \langle z\rangle$. But then, we must have, for some $\alpha,\beta\in \wh\Z$ and $u\in [\Fd,\Fd]$, $x^g=x^{z^\beta}=(x^\alpha)^u$. This, as before, implies that $uz^{-\beta}\in \langle x\rangle$, so $u$ must equal $x^\gamma z^\beta$. For $u$ to be in $[\Fd,\Fd]$, $\beta$ and $\gamma$ must equal $0$, and $\alpha$ must equal 1, so $g=1$.

Let us see that $A\Fd=G$. Take some $\sigma\in\gal$. We know that there is some $\alpha\in \wh\Z$ and some $g_z,g_x\in \Fd$ such that
$$
z^\sigma=(z^\alpha)^{g_z};x^\sigma=(x^\alpha)^{g_x}
$$
Take $\sigma'=\sigma g_z^{-1} z^{\beta}$. Then,
$$
z^{\sigma'}=z^\alpha;x^{\sigma'}=(x^\alpha)^{g_xg_z^{-1}z^{\beta}}=(x^\alpha)^{x^\gamma g_xg_z^{-1}z^{\beta}}
$$
It is clear that we can pick some $\gamma,\beta\in \wh\Z$ so that
$$
x^{\gamma}g_xg_z^{-1}z^{\beta}\in [\Fd,\Fd]
$$
So, for this $\beta$, $\sigma'g_z^{-1} z^{\beta}$ belongs to $A$. It follows that $A$ intersects every class in $G/\Fd=\gal$, so $A\Fd=G$, as we wanted. It follows that $A$ gives a section of
$$
1\lra \Fd \lra G\lra \gal\lra 1
$$
In particular, $A\cong \gal$. If we look at the group $S_\infty=A\langle z\rangle$, we see that it is a semidirect product and, by proposition \ref{accionrara}, it is isomorphic to the semidirect product $\gal\rtimes \ZZZ$.

We are going to use a lemma to prove that every dessin with one face is defined over its field of moduli. We will prove it at the end of the section.

\begin{lema}\label{cohom}
Let $U$ be an open subgroup of $\Gal(\QQ/\Q)\ltimes \wh\Z(1)$. Then, there exist a subgroup $T<U$, a number field $K$ and an integer $m$ such that
\begin{itemize}
\item $T\cap m\ZZZ=1$
\item $T(m\ZZZ)=U$
\item $T\cong\Gal(\QQ/K)$
\end{itemize}
\end{lema}

\noindent Let us prove the main theorem, and leave this lemma for later.

\begin{teor}
Every dessin with one face is defined over its field of moduli.
\end{teor}
\begin{proof}
Let the dessin be given by an open subgroup $H<\Fd$. The fact that it has one face means that $z$ acts transitively on the edges, i.e. it acts transitively on the cosets of $H$, so $\langle z\rangle H=\Fd$.

Now, as we know, the field of moduli of the dessin is the fixed field of $\M(H)$, the image in $\gal$ of $N_G(\Fd)$. Consider now the group $A$ as previously defined. Take $U=N_G(H)\cap S_\infty=N_G(H)\cap A\langle z\rangle$.

Since $U$ is open, applying lemma \ref{cohom}, we see that $U=T\langle z^m\rangle$ for some $m$, where $T\cap \zz=1$ and $T\zz=U$. Since $T\cap \zz=1$ and $T\cap \Fd\subset A\zz\cap \Fd=\zz$, it follows that $T\cap \Fd=1$.

Also, let us prove that the image of $T$ in $\gal$ is $\M(H)$. Suppose $\sigma\in \M(H)$. Then, we can take its unique representative $\sigma'\in A$, and for some $g\in\Fd$, $\sigma'g$ will be in $N_G(H)$. Now, every coset of $N_{\Fd}(H)$ is $z^lN_{\Fd}(H)$ for some $l$; so $\sigma' z^l$ will also be in $N_G(H)$. Since $\sigma'z^l\in A\zz$, it is in $N_G(H)\cap A\zz=U$. Therefore, the image of $U$ in $\gal$ is $\M(H)$. Since the images of $U$ and $T$ in $\gal$ are the same (because $U=T\zz$), the image of $T$ is also $\M(H)$. So $T\Fd=N_G(H)\Fd$, and $TN_{\Fd}(H)=N_G(H)$.

So $T$ satisfies the following:
$$
T\cap \Fd=1;T\Fd=N_G(H)\Fd
$$
Therefore, if we take $\wh H=TH$, we will have that
$$
\wh H \cap \Fd=H;\wh H N_{\Fd}(H)=N_G(H)
$$
And this is precisely the characterization of $H$ being definable over its field of moduli, by proposition \ref{model}.
\end{proof}

\noindent Let us prove lemma \ref{cohom}. To do it, we are going to use some cohomology of groups, such as the relationship between $H^2$ and extensions, Hilbert's Theorem 90 and the long exact sequence in cohomology. For a reference about these facts, see \cite{serre}. Suppose we have some open subgroup $U$ in $\gal \ltimes \wh\Z(1)$. Its intersection with $\wh\Z(1)$ is an open subgroup, so it equals $m\ZZZ$ for some $m$. Let $K$ be such that the image of $U$ in $\gal$ equals $\Gal(\QQ/K)$. Then, we have an exact sequence
$$
1\lra m\ZZZ\lra U\lra \Gal(\QQ/K)\lra 1
$$
We want to prove that it is split. If it is split, then the section $T$ of $\Gal(\QQ/K)$ will be the group we are looking for and the result will be proven.

\begin{lema}
For every $l\in \N$, the following sequence splits:
$$
1\lra \frac{m\ZZZ}{ml\ZZZ} \lra \frac{U}{ml\ZZZ} \lra \Gal(\QQ/K)\lra 1
$$
\end{lema}
\begin{proof}
The first group in the sequence is the cyclic group $C_l$, with the same module structure as the group $\mu_l$ of $l$-th roots of unity. Let $\phi\in H^2(\Gal(\QQ/K),\mu_l)$ be the cocycle associated to this extension. Consider the sequence
$$
1\lra \frac{\ZZZ}{ml\ZZZ} \lra \frac{\Gal(\QQ/K)\ltimes \ZZZ}{ml\ZZZ} \lra \Gal(\QQ/K)\lra 1
$$
It is clearly split because a section is just restricting the section of $\gal\ltimes \ZZZ$ to $\Gal(\QQ/K)$. Also, the inclusion $i:\frac{m\ZZZ}{ml\ZZZ}\lra \frac{\ZZZ}{ml\ZZZ}\cong \mu_{ml}$ induces a map $i_*:H^2(\GK,\mu_l)\lra H^2(\GK,\mu_{ml})$, and $i_*(\phi)$ is the cocycle corresponding to this second extension. Since the extension is split, $i_*(\phi)=1$ in the cohomology group.

\noindent Now, we claim that $i_*$ is injective. From this, it will follow that the cocycle for the first extension was 1, and therefore the extension also splits. Take the exact sequence
$$
1\lra \mu_{l} \lra \QQ^\times \overset{\cdot^{l}}{\lra}\QQ^\times \lra 1
$$
For this exact sequence, take the long exact sequence in cohomology, which goes
$$
\cdots \lra H^1(\GK,\QQ^\times)\lra H^2(\GK,\mu_{l})\lra H^2(\GK,\QQ^\times)\lra \cdots
$$
Hilbert's Theorem 90 states that for any field $K$, $H^1(\Gal(\o{K}/K),\overline K^\times)=0$. Therefore, the map that goes from $H^2(\GK,\mu_{ml})$ to $ H^2(\GK,\QQ^\times)$ is injective. In particular, since it is the composition of the following two maps:
$$H^2(\GK,\mu_{l})\lra H^2(\GK,\mu_{ml})\lra H^2(\GK,\QQ^\times)$$
The first map is also injective. This is what we wanted to prove, since this map is $i_*$. Therefore the cocycle $\phi$ is 0 in $H^2(\GK,\mu_l)$.
\end{proof}
\noindent We have proven that for every $l$, the following sequence splits:
$$
1\lra \frac{m\ZZZ}{ml\ZZZ} \lra \frac{U}{ml\ZZZ} \lra \Gal(\QQ/K)\lra 1
$$
Let $s_l:\GK\lra  \frac{U}{ml\ZZZ}$ be a section. We need to prove that there is a section $s:\GK \lra U$. Of course, if the image of every other section of the form $s_{lk}$ in the quotient $\frac{U}{ml\ZZZ}$ was $s_l$, our job would be done. So what we need to do is construct such compatible sections.

Take a section $s_{lk}:\GK\lra \frac{U}{mlk\ZZZ}$, and let $\o{s_{lk}}$ be its image in the quotient $\frac{U}{ml\ZZZ}$, i.e. if $\pi$ is the projection, $\o{s_{lk}}=\pi\circ s_{lk}$. $\o{s_{lk}}$ is also a section of the sequence for $l$, so let us see the relation between $s_l$ and $\o{s_{lk}}$. For every $\sigma\in \GK$, $s_l(\sigma)$ and $\o{s_{lk}}(\sigma)$ are in the same coset of $\mu_l$, so there is some $\alpha(\sigma)\in \mu_l$ such that
$$
\o{s_{lk}}(\sigma)=s_l(\sigma)\alpha(\sigma)
$$
Now, both $s_l$ and $\o{s_{lk}}$ are group homomorphisms, so, for $\sigma,\tau\in \GK$,
$$
s_l(\sigma\tau)\alpha(\sigma\tau)=\o{s_{lk}}(\sigma\tau)=\o{s_{lk}}(\sigma)\o{s_{lk}}(\tau)=s_l(\sigma)\alpha(\sigma)s_l(\tau)\alpha(\tau)=s_l(\sigma)s_l(\tau)\alpha(\sigma)^{s_l(\tau)}\alpha(\tau)=s_l(\sigma\tau)\alpha(\sigma)^{s_l(\tau)}\alpha(\tau)
$$
Since $\GK$ acts as automorphisms of $\mu_l$, $\alpha(\sigma)^{s_l(\tau)}=\alpha(\sigma)^\tau$, and what we have is
$$
\alpha(\sigma\tau)=\alpha(\sigma)^\tau\alpha(\tau)
$$
This means that $\alpha$ is a cocycle in $H^1(\GK,\mu_l)$. This cohomology group is not $0$, but $H^1(\GK,\QQ^\times)$ is, by Theorem 90. Therefore, in $H^1(\GK,\QQ^\times)$, the cocycle $\alpha$ is indeed $0$; so there exists some $\beta'\in \QQ$ such that $\alpha(\sigma)=\beta^{\prime \sigma}/\beta'$. Now, take some $\beta\in \QQ$ such that $\beta^{-k}=\beta'$, and define
$$s_{lk}'(\sigma)=s_{lk}(\sigma)\beta^\sigma/\beta$$
Note that $\beta^\sigma/\beta$ lies in $\mu_{lk}$. Let us check two things: that $s_{lk}'$ is another section, and that its projection $\o{s_{lk}'}$ equals $s_l$.

Since $\beta^\sigma/\beta\in \mu_{lk}$, it is clear that $s_{lk}'=s_{lk}$ modulo $m\ZZZ$, so there is no problem there. We only need to check that it is a homomorphism. Let $\sigma,\tau\in \GK$. Then
$$
s_{lk}'(\sigma)s_{lk}'(\tau)=s_{lk}(\sigma)\frac{\beta^\sigma}{\beta} s_{lk}(\tau)\frac{\beta^\tau}{\tau}=s_{lk}(\sigma)s_{lk}(\tau)\left(\frac{\beta^\sigma}{\beta}\right)^\tau \frac{\beta^\tau}{\beta}=s_{lk}(\sigma)s_{lk}(\tau)\frac{\beta^{\sigma\tau}}{\beta^\tau} \frac{\beta^\tau}{\beta}=
s_{lk}(\sigma\tau)\frac{\beta^{\sigma\tau}}{\beta}=s_{lk}'(\sigma\tau)
$$
For the second part, let us check that $\o{s_{lk}'}=s_l$. Note that the projection $\frac{m\ZZZ}{mlk\ZZZ}\lra \frac{m\ZZZ}{ml\ZZZ}$ can be given by $\xi\mapsto \xi^k$.
$$
\o{s_{lk}'}(\sigma)=\o{s_{lk}}(\sigma)\left(\frac{\beta^\sigma}{\beta}\right)^k=\o{s_{lk}}(\sigma)\alpha(\sigma)^{-1}=s_l(\sigma)
$$
So $s_{lk}'$ is a section that projects to $s_l$. We can continue doing this for bigger $k$'s, and we will obtain a compatible sequence of sections, which will yield a section of
$$
1\lra m\ZZZ\lra U\lra \GK\lra 1
$$
As we wanted in the first place. This concludes the proof of lemma \ref{cohom}.

\section{A dessin that is not defined over its field of moduli}\label{notdefined}

We are going to give an example of a dessin that isn't defined over its field of moduli, following \cite{rabbit}.

The easiest field where we can prove that a dessin isn't defined is $\R$. Let us see how complex conjugation acts on $\Fd$ when we embed $\gal$ in $\Aut(\Fd)$. Conjugation is a continuous map: therefore, it maps the fundamental group of $\P^1\setminus\{0,1,\infty\}$ to itself. By this homomorphism, $x$, the path that goes around $0$ is mapped to $x^{-1}$ and $y$ is mapped to $y^{-1}$. 

Take a Belyi pair $(C,f)$, with a base point $P$. Let $P^x$ be the monodromy action of $x\in F_2=\pi_1(\P^1\setminus\{0,1,\infty\})$ on $P$. Then, if we look at the conjugate Belyi pair $(\o C,\o f)$ (conjugate in the sense of complex conjugation), the point $\o P^{\o x}$ will be determined, as usual, by the lifting of $x$ to $\o C$. Now, we just need to note that $\o C$ is homeomorphic to $C$, that $\o f$ equals $f$ composed with conjugation in $\P^1$. Thus, we have the following commutative diagram, if we call conjugation $c$, and we use $\wt{x^{-1}}$ for the lifting of $x^{-1}$ to $C$. 
\begin{center}\hspace{0.5 cm}
\begindc{\commdiag}[4]
\obj(-5,10)[a]{$[0,1]$}
\obj(10,0)[b]{$\P^1$}
\obj(10,10)[c]{$\P^1$}
\obj(10,20)[d]{$C$}
\mor{a}{b}{$x$}
\mor{a}{c}{$x^{-1}$}
\mor{a}{d}{$\wt{x^{-1}}$}
\mor{d}{c}{$f$}
\mor{c}{b}{$c$}
\cmor((12,20)(13,20)(14,19)(17,10)(14,1)(13,0)(12,0)) \pleft(20,10){$\o{f}$}
\enddc
\end{center}
\noindent Therefore, the path $x$ will lift to the same path in $\o{C}$ that $x^{-1}$ lifts to in $C$. Since, if we identify $C$ with $\o C$, we have that $\o f=c\circ f$, then $x=c\circ f \circ \wt{x^{-1}}=\o{f} \circ \wt{x^{-1}}$. Therefore, $\o{P}^{\o{x}}=P^{x^{-1}}$, if we identify $C$ with $\o{C}$.
 
In conclusion, the automorphism induced by $c$, the conjugation, in $\Fd$ is given by
$$
x\longmapsto x^{-1}$$ $$
y\longmapsto y^{-1}$$
Take a dessin given by an open subgroup $H<\Fd$. Its  moduli field is real if it is fixed by the conjugation $c$. Suppose we number the points on the dessin $P_1,\ldots ,P_n$. Let $s_x$ and $s_y$ be the permutations such that $P_i^x=P_{s_x(i)}$. Then, we have that
$$
(P_i^c)^x=(P_i^{x^{c^{-1}}})^c=(P_i^{x^{-1}})^c=P_{i^{s_x^{-1}}}^c
$$
So the permutation $x$ induces on the conjugate dessin is $s_x^{-1}$, and similarly for $y$.

Suppose there is an isomorphism $\phi:(C,f)\lra (\o C,\o f)$. This isomorphism is determined by the image of a point, so it is equivalent to a permutation $\omega\in S_n$ defined in the following way:
$$
\phi(P_i)=P_{i^\omega}^c
$$For $\phi$ to be an isomorphism, it is required that it commute with the action of $x$, that is, $\phi(P_i^x)=\phi(P_i)^x$. This, in terms of $\omega$, is $x\omega=\omega x^{-1}$. Therefore, $x^{-1}=x^{\omega}$.

Now, suppose a dessin has a real model $\wh H$. Then the image of $\wh H$ in $\gal$ is the Galois group of a real field, so it contains $c$. Therefore, there will be some $c'\in \wh H$ such that its image in $\gal$ is $c$. Since $c$ has order 2, $c^{\prime 2}$ will lie in $\Fd$ Also, $c^{\prime 2}\in\wh H$, and since $\wh H\cap \Fd=H$, $c^{\prime 2}$ must lie in $H$.

Suppose $c'=(c,g)$. Then, $c^{\prime 2}\in H$ means that
$$
c^{\prime 2}=(c,g)(c,g)=(c^2,g^cg)=(1,g^cg)\in H
$$
So we must have that $g^cg\in H$. In terms of permutations, it means that $g^cg$ acts on the edges of $H$ as the identity. Let us see what this means: $g$ maps $H^c$ to $H$, so it maps $\KK^{H^c}$ to $\KK^H$. Therefore, the corresponding morphism of dessins maps $(C,f)$ to $(\o C,\o f)$, so we get a permutation $\omega\in S_n$ as before, given by
$$
\phi(P_i)=P_{i^\omega}^c
$$
Now, $g^c$ maps $(\o C,\o f)$ to $(C,f)$, and it is given by
$$
\phi^c(P_i^c)=(\phi(P_i))^c=P_{i^\omega}^{c^2}=P_{i^\omega}
$$
So, for us to have that $g^cg\in H$, we must have that $\phi^c\circ \phi=1$, so, for every point,
$$
P_i=\phi^c(\phi(P_i))=\phi^c(P_{i^\omega}^c)=P_{i^{\omega^2}}
$$
In other words, $\omega$ must have order 2.

Therefore, we have the following, which is theorem 2 in \cite{rabbit}.

\begin{prop}
Let a dessin d'enfant be given by its cartographic group $\langle x,y\rangle<S_n$. Then, the dessin's field of moduli is real if and only if there exists some $\omega\in S_n$ such that $x^\omega=x^{-1}$ and $y^{\omega}=y^{-1}$. The dessin can be defined over a real field if and only if this $\omega$ can be chosen of order 2.
\end{prop}

\noindent It is relatively easy then to give dessins that are invariant under conjugation but are not defined over the reals. For example, take the following dessin on a curve of genus 1. The opposing sides of the rectangle are identified.

\begin{center}
\begin{tikzpicture}[line cap=round,line join=round,>=triangle 45,x=0.5cm,y=0.5cm]
\clip(-1,-1) rectangle (25,7);
\draw (0,0)--(24,0)--(24,6)--(0,6)--(0,0);
\draw (6,0)--(6,6);
\draw (12,0)--(12,6);
\draw (18,0)--(18,6);
\draw (0,3)--(2,3);
\draw (6,3)--(8,3);
\draw (12,3)--(14,3);
\draw (18,3)--(20,3);

\draw (9,0)--(9,2);
\draw (3,4)--(3,6);
\draw (15,4)--(15,6);
\draw (21,0)--(21,2);
\begin{scriptsize}
\puntonegro{(0,0)} 
\puntonegro{(0,6)}
\puntonegro{(6,0)} 
\puntonegro{(6,6)}
\puntonegro{(12,0)} 
\puntonegro{(12,6)}
\puntonegro{(18,0)} 
\puntonegro{(18,6)}
\puntonegro{(24,0)} 
\puntonegro{(24,6)}
\puntonegro{(2,3)}
\puntonegro{(8,3)}
\puntonegro{(14,3)}
\puntonegro{(20,3)}
\puntonegro{(3,4)}
\puntonegro{(9,2)}
\puntonegro{(15,4)}
\puntonegro{(21,2)}
\puntoblanco{(0,3)}
\puntoblanco{(6,3)}
\puntoblanco{(12,3)}
\puntoblanco{(18,3)}
\puntoblanco{(24,3)}
\puntoblanco{(3,0)}
\puntoblanco{(9,0)}
\puntoblanco{(15,0)}
\puntoblanco{(21,0)}
\puntoblanco{(3,6)}
\puntoblanco{(9,6)}
\puntoblanco{(15,6)}
\puntoblanco{(21,6)}

\draw (0,1.5) node[anchor=east] {1};
\draw (0,4.5) node[anchor=east] {3};
\draw (6,1.5) node[anchor=east] {9};
\draw (6,4.5) node[anchor=east] {7};
\draw (12,1.5) node[anchor=east] {13};
\draw (12,4.5) node[anchor=east] {15};
\draw (18,1.5) node[anchor=east] {21};
\draw (18,4.5) node[anchor=east] {19};
\draw (24,1.5) node[anchor=west] {1};
\draw (24,4.5) node[anchor=west] {3};
\draw (-1,0.5) node[anchor=east] {9};
\draw (-1,-0.5) node[anchor=east] {3};
\draw (1,3) node[anchor=south] {2};
\draw (7,3) node[anchor=south] {8};
\draw (13,3) node[anchor=south] {14};
\draw (19,3) node[anchor=south] {20};
\draw (1.5,0) node[anchor=north] {4};
\draw (4.5,0) node[anchor=north] {6};
\draw (7.5,0) node[anchor=north] {10};
\draw (10.5,0) node[anchor=north] {12};
\draw (13.5,0) node[anchor=north] {16};
\draw (16.5,0) node[anchor=north] {18};
\draw (19.5,0) node[anchor=north] {22};
\draw (22.5,0) node[anchor=north] {24};
\draw (1.5,6) node[anchor=south] {4};
\draw (4.5,6) node[anchor=south] {6};
\draw (7.5,6) node[anchor=south] {10};
\draw (10.5,6) node[anchor=south] {12};
\draw (13.5,6) node[anchor=south] {16};
\draw (16.5,6) node[anchor=south] {18};
\draw (19.5,6) node[anchor=south] {22};
\draw (22.5,6) node[anchor=south] {24};
\draw (3,5) node[anchor=east] {5};
\draw (9,1) node[anchor=east] {11};
\draw (15,5) node[anchor=east] {17};
\draw (21,1) node[anchor=east] {23};
\end{scriptsize}
\end{tikzpicture}
\end{center}

\noindent This dessin has the following monodromy action, as it is seen from the picture:
\begin{align*}
x &\longmapsto (4,1,24,3)(9,6,7,10)(16,13,12,15)(21,18,19,22)\\
y &\longmapsto (1,2,3)(4,5,6)(9,8,7)(12,11,10)(13,14,15)(16,17,18)(21,20,19)(24,23,22)
\end{align*}
\noindent It clearly has an automorphism of order 2, given by sending the edges numbered $i$ to $i+12$. Also, we have that, if we take
$$
\omega=(1,7,13,19)(2,8,14,20)(3,9,15,21)(4,10,16,22)(5,11,17,23)(6,12,18,24)
$$
Conjugation by $\omega$ sends $x$ to $x^{-1}$ and $y$ to $y^{-1}$. In particular, $c$ leaves the dessin invariant, so its field of moduli is real. However, $\omega$ has order 4, and if there is another $\omega'$ such that $x^\omega=x^{-1}$ and $y^\omega=y^{-1}$, then
$$
x^{\omega}=x^{-1}=x^{\omega'}
$$
And the same for $y$. Therefore, $\omega\omega^{\prime-1}$ is an automorphism of the dessin. So any other $\omega'$ is of the form $g\phi$, where $g$ is an automorphism. The only nontrivial automorphism is the one sending $i$ to $i+12$, and we can check that this composed with $\omega$ equals $\omega^{-1}$.

Therefore, we have an example of a dessin that is not defined over its field of moduli.
\chapter{A regular dessin whose field of moduli is $\Q(\sqrt[3] 2)$}

In this part we are going to give an example of a regular dessin d'enfant whose field of moduli is $\Q(\rd)$. This is an example of a regular dessin whose field of moduli is not an abelian extension of $\Q$, thus answering a question in \cite{cojowo}.

The way we are going to do this is as follows: we are going to explicitly construct a dessin on the elliptic curve $C=V(Y^2-X(X-1)(X-\rd))$. Then, we are going to consider its regular cover and we are going to prove that when a Galois automorphism $\sigma\in \gal$ doesn't fix $\rd$, it doesn't fix this regular cover either. This cover has genus 145. Then, we are going to consider a dessin which is covered by this one, so it is somewhat simpler (it has genus 61), and it also has the same field of moduli. We are also going to give an explicit description for the field of functions of this last dessin. Finally, we are going to prove that the underlying curve for this dessin also has the same field of moduli.

There is another example of a regular dessin with non-abelian cubic field of moduli. It is the regular cover of a dessin described by Shabat and Voevodsky in \cite{voed} and by Malle in \cite{malle}. In the last section, we will expand on this example.

\section{A dessin $D_0$ over an elliptic curve}\label{dibujitos}

Let us define a dessin on the elliptic curve given by the equation $Y^2=X(X-1)(X-\sqrt[3]{2})$. The dessin we are going to give is also described in \cite{abc}.

Consider the map on $C$ given by $(X,Y)\longmapsto X$, or, in homogeneous coordinates, $(X:Y:Z)\mapsto (X:Z)$. This map is ramified at four points: $(0:0:1)$, $(1:0:1)$, $(\rd:0:1)$ and the point at infinity, $(0:1:0)$. Their images are $\{0,1,\rd,\infty \}\subset \P^1$. In order to have a Belyi map for $C$, we need a map which is only ramified over 3 points. To do this, we are going to mirror the proof we gave for Belyi's theorem, in section \ref{belyi1}: we compose with the map $X\longmapsto X^3$, which ramifies only over $0$ and $\infty$, and maps the points $\{0,1,\rd,\infty\}$ to $\{0,1,2,\infty\}$. Then, we compose with the map $X\longmapsto (X-1)^2$, which ramifies over $0$ and $\infty$ and maps $\{0,1,2,\infty\}$ to $\{0,1,\infty\}$. Finally we will compose with $X\longmapsto 1-X$ because it is more pleasant to have $0$ mapped to $0$ and $1$ mapped to $1$.

We now have a Belyi map $f:C\longmapsto \P^1$, which is given by the composition of all the maps we have mentioned above: it is given by $f(X,Y)=1-(X^3-1)^2$, or equivalently, by $f(X:Y:Z)=(Z^6-(X^3-Z^3)^2:Z^6)$. In order to be able to draw the dessin d'enfant corresponding to $f$, we need to know $f^{-1}([0,1])$. Let us do this step by step. We will use the usual notation for dessins d'enfants: a black dot will represent a point in the preimage of $0$ and a white point will represent one in the preimage of $1$. The starting picture, just the segment $[0,1]$, would look something like this:
\begin{center}
\begin{tikzpicture}[line cap=round,line join=round,>=triangle 45,x=2.0cm,y=2.0cm]
\clip(-1,-0.2) rectangle (2,0.4);
\draw (0,0)-- (1,0);
\begin{scriptsize}
\fill [color=black] (0,0) circle (1.5pt);
\draw [color=black] (1,0) circle (1.5pt);
\fill[color=white] (1,0) circle (1.3pt);
\draw (0,0.24) node[anchor=north west] {$$ 0 $$};
\draw (1,0.24) node[anchor=north west] {$$1$$};
\end{scriptsize}
\end{tikzpicture}
\end{center}
\noindent Now, the map $X\longmapsto 1-X$ just flips the segment, so its preimage is the same picture with the black and white dots interchanged. The map $X\longmapsto X^2$ looks like this:

\begin{tabular}{ccccc}
\begin{tikzpicture}[line cap=round,line join=round,>=triangle 45,x=1.5cm,y=1.5cm]
\clip(-1.5,-0.3) rectangle (1.5,0.3);
\draw (-1,0)-- (1,0);
\begin{scriptsize}
\fill [color=black] (-1,0) circle (1.5pt);
\fill [color=black] (1,0) circle (1.5pt);
\draw [color=black] (0,0) circle (1.5pt);
\fill[color=white] (0,0) circle (1.3pt);
\draw (0,0.24) node[anchor=north west] {$$ 0 $$};
\draw (1,0.24) node[anchor=north west] {$$1$$};
\draw (-1,0.24) node[anchor=north west] {$$-1$$};
\end{scriptsize}
\end{tikzpicture}
& 
\begin{tikzpicture}[line cap=round,line join=round,x=1.0cm,y=1.0cm]
\clip(-0.6,-0.3) rectangle (0.6,0.5);
\draw[|->] (-0.5,0) -- (0.5,0);
\draw (0,0) node[anchor=south] {\small{$X^2$}};
\end{tikzpicture}
&
\begin{tikzpicture}[line cap=round,line join=round,>=triangle 45,x=1.5cm,y=1.5cm]
\clip(-1,-0.3) rectangle (2,0.5);
\draw (0,0)-- (1,0);
\begin{scriptsize}
\fill [color=black] (1,0) circle (1.5pt);
\draw [color=black] (0,0) circle (1.5pt);
\fill[color=white] (0,0) circle (1.3pt);
\draw (0,0.24) node[anchor=north west] {$$ 0 $$};
\draw (1,0.24) node[anchor=north west] {$$1$$};
\end{scriptsize}
\end{tikzpicture}
&
\begin{tikzpicture}[line cap=round,line join=round,x=1.0cm,y=1.0cm]
\clip(-0.6,-0.3) rectangle (0.6,0.5);
\draw[|->] (-0.5,0) -- (0.5,0);
\draw (0,0) node[anchor=south] {\small{$1-X$}};
\end{tikzpicture}
&
\begin{tikzpicture}[line cap=round,line join=round,>=triangle 45,x=1.5cm,y=1.5cm]
\clip(-1,-0.3) rectangle (2,0.5);
\draw (0,0)-- (1,0);
\begin{scriptsize}
\fill [color=black] (0,0) circle (1.5pt);
\draw [color=black] (1,0) circle (1.5pt);
\fill[color=white] (1,0) circle (1.3pt);
\draw (0,0.24) node[anchor=north west] {$$ 0 $$};
\draw (1,0.24) node[anchor=north west] {$$1$$};
\end{scriptsize}
\end{tikzpicture}
\end{tabular}

\noindent Now, we have to compose with the map $X^3$. This map is ramified three times at $0$. Here, $\xi$ is the cubic root of unity $\frac{-1+i\sqrt 3}{2}$:
\begin{center}
\begin{tabular}{ccc}
\begin{tikzpicture}[line cap=round,line join=round,>=triangle 45,x=1.0cm,y=1.0cm]
\clip(-2.5,-2.5) rectangle (2.7,2.5);
\draw (-1,-1.73)-- (0,0);
\draw (0,0)-- (2,0);
\draw (0,0)-- (-1,1.73);
\begin{scriptsize}
\draw (2,0) node[anchor=south west]  {$\sqrt[3]{2}$};
\draw (-1,1.98) node[anchor=north west] {$\xi\sqrt[3]{2}$};
\draw (-1,-1.5) node[anchor=north west] {$ \xi^2\sqrt[3]{2} $};
\draw (1,0) node[anchor=south west] {$1 $};
\draw (-0.5,1.1) node[anchor=north west] {$ \xi $};
\draw (-0.5,-0.62) node[anchor=north west] {$ \xi^2 $};
\draw (0,0) node[anchor=south west] {$ 0 $};
\fill [color=black] (0,0) circle (1.5pt);
\draw [color=black] (1,0) circle (1.5pt);
\fill [color=white] (1,0) circle (1.3pt);
\fill [color=black] (2,0) circle (1.5pt);
\draw [color=black] (-0.5,0.87) circle (1.5pt);
\fill [color=white] (-0.5,0.87) circle (1.3pt);
\fill [color=white] (-0.5,-0.87) circle (1.3pt);
\draw [color=black] (-0.5,-0.87) circle (1.5pt);
\fill [color=black] (-1,1.73) circle (1.5pt);
\fill [color=black] (-1,-1.73) circle (1.5pt);
\end{scriptsize}
\end{tikzpicture}
& 
\begin{tikzpicture}[line cap=round,line join=round,x=1.0cm,y=1.0cm]
\clip(-0.6,-2.5) rectangle (0.6,2.5);
\draw[|->] (-0.5,0) -- (0.5,0);
\draw (0,0) node[anchor=south] {$X^3$};
\end{tikzpicture}
&
\begin{tikzpicture}[line cap=round,line join=round,>=triangle 45,x=1.0cm,y=1.0cm]
\clip(-1.5,-2.5) rectangle (1.5,2.5);
\draw (-1,0)-- (1,0);
\fill [color=black] (-1,0) circle (1.5pt);
\fill [color=black] (1,0) circle (1.5pt);
\draw [color=black] (0,0) circle (1.5pt);
\fill[color=white] (0,0) circle (1.3pt);
\begin{scriptsize}
\draw (0,0) node[anchor=south west] {$ 1 $};
\draw (1,0) node[anchor=south west] {$2$};
\draw (-1,0) node[anchor=south west] {$0$};
\end{scriptsize}
\end{tikzpicture}
\end{tabular}

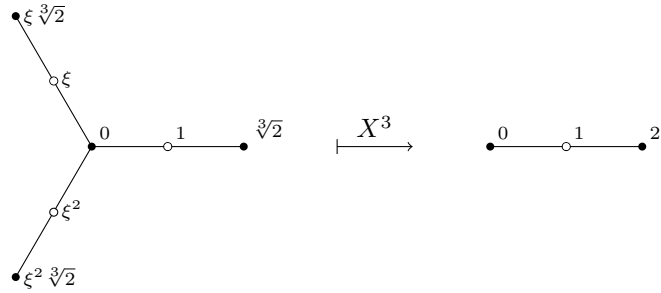
\captionof{figure}{The dessin corresponding to $X\longmapsto 1-(X^3-1)^2$.}\label{des6}
\end{center}
\noindent Finally, we need the preimage of this graph on the curve $C$ by the projection on the first coordinate to know what the dessin looks like. The projection has degree 2 and it is ramified over $0$, $1$, $\rd$ and $\infty$. Since it is ramified at $\infty$, it has only one face. We would like to know the ordering of the edges around that face. To do this, call the canonical generators of the fundamental group of $\P^1\setminus \{0,1,\infty\}$ $x $ and $y$. Recall that the monodromy action on the edges of $x$ consists on rotating counterclockwise around black vertices, and the action of $y$ consists on rotating counterclockwise around white vertices. Therefore, $xy$ will consist on rotating clockwise around a face (we adopt the convention that the monodromy group acts on the right, so $xy$ means $x$ then $y$). Therefore, if we number the edges in figure \ref{des6} like so:
\begin{center}
\begin{tabular}{cc|cc}
Endpoints & Label & Endpoints & Label
\\
\hline
$0-1$ & 1 & $1-\rd$ & 2 \\
$0-\xi$ & 3 & $1-\xi\rd$ & 4 \\
$0-\xi^2$ & 5 & $1-\xi^2\rd$ & 6 \\
\end{tabular}
\end{center}
\noindent The action of $x$ corresponds to the permutation $(135)$, that of $y$ corresponds to $(12)(34)(56)$, and $xy$ corresponds to $(143652)$. Now, the dessin we are looking at has a monodromy action that factors through this one. We can number its edges $1,\ldots,6,1',\ldots ,6'$. Then, since $xy$ has order 12 because the action is ramified at $\infty$, it must correspond, with appropriate labeling, to
$$
z\longmapsto(1436521'4'3'6'5'2')
$$
So, the face of the dessin will look like figure \ref{dodec}. The labels for the vertices mark their image by the map $(X,Y)\longmapsto X$.
\begin{center}
\begin{tikzpicture}[line cap=round,line join=round,>=triangle 45,x=1.5cm,y=1.5cm]
\clip(-1.7,-1.5) rectangle (1.5,1.5);
\draw (1,0)-- (0.97,0.26);
\draw (0.97,0.26)-- (0.87,0.5);
\draw (0.87,0.5)-- (0.71,0.71);
\draw (0.71,0.71)-- (0.5,0.87);
\draw (0.5,0.87)-- (0.26,0.97);
\draw (0.26,0.97)-- (0,1);
\draw (0,1)-- (-0.26,0.97);
\draw (-0.26,0.97)-- (-0.5,0.87);
\draw (-0.5,0.87)-- (-0.71,0.71);
\draw (-0.71,0.71)-- (-0.87,0.5);
\draw (-0.87,0.5)-- (-0.97,0.26);
\draw (-0.97,0.26)-- (-1,0);
\draw (-1,0)-- (-0.97,-0.26);
\draw (-0.97,-0.26)-- (-0.87,-0.5);
\draw (-0.87,-0.5)-- (-0.71,-0.71);
\draw (-0.71,-0.71)-- (-0.5,-0.87);
\draw (-0.5,-0.87)-- (-0.26,-0.97);
\draw (-0.26,-0.97)-- (0,-1);
\draw (0,-1)-- (0.26,-0.97);
\draw (0.26,-0.97)-- (0.5,-0.87);
\draw (0.5,-0.87)-- (0.71,-0.71);
\draw (0.71,-0.71)-- (0.87,-0.5);
\draw (0.87,-0.5)-- (0.97,-0.26);
\draw (0.97,-0.26)-- (1,0);
\begin{scriptsize}
\fill [color=black] (1,0) circle (1.5pt);
\draw (1,0) node[anchor= west] {$0$};
\puntonegro{(0.87,0.5)}
\draw (0.87,0.5) node[anchor=south west] {$\xi^2\rd$};
\fill [color=black] (0.5,0.87) circle (1.5pt);
\draw (0.5,0.87) node[anchor=south west] {$0$};
\puntonegro{(0,1)}
\draw (0,1) node[anchor= south] {$\xi\rd$};
\fill [color=black] (-0.5,0.87) circle (1.5pt);
\draw (-0.5,0.87) node[anchor=south east] {$0$};
\puntonegro{(-0.87,0.5)}
\draw (-0.87,0.5) node[anchor=south east] {$\rd$};
\fill [color=black] (-1,0) circle (1.5pt);
\draw (-1,0) node[anchor= east] {$0$};
\puntonegro{(-0.87,-0.5)}
\draw (-0.87,-0.5) node[anchor=north east] {$\xi^2\rd$};
\fill [color=black] (-0.5,-0.87) circle (1.5pt);
\draw (-0.5,-0.87) node[anchor=north east] {$0$};
\puntonegro{(0,-1)}
\draw (0,-1) node[anchor= north] {$\xi\rd$};
\fill [color=black] (0.5,-0.87) circle (1.5pt);
\draw (0.5,-0.87) node[anchor=north west] {$0$};
\puntonegro{(0.87,-0.5)}
\draw (0.87,-0.5) node[anchor=north west] {$\rd$};
\draw (0,0)-- ++(-1.0pt,-1.0pt) -- ++(2.0pt,2.0pt) ++(-2.0pt,0) -- ++(2.0pt,-2.0pt) ++ (-1.0pt,-0.4pt) -- ++(0,2.8pt) ++(-1.4pt,-1.4pt)-- ++(2.8pt,0);
\draw (0,0) node[anchor=south west] {$\infty$};
\puntoblanco{(0.27,0.96)} \draw (0.27,0.96) node[anchor=south] {$\xi$};
\puntoblanco{(0.71,0.71)} \draw (0.71,0.71) node[anchor=south west] {$\xi^2$};
\puntoblanco{(0.96,0.27)} \draw (0.96,0.27) node[anchor=west] {$\xi^2$};
\puntoblanco{(-0.27,0.96)} \draw (-0.27,0.96) node[anchor=south] {$\xi$};
\puntoblanco{(-0.71,0.71)} \draw (-0.71,0.71) node[anchor=south east] {$1$};
\puntoblanco{(-0.96,0.27)} \draw (-0.96,0.27) node[anchor=east] {$1$};
\puntoblanco{(-0.27,-0.96)} \draw (-0.27,-0.96) node[anchor=north] {$\xi$};
\puntoblanco{(-0.71,-0.71)} \draw (-0.71,-0.71) node[anchor=north east] {$\xi^2$};
\puntoblanco{(-0.96,-0.27)} \draw (-0.96,-0.27) node[anchor=east] {$\xi^2$};
\puntoblanco{(0.27,-0.96)} \draw (0.27,-0.96) node[anchor=north] {$\xi$};
\puntoblanco{(0.71,-0.71)} \draw (0.71,-0.71) node[anchor=north west] {$1$};
\puntoblanco{(0.96,-0.27)} \draw (0.96,-0.27) node[anchor=west] {$1$};
\end{scriptsize}
\end{tikzpicture}

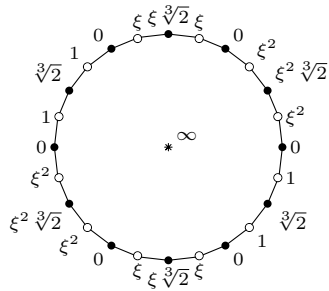
\captionof{figure}{The labels mark the $x$ coordinate of the points}\label{dodec}
\end{center}
\noindent The vertices of the face are the preimages of $\{0,1,\xi,\xi^2,\rd,\xi\rd,\xi^2\rd \}$. Some of the points, namely $\xi$, $\xi^2$, $\xi\rd$ and $\xi^2\rd$ have two different preimages and nonetheless these have the same label in the figure, because we are only labeling the $X$ coordinates, for simplicity. Note that the points are in the same order as in the dessin in figure \ref{des6}.

The dessin corresponding to $(C,f)$ is figure \ref{dodec} with some of its edges identified. What we are going to do now is figure out how the edges are identified. Since the curve is an orientable surface, the edges in the figure are identified in pairs, and the first coordinates of two identified edges must coincide. This leaves for each edge only two possibilities as to which edge it is identified with: take, for instance, the edges with endpoints marked 0 and 1, of which there are four. They are oriented in two different ways: two of them have endpoints $0-1$ in clockwise order, and the other two have endpoints $1-0$. Two edges with the same orientation cannot be identified, for the resulting surface must be orientable: therefore, each of the edges is identified with one of the other two which have the opposite orientation, which gives two possibilities for the way they are identified.

Now, if we look at the points over $\xi\rd$, we know that these are not ramified, so there must be two of them. There are two ways of identifying the edges marked $\xi-\xi\rd$, but the one identifying the non-adjacent ones gives only one preimage of $\xi\rd$, which is not the case. Therefore, the adjacent edges must be identified, and the same thing happens for $\xi^2\rd$. With the corresponding edges identified, the dessin must look like the following:
\begin{center}
\begin{tikzpicture}[line cap=round,line join=round,>=triangle 45,x=1.5cm,y=1.5cm]
\clip(-1.7,-1.5) rectangle (1.5,1.5);
\draw (-0.92,-0.38)-- (-0.71,-0.71);
\draw (-0.71,-0.71)-- (-0.38,-0.92);
\draw (-0.38,-0.92)-- (0,-1);
\draw (0,-1)-- (0.38,-0.92);
\draw (0.38,-0.92)-- (0.71,-0.71);
\draw (0.71,-0.71)-- (0.92,-0.38);
\draw (0.92,-0.38)-- (1,0);
\draw (1,0)-- (0.92,0.38);
\draw (0.92,0.38)-- (0.71,0.71);
\draw (0.71,0.71)-- (0.38,0.92);
\draw (0.38,0.92)-- (0,1);
\draw (0,1)-- (-0.38,0.92);
\draw (-0.38,0.92)-- (-0.71,0.71);
\draw (-0.71,0.71)-- (-0.92,0.38);
\draw (-0.92,0.38)-- (-1,0);
\draw (-1,0)-- (-0.92,-0.38);
\draw (0.65,0.27) -- (0.92,0.38);
\draw (0.27,0.65) -- (0.38,0.92);
\draw (-0.65,-0.27) -- (-0.92,-0.38);
\draw (-0.27,-0.65) -- (-0.38,-0.92);
\begin{scriptsize}
\draw (0,0)-- ++(-1.0pt,-1.0pt) -- ++(2.0pt,2.0pt) ++(-2.0pt,0) -- ++(2.0pt,-2.0pt) ++ (-1.0pt,-0.4pt) -- ++(0,2.8pt) ++(-1.4pt,-1.4pt)-- ++(2.8pt,0);
\draw (0,0) node[anchor=south west] {$\infty$};
\puntonegro{(0.65,0.27)} \draw (0.65,0.27) node[anchor=north] {$\xi^2\rd$};
\puntonegro{(0.27,0.65)} \draw (0.27,0.65) node[anchor=east] {$\xi\rd$};
\puntonegro{(-0.65,-0.27)} \draw (-0.65,-0.27) node[anchor=south] {$\xi^2\rd$};
\puntonegro{(-0.27,-0.65)} \draw (-0.27,-0.65) node[anchor=west] {$\xi\rd$};
\puntonegro{(1,0)} \draw (1,0) node[anchor= west] {$0$};
\puntoblanco{(0.92,0.38)} \draw (0.92,0.38) node[anchor=south west] {$\xi^2$};
\puntonegro{(0.71,0.71)} \draw (0.71,0.71) node[anchor=south west] {$0$};
\puntoblanco{(0.38,0.92)} \draw (0.38,0.92) node[anchor=south west] {$\xi$};
\puntonegro{(0,1)} \draw (0,1) node[anchor= south] {$0$};
\puntonegro{(0,-1)} \draw (0,-1) node[anchor= north] {$0$};
\puntonegro{(-1,0)} \draw (-1,0) node[anchor= east] {$0$};
\puntoblanco{(-0.92,0.38)} \draw (-0.92,0.38) node[anchor=south east] {$1$};
\puntonegro{(-0.71,0.71)} \draw (-0.71,0.71) node[anchor=south east] {$\rd$};
\puntoblanco{(-0.38,0.92)} \draw (-0.38,0.92) node[anchor=south east] {$1$};
\puntoblanco{(-0.92,-0.38)} \draw (-0.92,-0.38) node[anchor=north east] {$\xi^2$};
\puntonegro{(-0.71,-0.71)} \draw (-0.71,-0.71) node[anchor=north east] {$0$};
\puntoblanco{(-0.38,-0.92)} \draw (-0.38,-0.92) node[anchor=north east] {$\xi$};
\puntoblanco{(0.92,-0.38)} \draw (0.92,-0.38) node[anchor=north west] {$1$};
\puntonegro{(0.71,-0.71)} \draw (0.71,-0.71) node[anchor=north west] {$\rd$};
\puntoblanco{(0.38,-0.92)} \draw (0.38,-0.92) node[anchor=north west] {$1$};
\end{scriptsize}
\end{tikzpicture}
\end{center}

\noindent The same kind of reasoning applies to the points over $\xi$ and $\xi^2$, so the resulting figure is

\begin{center}
\begin{tikzpicture}[line cap=round,line join=round,>=triangle 45,x=2cm,y=2cm]
\clip(-1.7,-1.5) rectangle (1.5,1.5);
\draw (1,1)--(1,-1)--(-1,-1)--(-1,1)--(1,1);
\draw (1,0)--(0.28,0.46);
\draw (1,0)--(0.28,-0.46);
\draw (-1,0)--(-0.28,0.46);
\draw (-1,0)--(-0.28,-0.46);
\begin{scriptsize}
\draw (0,0)-- ++(-1.0pt,-1.0pt) -- ++(2.0pt,2.0pt) ++(-2.0pt,0) -- ++(2.0pt,-2.0pt) ++ (-1.0pt,-0.4pt) -- ++(0,2.8pt) ++(-1.4pt,-1.4pt)-- ++(2.8pt,0);
\draw (0,0) node[anchor=south west] {$\infty$};
\puntonegro{(1,0)} \draw (1,0) node[anchor= west] {$0$};
\puntoblanco{(1,1)} \draw (1,1) node[anchor=south west] {$1$};
\puntonegro{(0,1)} \draw (0,1) node[anchor= south] {$\rd$};
\puntonegro{(0,-1)} \draw (0,-1) node[anchor= north] {$\rd$};
\puntonegro{(-1,0)} \draw (-1,0) node[anchor= east] {$0$};
\puntoblanco{(-1,1)} \draw (-1,1) node[anchor=south east] {$1$};
\puntoblanco{(-1,-1)} \draw (-1,-1) node[anchor=north east] {$1$};
\puntoblanco{(1,-1)} \draw (1,-1) node[anchor=north west] {$1$};
\puntoblanco{(0.64,0.23)} \draw (0.64,0.23) node[anchor= south west] {$\xi$};
\puntoblanco{(0.64,-0.23)} \draw (0.64,-0.23) node[anchor= north west] {$\xi^2$};
\puntoblanco{(-0.64,0.23)} \draw (-0.64,0.23) node[anchor= south east] {$\xi^2$};
\puntoblanco{(-0.64,-0.23)} \draw (-0.64,-0.23) node[anchor= north east] {$\xi$};
\puntonegro{(0.28,0.46)} \draw (0.28,0.46) node[anchor= south west] {$\xi\rd$};
\puntonegro{(0.28,-0.46)} \draw (0.28,-0.46) node[anchor= north west] {$\xi^2\rd$};
\puntonegro{(-0.28,0.46)} \draw (-0.28,0.46) node[anchor= south east] {$\xi^2\rd$};
\puntonegro{(-0.28,-0.46)} \draw (-0.28,-0.46) node[anchor= north east] {$\xi\rd$};
\end{scriptsize}
\end{tikzpicture}
\end{center}

\noindent If we look at the remaining edges, we can see that for $0$, $1$ and $\rd$ to have one preimage each, there is just one way in which they can be identified, which is by identifying opposite sides of the square (note that this agrees with the fact that the curve has genus 1).

\begin{wrapfigure}[12]{l}{0.3\textwidth}
\centering
\vspace{-5 mm}
\begin{tikzpicture}[line cap=round,line join=round,>=triangle 45,x=1.5cm,y=1.5cm]
\clip(-1.7,-1.5) rectangle (1.5,1.5);
\draw (1,1)--(1,-1)--(-1,-1)--(-1,1)--(1,1);
\draw (1,0)--(0.28,0.46);
\draw (1,0)--(0.28,-0.46);
\draw (-1,0)--(-0.28,0.46);
\draw (-1,0)--(-0.28,-0.46);
\begin{scriptsize}
\puntonegro{(1,0)} 
\puntoblanco{(1,1)} 
\puntonegro{(0,1)} 
\puntonegro{(0,-1)}
\puntonegro{(-1,0)} 
\puntoblanco{(-1,1)} 
\puntoblanco{(-1,-1)} 
\puntoblanco{(1,-1)} 
\puntoblanco{(0.64,0.23)}
\puntoblanco{(0.64,-0.23)} 
\puntoblanco{(-0.64,0.23)}
\puntoblanco{(-0.64,-0.23)}
\puntonegro{(0.28,0.46)}
\puntonegro{(0.28,-0.46)} 
\puntonegro{(-0.28,0.46)}
\puntonegro{(-0.28,-0.46)}
\draw (1,0.5) node[anchor=west] {9};
\draw (1,-0.5) node[anchor=west] {3};
\draw (-1,0.5) node[anchor=east] {9};
\draw (-1,-0.5) node[anchor=east] {3};
\draw (0.5,1) node[anchor=south] {$6$};
\draw (-0.5,1) node[anchor=south] {$12$};
\draw (0.5,-1) node[anchor=north] {6};
\draw (-0.5,-1) node[anchor=north] {12};
\draw (0.78,0.11) node[anchor=south west] {1};
\draw (0.78,-0.11) node[anchor=north west] {2};
\draw (-0.78,0.11) node[anchor=south east] {8};
\draw (-0.78,-0.11) node[anchor=north east] {7};
\draw(0.46,0.34) node[anchor=south west] {4};
\draw(0.46,-0.34) node[anchor=north west] {5};
\draw(-0.46,0.34) node[anchor=south east] {11};
\draw(-0.46,-0.34) node[anchor=north east] {10};
\end{scriptsize}
\end{tikzpicture}
\vspace{-3 mm}
\caption{The dessin $D_0$ with a numbering on its edges\label{numeros12}}
\end{wrapfigure}

\noindent We want to look at this dessin's regular cover. In order to do that, we will look at its cartographic group (recall that the cartographic group of a dessin is the automorphism group of its regular cover, proposition \ref{autmono}). To be able to present it as a permutation group by its action on the edges of the dessin, we will number them as in figure \ref{numeros12}.

With this numbering, the cartographic group is generated by $x=(1,2,3,7,8,9)(6,12)$, $y=(1,4)(2,5)(7,10)(8,11)(3,6,9,12)$ and $z=(xy)^{-1}=(2,5,1,4,9,12,8,11,7,10,3,6)$ (we adopt the convention that the cartographic group acts on the right).

If we see the dessin as a subgroup $D_0$ of $\Fd$, it is the stabilizer of an edge in the dessin (any edge, since the action is transitive, and thus any two stabilizers are conjugate).
\newpage
\noindent If we take the edge marked 1, the stabilizer is the closure of the free group freely generated by
$$
\langle x^6,y^2,(y^x)^4,(y^{x^2})^2,(y^{x^3})^2,(y^{x^5})^2,x^y,(x^{yx})^2,x^{yx^2}, x^{yx^3},x^{yx^5},x^2y^2x,x^{-1}y^{-1}x^{-1}y^{-1}x\rangle
$$
(Note the convention that the cartographic group acts on the right, and that $a^b$ means $b^{-1}ab$). These generators are obtained as the generators for the fundamental group of the elliptic curve without the ramified points (recall that the subgroup associated to a cover is the pushforward of the fundamental group of the covering space, see section \ref{monodromy}). The elliptic curve without the ramification points is a torus with 12 points removed, which has fundamental group $F_{13}$. Each of the generators corresponds to a loop going around one of the points, except for the last two, which are the loops that generate the fundamental group of the torus. Then, its preimage in $\Fd$ is the same as its closure. Nonetheless, we won't need this presentation, and we will be able to obtain a nicer one.

We are interested in the regular cover $\D{0}$ of the dessin. Recall that the regular cover is given by $\D{0}=\cor_{\Fd} D_0$. The quotient $\Fd/\D{0}$ is then the cartographic group of the dessin, so $\D{0}$ is the set of relations defining this group. We will call this group $\mathcal C\cong \Fd/\D{0}$.

We will see that this dessin $\D{0}$ has $\Q(\rd)$ as field of moduli, by finding out also about the cartographic groups of its Galois conjugates, and seeing that they have different sets of relations.

\section{The cartographic group of $D_0$}

The cartographic group of $D_0$ is the group $\Fd/\D{0}$, and, since it acts faithfully on the edges of $D_0$, we can see it as the permutation group
$$
\CC=\langle x=(1,2,3,7,8,9)(6,12),y=(1,4)(2,5)(7,10)(8,11)(3,6,9,12)\rangle<S_{12}
$$
\noindent The first thing we notice is that both generators, and thus the whole group, preserve the partition
$$\{\{1,7\},\{2,8\},\{3,9\},\{4,10\},\{5,11\},\{6,12\} \} $$
This partition corresponds to the 180º rotational symmetry around the origin that the map exhibits, which is the map $(X,Y)\mapsto (X,-Y)$ in the curve. If we map $\{1,2,3,4,5,6\}\leftrightarrow \{\{1,7\},\{2,8\},\{3,9\},\{4,10\},\{5,11\},\{6,12\} \}$ in the obvious way, this induces a homomorphism $\pi:\CC \longrightarrow S_6$, which takes
$$
\begin{array}{rcl}
x=(1,2,3,7,8,9)(6,12)& \longmapsto &  (1,2,3)\\
y=(1,4)(2,5)(7,10)(8,11)(3,6,9,12) & \longmapsto &  (1,4)(2,5)(3,6)
\end{array}
$$
\noindent After few calculations, we see that the image of this map is the group
$$
F=\langle
x,y|x^3=y^2=[x,x^y]=1
\rangle \cong\zn{2} \ltimes (\zn{3} \times \zn{3})
$$
\noindent Where the action of $\zn{2}$ on $\zn{3}\times \zn{3}$ is given by interchanging the coordinates, and the map that takes $\pi(\CC)$ to $\zn{2} \ltimes (\zn{3} \times \zn{3})$ takes $\pi(x)$ to $(0,1,0)$ and $\pi(y)$ to $(1,0,0)$. $\pi(x^y)$ is then $(0,0,1)$. The semidirect product $\zn{2} \ltimes (\zn{3} \times \zn{3})$ can also be seen as the wreath product $\zn{3} \wr_{\{1,2\}} S_2$.

The kernel of the map $\pi$ is composed of the permutations that fix the previous partition pointwise, that is
$$
\CC\cap \langle (1,7),(2,8),(3,9),(4,10),(5,11),(6,12)\rangle
$$
We will call the group $\langle (1,7),(2,8),(3,9),(4,10),(5,11),(6,12)\rangle=\overline K$, and its intersection with $\CC$ will be called $K$. Let us find out what $K$ is. It is the kernel of the map $\pi$, so it is generated by the relations defining its image group. Thus, it is the normal subgroup of $\CC$ generated by $\{x^3,y^2,[x,x^y] \}$. To simplify the notation, we will call the transpositions in $\K$ $a_i=(i,i+6)$. Since the group they generate is commutative, we will use additive notation, so, for instance, $a_1+a_2$ means $(1,7)(2,6)$. Note that the group generated by the $a_i$'s is isomorphic to $(\zn{2})^6$. With this notation, we have
$$
x^3=a_1+a_2+a_3+a_6;y^2=a_3+a_6;[x,x^y]=a_2+a_3+a_5+a_6
$$
All the generators lie in the subgroup $\overline K\cap A_{12}$. Let us produce a set of generators of this group within $K$ to prove that $K=\overline K\cap A_{12} \cong (\zn{2})^5$. The set of generators can be given this way:
$$
\begin{array}{cc}
x^3y^2=a_1+a_2 & (x^3y^2)^y=a_4+a_5 \\
(x^3y^2)^x=a_2+a_3 & (x^3y^2)^{xy}=a_5+a_6 \\
y^2=a_3+a_6&  \\
\end{array}
$$
So we conclude that the kernel of $\pi$ is $K=A_{12}\cap \overline K$, which is isomorphic to $(\zn{2})^5$. Thus, the group $\CC$ is an extension of $F\cong \zn{3} \wr S_2$ by the group $(\zn{2})^5$, and in particular its order is $(3^2\cdot 2)\cdot 2^5=2^6\cdot 3^2$. Since the normal subgroup is abelian, $F$ is mapped into $\mathrm{Aut}(K)$, and if we see $K$ as a vector subspace of $(\zn{2})^6$, we can write the images of $F$ in $\mathrm{Aut}(K)$ using matrices, like so:
$$
x\longmapsto \left(\begin{matrix}
0 & 0 & 1 & 0 & 0 & 0\\
1 & 0 & 0 & 0 & 0 & 0\\
0 & 1 & 0 & 0 & 0 & 0\\
0 & 0 & 0 & 1 & 0 & 0\\
0 & 0 & 0 & 0 & 1 & 0\\
0 & 0 & 0 & 0 & 0 & 1\\
\end{matrix}\right);
y\longmapsto \left(\begin{matrix}
0 & 0 & 0 & 1 & 0 & 0\\
0 & 0 & 0 & 0 & 1 & 0\\
0 & 0 & 0 & 0 & 0 & 1\\
1 & 0 & 0 & 0 & 0 & 0\\
0 & 1 & 0 & 0 & 0 & 0\\
0 & 0 & 1 & 0 & 0 & 0\\
\end{matrix}\right)
$$
$K$ is an $F$-module with the action induced by these matrices.

The subgroup of $\Fd$ corresponding to the dessin $\D{0}$ is the group of relations defining $\CC$. For example, one such relation is $x^3y^2(x^3y^2)^x(x^3y^2)^{x^2}=1$, since
$$
x^3y^2(x^3y^2)^x(x^3y^2)^{x^2}=(a_1+a_2)+(a_1+a_2)^x+(a_1+a_2)^{x^2}=a_1+a_2+a_2+a_3+a_3+a_1=0
$$
We will see that this relation is not satisfied by the dessins which are Galois conjugate to $\D{0}$.



\section{The dessins conjugate to $\D{0}$}

We will now look at the conjugate dessins to our original dessin $(C,f)$ and its regular cover $\D{0}$. Recall that the Galois action preserves regular covers (proposition \ref{autmono}), so the dessins conjugate to $\D{0}$ are the regular covers of the dessins conjugate to $D_0$. Therefore, if we take a $\sigma\in \gal$ that maps our dessin $(C,f)$ to some dessin $(C^\sigma,f^\sigma)$, it will map $\reg$ to $\widetilde{(C^\sigma,f^\sigma)}$. What we need then is to find the dessins conjugate to $(C,f)$ and their regular covers.

Our dessin is defined over the field $\Q(\rd)$, so its field of moduli is contained in this field, and to know its image by a Galois automorphism $\sigma$, it suffices to know $\sigma(\rd)$. Therefore, there are two other dessins, $D_1$ and $D_2$, conjugate to $(C,f)$, which are given by the Belyi pairs $(C_1,f_1)$ and $(C_2,f_2)$, where
$$
C_1=V(Y^2=X(X-1)(X-\xi\rd));C_2=V(Y^2=X(X-1)(X-\xi^2\rd))
$$
\noindent The maps $f_1$ and $f_2$ have the same expression as $f$, namely $(X,Y)\longmapsto 1-(X^3-1)^2$, since it is defined over $\Q$. Note that these dessins are all different, since the underlying curves are not isomorphic. A different proof of this fact will follow from the fact that their regular covers are not isomorphic.

We will use the same procedure as before to draw the dessins. We start with the map $X\mapsto 1-(X^3-1)^2$, as in figure \ref{des6}, since this part of the map is the same for all three. Now, the last part is different, since for $D_1$ it is ramified over $\{0,1,\xi\rd,\infty\}$ and for $D_2$ over $\{0,1,\xi^2\rd,\infty \}$. Like it happened before, the only face of the dessins will look the same as figure \ref{dodec}, since $\infty$ is ramified. For $D_1$, the points $\rd$ and $\xi^2\rd$ are unramified, so we know the way their neighboring edges are identified. Once the neighboring sides to the points over $\xi^2\rd$ are identified, we see the way the neighboring sides to $\xi^2$ are identified. An analogous reasoning holds for $D_2$.

\begin{center}
\begin{tabular}{cc}
\begin{tikzpicture}[line cap=round,line join=round,>=triangle 45,x=1.5cm,y=1.5cm]
\clip(-2.4,-1.3) rectangle (2.4,1.3);
\draw (-2,-1)--(-2,1)--(2,1)--(2,-1)--(-2,-1);
\draw (-2,0)--(-1,0);
\draw (2,0) --(1,0);
\draw (-2,-1) --(-0.5,-0.3);
\draw (2,1) --(0.5,0.3);
\begin{scriptsize}
\draw (0,0)-- ++(-1.0pt,-1.0pt) -- ++(2.0pt,2.0pt) ++(-2.0pt,0) -- ++(2.0pt,-2.0pt) ++ (-1.0pt,-0.4pt) -- ++(0,2.8pt) ++(-1.4pt,-1.4pt)-- ++(2.8pt,0);
\draw (0,0) node[anchor=south west] {$\infty$};
\puntoblanco{(-2,0)} \draw (-2,0) node[anchor=east] {$1$};
\puntoblanco{(2,0)} \draw (2,0) node[anchor=west] {$1$};
\puntonegro{(-1,0)} \draw (-1,0) node[anchor=west] {$\rd$};
\puntonegro{(1,0)} \draw (1,0) node[anchor=east] {$\rd$};
\puntonegro{(-2,1)} \draw (-2,1) node[anchor= south east] {$0$};
\puntonegro{(-2,-1)} \draw (-2,-1) node[anchor= north east] {$0$};
\puntonegro{(2,1)} \draw (2,1) node[anchor= south west] {$0$};
\puntonegro{(2,-1)} \draw (2,-1) node[anchor= north west] {$0$};
\puntoblanco{(1,1)} \draw (1,1) node[anchor= south] {$\xi$};
\puntoblanco{(-1,1)} \draw (-1,1) node[anchor= south] {$\xi$};
\puntoblanco{(1,-1)} \draw (1,-1) node[anchor= north] {$\xi$};
\puntoblanco{(-1,-1)} \draw (-1,-1) node[anchor= north] {$\xi$};
\puntonegro{(0,1)} \draw (0,1) node[anchor=south] {$\xi\rd$};
\puntonegro{(0,-1)} \draw (0,-1) node[anchor=north] {$\xi\rd$};
\puntoblanco{(1.25,0.65)} \draw (1.25,0.65) node[anchor= north west] {$\xi^2$};
\puntoblanco{(-1.25,-0.65)} \draw (-1.25,-0.65) node[anchor= south east] {$\xi^2$};
\puntonegro{(0.5,0.3)} \draw (0.5,0.3) node[anchor=south east] {$\xi^2\rd$};
\puntonegro{(-0.5,-0.3)} \draw (-0.5,-0.3) node[anchor=north west] {$\xi^2\rd$};
\end{scriptsize}
\end{tikzpicture}
&
\begin{tikzpicture}[line cap=round,line join=round,>=triangle 45,x=1.5cm,y=1.5cm]
\clip(-2.4,-1.3) rectangle (2.4,1.3);
\draw (-2,-1)--(-2,1)--(2,1)--(2,-1)--(-2,-1);
\draw (-2,0)--(-1,0);
\draw (2,0) --(1,0);
\draw (-2,1) --(-0.5,0.3);
\draw (2,-1) --(0.5,-0.3);
\begin{scriptsize}
\draw (0,0)-- ++(-1.0pt,-1.0pt) -- ++(2.0pt,2.0pt) ++(-2.0pt,0) -- ++(2.0pt,-2.0pt) ++ (-1.0pt,-0.4pt) -- ++(0,2.8pt) ++(-1.4pt,-1.4pt)-- ++(2.8pt,0);
\draw (0,0) node[anchor=south west] {$\infty$};
\puntoblanco{(-2,0)} \draw (-2,0) node[anchor=east] {$1$};
\puntoblanco{(2,0)} \draw (2,0) node[anchor=west] {$1$};
\puntonegro{(-1,0)} \draw (-1,0) node[anchor=west] {$\rd$};
\puntonegro{(1,0)} \draw (1,0) node[anchor=east] {$\rd$};
\puntonegro{(-2,1)} \draw (-2,1) node[anchor= south east] {$0$};
\puntonegro{(-2,-1)} \draw (-2,-1) node[anchor= north east] {$0$};
\puntonegro{(2,1)} \draw (2,1) node[anchor= south west] {$0$};
\puntonegro{(2,-1)} \draw (2,-1) node[anchor= north west] {$0$};
\puntoblanco{(1,1)} \draw (1,1) node[anchor= south] {$\xi^2$};
\puntoblanco{(-1,1)} \draw (-1,1) node[anchor= south] {$\xi^2$};
\puntoblanco{(1,-1)} \draw (1,-1) node[anchor= north] {$\xi^2$};
\puntoblanco{(-1,-1)} \draw (-1,-1) node[anchor= north] {$\xi^2$};
\puntonegro{(0,1)} \draw (0,1) node[anchor=south] {$\xi^2\rd$};
\puntonegro{(0,-1)} \draw (0,-1) node[anchor=north] {$\xi^2\rd$};
\puntoblanco{(1.25,-0.65)} \draw (1.25,-0.65) node[anchor= north east] {$\xi$};
\puntoblanco{(-1.25,0.65)} \draw (-1.25,0.65) node[anchor= south west] {$\xi$};
\puntonegro{(0.5,-0.3)} \draw (0.5,-0.3) node[anchor=north east] {$\xi\rd$};
\puntonegro{(-0.5,0.3)} \draw (-0.5,0.3) node[anchor=south west] {$\xi\rd$};
\end{scriptsize}
\end{tikzpicture}
\end{tabular}

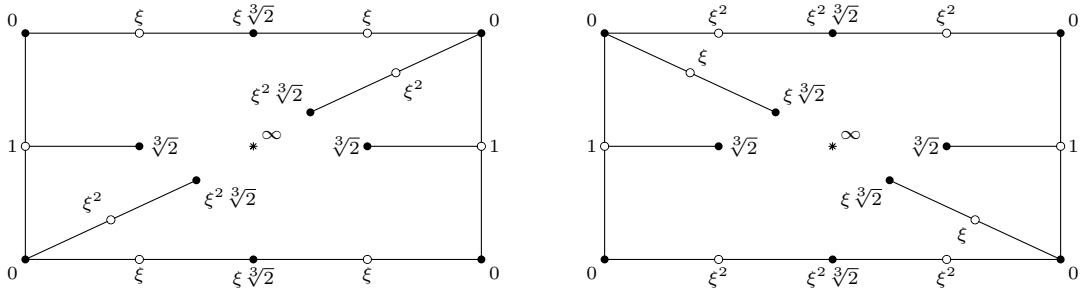
\captionof{figure}{To the left, the points on $C_1$ and to the right, the points on $C_2$}
\end{center}

\noindent In order to see how the remaining edges are identified, note that the edges surrounding 1 have just one way to be identified in order for 1 to be ramified, and the same thing happens with the edges surrounding $\xi\rd$ (in $C_1$). The four remaining edges then have only one way to be paired for $0$ to have one preimage, and similarly for $C_2$. Thus, the opposing sides of the rectangle are identified. We can then number the edges as follows:

\begin{center}
\begin{tabular}{cc}
\begin{tikzpicture}[line cap=round,line join=round,>=triangle 45,x=1.5cm,y=1.5cm]
\clip(-2.4,-1.3) rectangle (2.4,1.3);
\draw (-2,-1)--(-2,1)--(2,1)--(2,-1)--(-2,-1);
\draw (-2,0)--(-1,0);
\draw (2,0) --(1,0);
\draw (-2,-1) --(-0.5,-0.3);
\draw (2,1) --(0.5,0.3);
\begin{scriptsize}
\puntoblanco{(-2,0)}
\puntoblanco{(2,0)}
\puntonegro{(-1,0)} 
\puntonegro{(1,0)} 
\puntonegro{(-2,1)}
\puntonegro{(-2,-1)}
\puntonegro{(2,1)}
\puntonegro{(2,-1)} 
\puntoblanco{(1,1)}
\puntoblanco{(-1,1)} 
\puntoblanco{(1,-1)} 
\puntoblanco{(-1,-1)}
\puntonegro{(0,1)}
\puntonegro{(0,-1)}
\puntoblanco{(1.25,0.65)} 
\puntoblanco{(-1.25,-0.65)}
\puntonegro{(0.5,0.3)} 
\puntonegro{(-0.5,-0.3)}
\draw (-1.5,1) node[anchor=south]{1};
\draw (-0.5,1) node[anchor=south]{4};
\draw (0.5,1) node[anchor=south]{10};
\draw (1.5,1) node[anchor=south]{7};
\draw (-1.5,-1) node[anchor=north]{1};
\draw (-0.5,-1) node[anchor=north]{4};
\draw (0.5,-1) node[anchor=north]{10};
\draw (1.5,-1) node[anchor=north]{7};
\draw (2,0.5) node[anchor=west]{9};
\draw (2,-0.5) node[anchor=west]{3};
\draw (-2,0.5) node[anchor=east]{9};
\draw (-2,-0.5) node[anchor=east]{3};
\draw (-1.7,-0.8) node[anchor=south]{2};
\draw (1.7,0.8) node[anchor=north]{8};
\draw (-0.9,-0.5) node[anchor=south]{5};
\draw (0.9,0.5) node[anchor=north]{11};
\draw (1.5,0) node[anchor=north]{12};
\draw (-1.5,0) node[anchor=south]{6};
\end{scriptsize}
\end{tikzpicture}
&
\begin{tikzpicture}[line cap=round,line join=round,>=triangle 45,x=1.5cm,y=1.5cm]
\clip(-2.4,-1.3) rectangle (2.4,1.3);
\draw (-2,-1)--(-2,1)--(2,1)--(2,-1)--(-2,-1);
\draw (-2,0)--(-1,0);
\draw (2,0) --(1,0);
\draw (-2,1) --(-0.5,0.3);
\draw (2,-1) --(0.5,-0.3);
\begin{scriptsize}
\puntoblanco{(-2,0)}
\puntoblanco{(2,0)}
\puntonegro{(-1,0)} 
\puntonegro{(1,0)} 
\puntonegro{(-2,1)}
\puntonegro{(-2,-1)}
\puntonegro{(2,1)}
\puntonegro{(2,-1)} 
\puntoblanco{(1,1)}
\puntoblanco{(-1,1)} 
\puntoblanco{(1,-1)} 
\puntoblanco{(-1,-1)}
\puntonegro{(0,1)}
\puntonegro{(0,-1)}
\puntoblanco{(1.25,-0.65)} 
\puntoblanco{(-1.25,0.65)}
\puntonegro{(0.5,-0.3)} 
\puntonegro{(-0.5,0.3)}
\draw (-1.5,1) node[anchor=south]{2};
\draw (-0.5,1) node[anchor=south]{5};
\draw (0.5,1) node[anchor=south]{11};
\draw (1.5,1) node[anchor=south]{8};
\draw (-1.5,-1) node[anchor=north]{2};
\draw (-0.5,-1) node[anchor=north]{5};
\draw (0.5,-1) node[anchor=north]{11};
\draw (1.5,-1) node[anchor=north]{8};
\draw (2,0.5) node[anchor=west]{9};
\draw (2,-0.5) node[anchor=west]{3};
\draw (-2,-0.5) node[anchor=east]{3};
\draw (-2,0.5) node[anchor=east]{9};
\draw (-1.7,0.8) node[anchor=north]{1};
\draw (1.7,-0.8) node[anchor=south]{7};
\draw (-0.9,0.5) node[anchor=north]{4};
\draw (0.9,-0.5) node[anchor=south]{10};
\draw (1.5,0) node[anchor=north]{12};
\draw (-1.5,0) node[anchor=south]{6};
\end{scriptsize}
\end{tikzpicture}
\end{tabular}

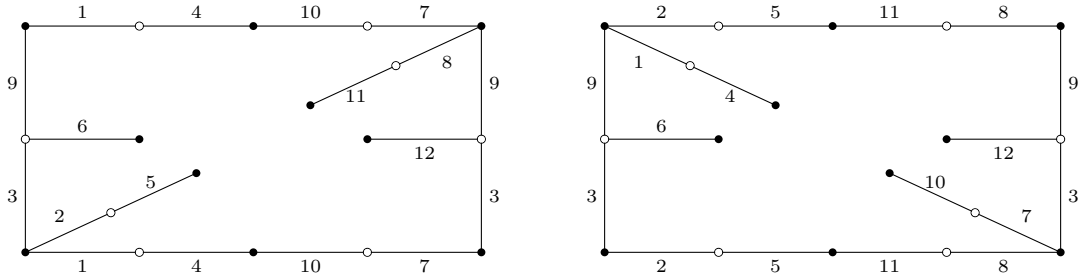
\captionof{figure}{The way the edges are identified in $C_1$ and $C_2$, with a numbering.}
\end{center}

\noindent This way, the cartographic group of $D_1$ can be seen as generated by $x=(1,2,3,7,8,9)(4,10)$ and $y=(1,4)(2,5)(7,10)(8,11)(3,6,9,12)$, and the cartographic group of $D_2$ is generated by $x=(1,2,3,7,8,9)(5,11)$ and $y=(1,4)(2,5)(7,10)(8,11)(3,6,9,12)$. We will call the cartographic groups $\CC_1=\Fd/\D{1}$ and $\CC_2=\Fd/\D{2}$. In terms of subgroups of $\Fd$, this means that to the dessins $D_1$ and $D_2$ correspond two open subgroups such that the action of $\Fd$ on their right cosets gives these permutations.

Note that with the numbering we have chosen, the maps $\pi_1:\CC_1\longrightarrow F<S_6$ and $\pi_2$ can be defined the same as for $\CC$, and that the images of $x$ and $y$ are still $(123)$ and $(14)(25)(36)$ respectively. The kernels are the same subgroups of $S_{12}$ (isomorphic to $\zn{2}^5$), and the kernels of $\pi_1$ and $\pi_2$ are isomorphic to the kernel of $\pi$ as $F$-modules.

However, in $\CC_1$, $x^3=a_1+a_2+a_3+a_4$, so $x^3y^2=a_1+a_2+a_3+a_4+a_3+a_6=a_1+a_2+a_4+a_6$. This gives a hint that there may not be an isomorphism between the cartographic groups preserving $x$ and $y$. Indeed, the relation $x^3y^2(x^3y^2)^x(x^3y^2)^{x^2}=1$, which holds in $\CC$, in $\CC_1$ reads
\begin{align*}
x^3y^2(x^3y^2)^x(x^3y^2)^{x^2}=&(a_1+a_2+a_4+a_6)+(a_1+a_2+a_4+a_6)^x+(a_1+a_2+a_4+a_6)^{x^2}=\\ =&(a_1+a_2+a_4+a_6)+(a_2+a_3+a_4+a_6)+(a_3+a_1+a_4+a_6)=a_4+a_6\neq 0
\end{align*}
\noindent And analogously, in $\CC_2$,
$$
x^3y^2(x^3y^2)^x(x^3y^2)^{x^2}=a_5+a_6
$$
\noindent This means that there is no isomorphism from $\CC$ to $\CC_1$ (or $\CC_2$) taking $x$ to $x$ and $y$ to $y$, since it wouldn't map $x^3y^2(x^3y^2)^x(x^3y^2)^{x^2}$ to $1$. If there is no isomorphism, the dessin $\D{0}$ is not fixed by the Galois action: take $\sigma\in \gal$ such that $\sigma(\rd)=\xi\rd$. Then, $D_0^\sigma=D_1$, and $(\D{0})^\sigma=\widetilde{D_0^\sigma}=\widetilde{D_1}$. But $\D{0}$ and $\D{1}$ are not conjugate: since they are normal, if they were conjugate they would have to be the same group, but $x^3y^2(x^3y^2)^x(x^3y^2)^{x^2}$ belongs to $\D{0}$ but not to $\widetilde D_1$. The same applies to a $\sigma \in \gal$ taking $\rd$ to $\xi^2\rd$: it would map $\D{0}$ to $\wt{D_2}$, and it is not the same group for the same reason.

We conclude that an automorphism $\sigma\in \gal$ fixes $\D{0}$ if and only if it fixes $\rd$. This means that the field of moduli of $\D{0}$ is $\Q(\rd)$. In particular, we have constructed a regular dessin with non-abelian field of moduli. For $\D{1}$ and $\D{2}$, since they are conjugate to $\D{0}$, we have that their fields of moduli are $\Q(\xi\rd)$ and $\Q(\xi^2\rd)$, respectively.

We can calculate the genus $g$ of $\D{0}$: the monodromy action on it is given by the group acting on itself. Therefore, its edges are the elements of the group $\CC$, it has a face for every 12 edges (since $xy$ has order 12), a black vertex for every $6$ edges (the order of $x$), and a white vertex for every $4$ edges. Its Euler-Poincaré characteristic is then $2g-2=|\CC|\left(\frac{1}{|x|}+\frac{1}{|y|}+\frac{1}{|xy|}-1\right)=2^6\cdot 3^2\left(\frac{1}{6}+\frac{1}{4}+\frac{1}{12}-1\right)=2^4\cdot 3\cdot 5=288$. So its genus is $145$.

We are going to make another interpretation of this picture in terms of subgroups of $\Fd$, and to use it to give an example with smaller genus. We have three subgroups $\wt{D_0}$, $\wt{D_1}$ and $\wt{D_2}$ of $\Fd$, which are permuted by the action of the Galois group. Their intersection is then some subgroup $A$ of $\Fd$, and by its construction it must be fixed by the Galois action (and since it is a regular dessin, it is defined over its field of moduli, which is $\Q$). The quotients $\Fd/\wt{D_i}$ are the cartographic groups of the dessins, and we have seen they can all be mapped into the group $F$. This means that the group $B$ of defining relations for $F$ contains all of them, and it is also fixed by the Galois action.

\begin{center}\hspace{0.5 cm}
\begindc{\commdiag}[40]
\obj(0,2)[f]{$\Fd$}
\obj(0,1)[b]{$B$}
\obj(-1,0)[d0]{$\wt{D_0}$}
\obj(0,0)[d1]{$\wt{D_1}$}
\obj(1,0)[d2]{$\wt{D_2}$}
\obj(0,-1)[a]{$A$}
\mor{b}{f}{$18$}[\atleft,\solidline]
\mor{d0}{b}{$2^5$}[\atleft,\solidline]
\mor{d1}{b}{$2^5$}[\atleft,\solidline]
\mor{d2}{b}{$2^5$}[\atright,\solidline]
\mor{a}{d0}{}[\atleft,\solidline]
\mor{a}{d1}{}[\atleft,\solidline]
\mor{a}{d2}{}[\atright,\solidline]
\enddc
\end{center}

\noindent Let us find out what lies between $A$ and $B$, specifically, about the group $B/A$. Note that, since $B/\D{0}$, $B/\D{1}$ and $B/\D{2}$ are commutative, $B/(\D{0}\cap \D{1} \cap \D{2})=B/A$ also is commutative, and therefore $\Fd/B=F$ acts on it, so $B/A$ is an $F$-module. In order to find out about its structure, we can use the $F$-module homomorphism
$$
\f{J}{\displaystyle\frac{B}{\wt{D_0}\cap\wt{D_1}\cap \D{2}}}{\displaystyle\frac{B}{\wt{D_0}}\times \frac{B}{\wt{D_1}}\times \frac{B}{\D{2}}}{g}{\left(g \mod \wt{D_0},g\mod \wt{D_1},g\mod \D{2}\right)}
$$
\noindent This map is obviously well-defined and injective, and therefore its image is isomorphic to $B/(\wt{D_0}\cap\wt{D_1}\cap \D{2})$. We will call its image $\oB=J(B)$, and as a module it is a submodule of $K\times K\times K$ (recall the definition of $K$: $K\cong B/\D{0}\cong B/\D{1}\cong B/\D{2}\cong \zn{2}^5$).

The $F$ action on $\oB$ comes from the $F$ action on each of the modules $K$, like so: $J(g)^a=J(g^a)=(g^a \mod \wt{D_0},g^a \mod \wt{D_1},g^a \mod \D{2})$. Overall, we know the image of $J$ is generated as an $F$-module by $J(x^3)$, $J(y^2)$ and $J([x,x^y])$, because $x^3$, $y^2$ and $[x,x^y]$ generate $B$ as an $F$-module. We can use this to compute the image. The map on the generators goes as follows:
$$
\begin{array}{ccrll}
x^3 &\longmapsto &(a_1+a_2+a_3+a_6,&a_1+a_2+a_3+a_4,&a_1+a_2+a_3+a_5\\
y^2 &\longmapsto &(a_3+a_6,&a_3+a_6,&a_3+a_6)\\
\left[x,x^y\right] & \longmapsto &(a_2+a_3+a_5+a_6,&a_1+a_3+a_4+a_6,&a_1+a_2+a_4+a_5)\\
\end{array}
$$
\noindent We need to play around with the images for a bit to produce as many generators as we can. For example, using the image of $y^2$, we can generate every element of the form $(a,a,a)$, like so:
\begin{equation}\label{gens1}
\begin{array}{ccrllccrll}
y^2 &\longmapsto &(a_3+a_6,&a_3+a_6,&a_3+a_6)& &\\
y^2(y^2)^x &\longmapsto &(a_1+a_3,&a_1+a_3,&a_1+a_3)&(y^2)^y(y^2)^{xy} &\longmapsto &(a_4+a_6,&a_4+a_6,&a_4+a_6)\\
y^2(y^2)^{x^2} &\longmapsto &(a_2+a_3,&a_2+a_3,&a_2+a_3)&(y^2)^y(y^2)^{x^2y} &\longmapsto &(a_5+a_6,&a_5+a_6,&a_5+a_6)\\
\end{array}
\end{equation}
\noindent Also, we can produce other elements, like
$$
x^3y^2 \longmapsto (a_1+a_2,a_1+a_2+a_4+a_6,a_1+a_2+a_5+a_6)$$
\begin{equation}\label{gens2}
\begin{array}{ccrllccrll}
\phi=x^3y^2(y^2)^x(y^2)^{x^2} &\longmapsto &(0,&a_4+a_6,&a_5+a_6)&
\phi^y & \longmapsto &(0,&a_1+a_3,&a_2+a_3)\\
\phi^{x^y} &\longmapsto &(0,&a_4+a_5,&a_4+a_6)&
\phi^{yx} & \longmapsto &(0,&a_1+a_2,&a_1+a_3)\\
\end{array}
\end{equation}
\noindent We claim that the generators we have come up with so far generate the whole module as a group, i.e. if we define the submodule $K'$ of $K$ as $K'=\{n_1a_1+n_2a_2+n_3a_3+n_4a_4+n_5a_5+n_6a_6\in K:n_1+n_2+n_3\equiv n_4+n_5+n_6\equiv 0\mod 2\}$, it is clearly a codimension $1$ submodule of $K$, that is, an index 2 subgroup, and we claim that the whole module $M$ is given by
$$
M=\{(a,a,a):a\in K \}+\{(0,a',a'^{(xx^y)^{-1}}): a\in K,a'\in K'\}
$$
\noindent The proof goes as follows: first of all, this is indeed an $F$-module, since it is invariant under the actions of $x$ and $y$ (because $(xx^y)^{-1}$ commutes with both $x$ and $y$), and second of all, the images of $x^3$, $y^2$ and $[x,x^y]$, which are the generators for the module, lie in $M$, since
\begin{align*}
J(x^3)=&(a_1+a_2+a_3+a_6,a_1+a_2+a_3+a_4,a_1+a_2+a_3+a_5)=\\
=&(a_1+a_2+a_3+a_6,a_1+a_2+a_3+a_6,a_1+a_2+a_3+a_6)+(0,a_4+a_6,a_5+a_6)\\
J(y^2)=&(a_3+a_6,a_3+a_6,a_3+a_6)\\
J([x,x^y])=&(a_2+a_3+a_5+a_6,a_1+a_3+a_4+a_6,a_1+a_2+a_4+a_5)=\\
=&(a_2+a_3+a_5+a_6,a_2+a_3+a_5+a_6,a_2+a_3+a_5+a_6)+(0,a_1+a_2+a_4+a_5,a_1+a_3+a_5+a_6)
\end{align*}
\noindent Also, we have already seen that the image of $J$ generates at least the whole module we have given, since the generators for the first summand are given in (\ref{gens1}) and the generators for the second are given in (\ref{gens2}). Therefore, $B/A\cong M$

The first module in the description of $\oB$ has dimension $5$ as a $\zn{2}$-vector space, and the second module has dimension $4$. The modules clearly intersect only at 0, so the sum is direct.

Let us give a nicer description of the module $\oB$. If we let $t=a_1+a_2+a_3+a_4+a_5+a_6\in K$, we can write an element $a\in K$ as $a=a'+\eps_a t$, where $a'\in K'$ and $\eps_a \in \{0,1\}$. Since an element of $\oB$ can be written as $(a,a,a)+(0,a',a'^{(xx^y)^{-1}})$, where $a\in K$ and $a'\in K'$, we can decompose $a$ to obtain
\begin{align*}
(a,a,a)+(0,k',k'^{(xx^y)^{-1}})&=(a'+\eps_a t,a'+\eps_a t,a'+\eps_a t)+(0,k',k'^{(xx^y)^{-1}})=\\ &=\eps_a(t,t,t)+(a',a',a')+(0,k',k'^{(xx^y)^{-1}})\\&=\eps_a(t,t,t)+(a',0,a'^{xx^y})+(0,k'+a',k'^{(xx^y)^{-1}}+a'+a'^{xx^y})
\end{align*}
\noindent Now, note that for all $a'\in K'$, $a+a^{xx^y}+a^{(xx^y)^{-1}}=0$, so $a'+a'^{xx^y}=a'^{(aa^y)^{-1}}$, and, if we call $b'=k'+a'$, this means that

$$
(a,a,a)+(0,k',k'^{(xx^y)^{-1}})=\eps_a(t,t,t)+(a',0,a'^{xx^y})+(0,b',b'^{(xx^y)^{-1}})
$$

\noindent Thus, $\oB$ can be decomposed in the following way:

\begin{equation}\label{desc}
\oB=\langle (t,t,t)\rangle  \oplus \{(0,a,a^{(xx^y)^{-1}}):a\in K\}\oplus \{(a,0,a^{xx^y}):a\in K\} =T\oplus \K_0 \oplus \K_1
\end{equation}

\noindent The sum is direct since one module doesn't intersect the sum of the other two. Note that the first module is isomorphic to $\zn{2}$ with the trivial action and that the second and third modules are isomorphic to $K$. The order of $M$ is then $2\cdot |K| \cdot |K|=2^9$.

Also, we can see what $\D{0}$ is mapped into in $\oB$: its image is composed of the elements with first coordinate $0$ (since this coordinate is $g\mod \D{0}$), which is the submodule $\K_0$. Similarly, $J(\D{1})=\K_1$, and finally, the image of $\D{2}$ is the elements of third coordinate $0$, which is the module $\K_2=\{(a^{(xx^y)^{-1}},a^{xx^y},0):a\in K\}$. It follows that $(\D{0}\cap \D{1})/(\D{0}\cap \D{1}\cap\D{2})\cong \K_0 \cap \K_1 = 0$, so $\D{0}\cap \D{1}=\D{0}\cap \D{1}\cap \D{2}$.

We can now draw a diagram of the subgroups of $\Fd$ with more information: If we call $B'=\D{0}\D{1}$ (that is, $B'=J^{-1}(\K_0+\K_1)$), $E_i=J^{-1}(K_i+T)$, and $T'=J^{-1}(T)$, we can draw the following:

\begin{center}\hspace{0.5 cm}
\begindc{\commdiag}[3]
\obj(0,30)[f]{$\Fd$}
\obj(0,20)[b]{$B$}
\obj(0,10)[b']{$B'$}
\obj(-20,5)[e0]{$E_0$}
\obj(20,5)[e1]{$E_1$}
\obj(-20,-5)[d0]{$\wt{D_0}$}
\obj(20,-5)[d1]{$\wt{D_1}$}
\obj(0,-10)[t]{$T'$}
\obj(0,-20)[a]{$A$}
\mor{b}{f}{$18$}[\atleft,\solidline]
\mor{b'}{b}{$2$}[\atleft,\solidline]
\mor{e0}{b}{$2^4$}[\atleft,\solidline]
\mor{d0}{b'}{$2^4$}[\atleft,\solidline]
\mor{e1}{b}{$2^4$}[\atleft,\solidline]
\mor{d1}{b'}{$2^4$}[\atright,\solidline]
\mor{d0}{e0}{$2$}[\atleft,\solidline]
\mor{d1}{e1}{$2$}[\atright,\solidline]
\mor{t}{e0}{$2^4$}[\atleft,\solidline]
\mor{t}{e1}{$2^4$}[\atright,\solidline]
\mor{a}{d0}{$2^4$}[\atleft,\solidline]
\mor{a}{d1}{$2^4$}[\atright,\solidline]
\mor{a}{t}{$2$}[\atright,\solidline]
\enddc
\end{center}

\noindent We are omitting $\D{2}$ and $E_2$ for simplicity. There are three groups, namely $E_0$, $E_1$ and $E_2$, which contain the dessins we are working with, the $\D{i}$'s, and since the Galois group preserves the subgroup structure, they can't be left fixed by its action. In the next section, we will see that the Galois group has the same action on these dessins, which have smaller genus, namely 61 (which we will see), and smaller degree, namely $18\cdot 2^4=288$.

\section{Another dessin with non-abelian field of moduli}

In the previous diagram, there was a group $E_i$ containing each of the $\D{i}$. Let us see that if $\sigma\in \gal$ sends $\D{i}$ to $\D{j}$, then $E_i^\sigma=E_j$ also. In order to to this, we are going to see that $E_i/\D{i}$ is the center of $\Fd/\D{i}= \CC_i$.

First of all, the center of $\Fd/\D{i}$ has to be contained in the projection of the center of $\Fd/B$, which is $\{1,xx^y,(xx^y)^2\}$. However, $xx^y$ doesn't act trivially on the module, so its class is not in the center. Therefore, $Z(\Fd/\D{i})\subset B/\D{i}\cong T\oplus K$, and it must be the set of elements fixed by the $F$-action. It is straightforward to see that the elements fixed by the $F$ action are the elements of $T=T'/A=\{0,(t,t,t)\}$.

Therefore, the modules $E_i$ are the centers of $\Fd/\D{i}$, and they are characteristic subgroups of them. Since the Galois group acts by automorphisms of $\Fd$, if $\sigma \in \gal$ takes $\D{i}$ to $\D{j}$, it will induce an isomorphism of the quotients $\Fd/\D{i}$ and $\Fd/\D{j}$, which will map $Z(\Fd/\D{i})=E_i$ to $Z(\Fd/\D{j})=E_j$. Therefore, the action of the Galois group on the $E_i$'s is the same as the action on the $\D{i}$'s, and in particular the field of moduli of $E_i$ is $\Q(\xi^i\rd)$.

If we look at $E_i/\D{i}$ as a subgroup of $\Fd/\D{i}$, it is a subgroup of order 2, and it is generated by $J^{-1}(t,t,t) \mod \D{i}=a_1+a_2+a_3+a_4+a_5+a_6$. As permutations, $E_i/\D{i}=\{1,(1,7)(2,8)(3,9)(4,10)(5,11)(6,12)\}< \CC_i$. Therefore,
$$\frac{\Fd}{E_i}\cong \frac{\CC_i}{E_i/\D{i}}=\frac{\CC_i}{\langle (1,7)(2,8)(3,9)(4,10)(5,11)(6,12)\rangle}$$
\noindent Note that $(xy)^6=(1,7)(2,8)(3,9)(4,10)(5,11)(6,12)$, so in the cartographic group of this dessin, $z$ will have order 6. Since the orders of $x$ and $y$ are $6$ and $4$ respectively, the formula for the Euler-Poincaré characteristic is $2-2g=24\cdot 12\cdot(\frac{1}{6}+\frac{1}{4}+\frac{1}{6}-1)=-120$, so the genus is $61$.

Now, the module $T'=E_0\cap E_1\cap E_2$ is fixed by the Galois action, and the quotient $B/T'$ is isomorphic to $K_0 \oplus  K_1$, since we are taking the quotient by the module $T$. The module $E_0/T$, for instance, maps into $(T\oplus K_0)/T \cong K_0$.

We can represent the module $K'$ by means of the finite field with four elements $\F_4$. If we see it as $\F_4=\F_2[\xi]$ for a cubic root of unity $\xi$, we can take the group isomorphism
$$
\begin{array}{rcl}
\Phi:K' & \longrightarrow & \F_4^2 \\
a_1+a_2 & \longmapsto & (1,0)  \\
a_2+a_3 & \longmapsto & (\xi^2,0)  \\
a_4+a_5 & \longmapsto & (0,1)  \\
a_5+a_6 & \longmapsto & (0,\xi^2)  \\
\end{array}
$$
\noindent And the homomorphism
$$
\begin{array}{rcl}
\rho:F & \longrightarrow & GL_2(\F_4) \\
x & \longmapsto & \left(\begin{matrix}\xi^2 & 0\\0 &1\end{matrix}\right)  \\
y & \longmapsto & \left(\begin{matrix}0 & 1\\1 &0\end{matrix}\right)  \\
\end{array}
$$
\noindent The reason we are taking $\xi^2$ and not $\xi$ is that in the next section, notation will be easier to remember. One can check that $\Phi(a)\rho(g)=\Phi(a^g)$, so we have a nice module isomorphic to $K'$. Since $B/T'\cong K'\times K'$, we have another description of the module $B/T'$. Also, its generators are given by
$$
\begin{array}{ccccccc}
x^3 & \overset{J}\longmapsto &(a_1+a_2+a_3+a_6,a_1+a_2+a_3+a_4,\cdot)&\overset{\mod T'}\equiv &(a_4+a_5,a_5+a_6)&\overset{\Phi\times \Phi}\longmapsto\\
&\overset{\Phi\times \Phi}\longmapsto & (0,1,0,\xi^2)\\
y^2 & \overset{J}\longmapsto &(a_3+a_6,a_3+a_6,\cdot )&\overset{\mod T'}\equiv & (a_1+a_2+a_4+a_5,a_1+a_2+a_4+a_5)&\overset{\Phi\times \Phi}\longmapsto \\
&\overset{\Phi\times \Phi}\longmapsto & (1,1,1,1)\\
\left[x,x^y\right] & \overset{J}\longmapsto &(a_2+a_3+a_5+a_6,a_1+a_3+a_4+a_6,\cdot)&\overset{\mod T'}\equiv& (a_2+a_3+a_5+a_6,a_1+a_3+a_4+a_6,)&\overset{\Phi\times \Phi}\longmapsto\\
&\overset{\Phi\times \Phi}\longmapsto & (\xi^2,\xi^2,\xi,\xi)\\
\end{array}
$$
\noindent We are now going to give a more explicit description of the dessins $E_i$, with equations. We are going to give more explicit constructions of the dessin $B$ and then describe $E_0$ as a covering of this curve.

So let us look at the dessin $B$, which is a regular dessin with cartographic group $F$. If we look back on the dessins we used to construct the dessin $D_0$, we had that $D_0$ was a two sheet covering of the dessin to the left in figure \ref{des6}, by a map that we will call $\pi$. This means that the regular cover of that dessin is itself covered by the regular cover of $D_0$, which suggests that it might be the group $B$. As a matter of fact it is, for if we look at the permutations induced by $x,y$, they are precisely, given the right numbering of the edges, $(1,2,3)$ and $(1,4)(2,5)(3,6)$. This means that the dessin we are looking for is the regular cover of this dessin. In $B$, $x$ has order $3$, $y$ has order $2$, and $xy$ has order $6$. This means two things: first, that $B$ has genus 1, for $\frac{1}{2}+\frac{1}{3}+\frac{1}{6}=1$, and also the only points that are ramified are $\{\rd,\xi\rd,\xi^2\rd\}$. This suggests taking the function $(X-\rd)(X-\xi\rd)(X-\xi^2\rd)=X^3-2$ which has order 1 at each of these points, and considering the Fermat curve $X^3+Y^3=2$, with the map $(X,Y)\longmapsto X$. This map has degree 3 and it is only ramified over these three points (note that the point at infinity, $(1:0)$, has three preimages, namely $(1:1:0)$, $(1:\xi:0)$ and $(1:\xi^2:0)$, so it is unramified).

We have the Belyi pair, which we will also call $B$, given by $$\left(\{(X:Y:Z)\in \P^2:X^3+Y^3=2Z^3 \},(X:Y:Z)\longmapsto Z^6-(Z^3-X^3)^2\right)$$
\noindent We can see that this map is regular: the Fermat curve has two automorphisms, namely $(x,y)\longmapsto (\xi x,y)$, and $(x,y)\longmapsto (y,x)$, both of which preserve the Belyi map (since $1-x^3=y^3-1$), and they generate the group $F$. Since the degree of the pair is $18$, and we have 18 automorphisms, it is the whole automorphism group, and the dessin is regular. This means that it is indeed the regular cover of the dessin in figure \ref{des6}, and it corresponds to the group $B$, as we desired.

Our objective now is to draw the dessin $B$. It is easy to do, since we know its cartographic group, which is $F$, and its edges can be mapped to the cosets of $B$. If we see $\Fd/B$ as a subgroup of $S_6$ (via the map $x\longmapsto (1,2,3)$ and $y\longmapsto (1,4)(2,5)(3,6)$), we can draw the dessin, as it is seen in figure \ref{Fermat}.

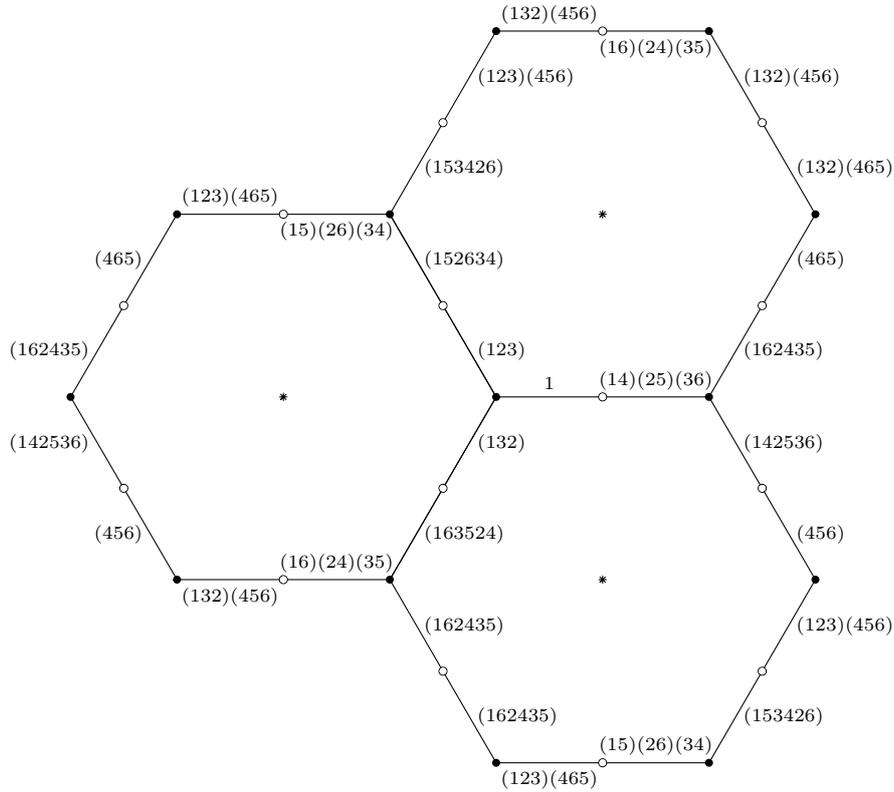
\begin{figure}[h!]
\centering
\begin{tikzpicture}[line cap=round,line join=round,>=triangle 45,x=0.7cm,y=0.7cm]
\clip(-10,-8) rectangle (8,8);
\begin{scriptsize}
\draw (0,0)-- (4,0)-- (6,3.46)-- (4,6.93)-- (0,6.93)-- (-2,3.46)-- (0,0);
\draw (4,0)-- (6,-3.46);
\draw (6,-3.46)-- (4,-6.93);
\draw (4,-6.93)-- (0,-6.93);
\draw (0,-6.93)-- (-2,-3.46);
\draw (-2,-3.46)-- (0,0);
\draw (0,0)-- (-2,-3.46);
\draw (-2,-3.46)-- (-6,-3.46);
\draw (-6,-3.46)-- (-8,0);
\draw (-8,0)-- (-6,3.46);
\draw (-6,3.46)-- (-2,3.46);
\draw (-2,3.46)-- (0,0);
\draw (2,3.46)-- ++(-1.0pt,-1.0pt) -- ++(2.0pt,2.0pt) ++(-2.0pt,0) -- ++(2.0pt,-2.0pt) ++ (-1.0pt,-0.4pt) -- ++(0,2.8pt) ++(-1.4pt,-1.4pt)-- ++(2.8pt,0);
\draw (2,-3.46)-- ++(-1.0pt,-1.0pt) -- ++(2.0pt,2.0pt) ++(-2.0pt,0) -- ++(2.0pt,-2.0pt) ++ (-1.0pt,-0.4pt) -- ++(0,2.8pt) ++(-1.4pt,-1.4pt)-- ++(2.8pt,0);
\draw (-4,0)-- ++(-1.0pt,-1.0pt) -- ++(2.0pt,2.0pt) ++(-2.0pt,0) -- ++(2.0pt,-2.0pt) ++ (-1.0pt,-0.4pt) -- ++(0,2.8pt) ++(-1.4pt,-1.4pt)-- ++(2.8pt,0);
\puntoblanco{(5,5.195)}
\puntoblanco{(5,1.73)}
\puntoblanco{(5,-1.73)}
\puntoblanco{(5,-5.195)}
\puntoblanco{(2,6.93)}
\puntoblanco{(2,-6.93)}
\puntoblanco{(2,0)}
\puntoblanco{(-1,5.195)}
\puntoblanco{(-1,-5.195)}
\puntoblanco{(-1,1.73)}
\puntoblanco{(-1,-1.73)}
\puntoblanco{(-4,-3.46)}
\puntoblanco{(-7,1.73)}
\puntoblanco{(-7,-1.73)}
\puntoblanco{(-4,3.46)}
\puntonegro{(0,0)}
\puntonegro{(4,0)}
\puntonegro{(6,3.46)}
\puntonegro{(4,6.93)}
\puntonegro{(0,6.93)}
\puntonegro{(-2,3.46)}
\puntonegro{(6,-3.46)}
\puntonegro{(4,-6.93)}
\puntonegro{(0,-6.93)}
\puntonegro{(-2,-3.46)}
\puntonegro{(-6,-3.46)}
\puntonegro{(-6,3.46)}
\puntonegro{(-8,0)}
\draw (1,0) node[anchor=south] {$1$};
\draw (3,0) node[anchor=south] {$(14)(25)(36)$};
\draw (4.5,-0.87) node[anchor=west] {$(142536)$};
\draw (5.5,-2.6) node[anchor=west] {$(456)$};
\draw (5.5,4.33) node[anchor=west] {$(132)(465)$};
\draw (4.5,0.87) node[anchor=west] {$(162435)$};
\draw (5.5,2.6) node[anchor=west] {$(465)$};
\draw (5.5,-4.33) node[anchor=west] {$(123)(456)$};
\draw (4.5,6.06) node[anchor=west] {$(132)(456)$};
\draw (4.5,-6.06) node[anchor=west] {$(153426)$};
\draw (3,6.93) node[anchor=north] {$(16)(24)(35)$};
\draw (1,6.93) node[anchor=south] {$(132)(456)$};
\draw (3,-6.93) node[anchor=south] {$(15)(26)(34)$};
\draw (1,-6.93) node[anchor=north] {$(123)(465)$};
\draw (-0.5,6.06) node[anchor=west] {$(123)(456)$};
\draw (-0.5,0.87) node[anchor=west] {$(123)$};
\draw (-0.5,-6.06) node[anchor=west] {$(162435)$};
\draw (-1.5,4.33) node[anchor=west] {$(153426)$};
\draw (-1.5,-4.33) node[anchor=west] {$(162435)$};
\draw (-0.5,-0.87) node[anchor=west] {$(132)$};
\draw (-1.5,2.6) node[anchor=west] {$(152634)$};
\draw (-1.5,-2.6) node[anchor=west] {$(163524)$};
\draw (-3,3.46) node[anchor=north] {$(15)(26)(34)$};
\draw (-3,-3.46) node[anchor=south] {$(16)(24)(35)$};
\draw (-5,3.46) node[anchor=south] {$(123)(465)$};
\draw (-5,-3.46) node[anchor=north] {$(132)(456)$};
\draw (-6.5,2.6) node[anchor=east] {$(465)$};
\draw (-6.5,-2.6) node[anchor=east] {$(456)$};
\draw (-7.5,0.87) node[anchor=east] {$(162435)$};
\draw (-7.5,-0.87) node[anchor=east] {$(142536)$};
\end{scriptsize}
\end{tikzpicture}
\caption{The dessin given by the group $B$, where the elements of $\Fd/B$ are seen as permutations.}\label{Fermat}
\end{figure}
\noindent In the preimage of $\{0,1,\infty\}$, there are 18 points, namely
$$
\begin{array}{ccc}
(0,\rd)& (1,1)& (1:-1:0)\\
(0,\xi\rd)& (1,\xi)& (1:-\xi:0)\\
(0,\xi^2\rd)& (1,\xi^2)& (1:-\xi^2:0)\\
(\rd,0)& (\xi,1)\\
(\xi\rd,0)& (\xi,\xi)\\
(\xi^2\rd,0)& (\xi,\xi^2)\\
& (\xi^2,1)\\
& (\xi^2,\xi)\\
& (\xi^2,\xi^2)\\
\end{array}
$$
\begin{figure}[h!]
\centering
\includegraphics[scale=1.3]{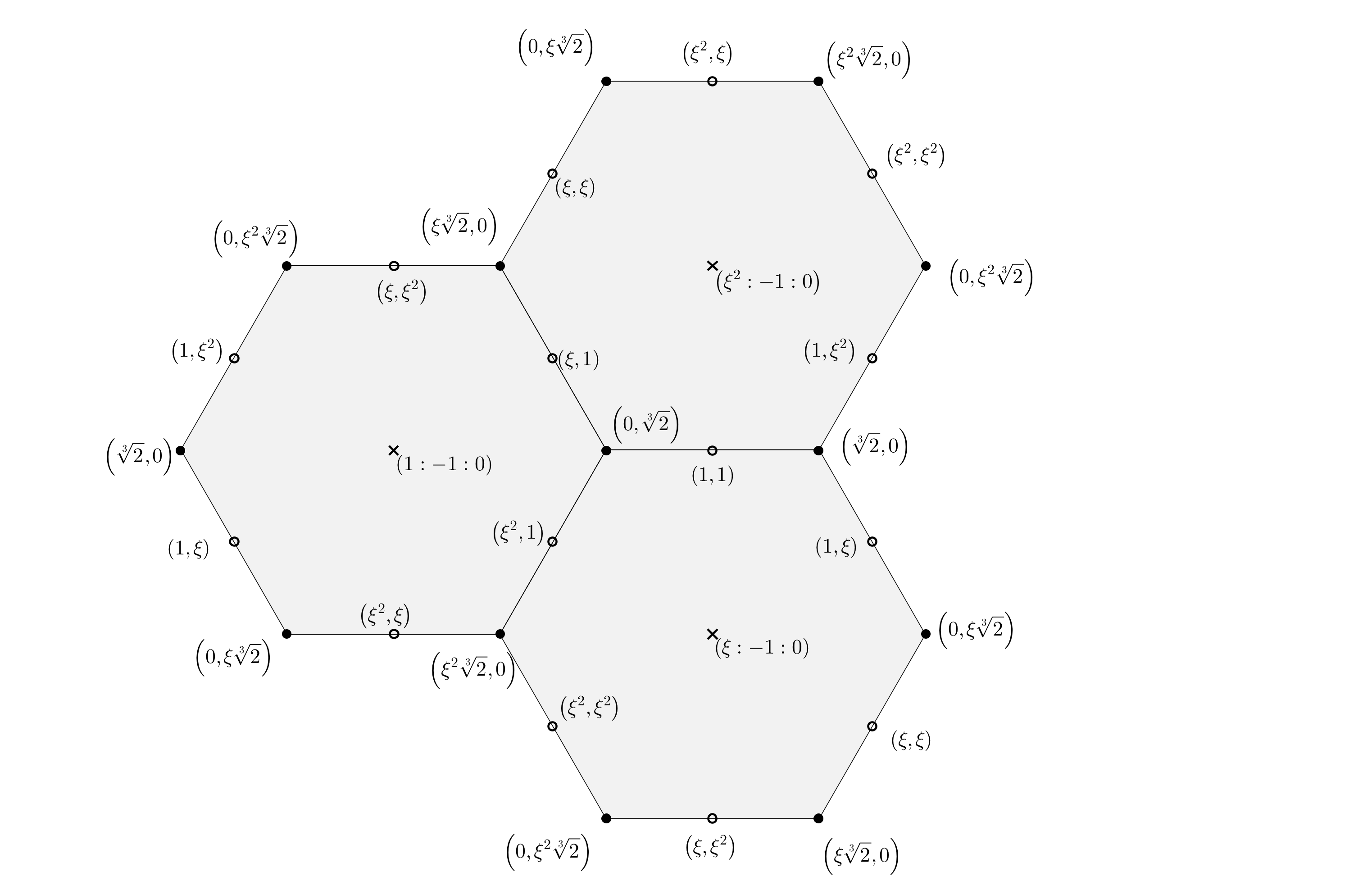}
\caption{The points on the dessin $B$}\label{fermatpuntos}
\end{figure}
\noindent Let us see how these points are distributed in the dessin: To do this, we can use its automorphism group, since we know it is $F=\langle x,y:y^2=x^3=[x,x^y]=1\rangle$, but on the other hand, the generator $x$ can be seen as the map $(X,Y)\longmapsto (\xi X,Y)$, and $y$ can be seen as the map $(X,Y)\longmapsto (Y,X)$. Then, the map $\pi$ is the quotient by the action of $x^y$, which is $(X,Y)\longmapsto (X,\xi Y)$.

In order to find out how the branch points are placed on the curve, we can use the action of the automorphism group on the dessin, which, since the dessin is regular, is the cartographic group acting on the left. Therefore, we can use figure \ref{Fermat} to see what the action is. Since the action is transitive on the edges and on vertices of the same color, we can choose any edge to contain the base point. We have chosen it to be the edge connecting $(0,\rd)$ and $(1,1)$ (which is indeed contained in the preimage of $[0,1]$, since $\pi\left(\{(t,\sqrt[3]{2-t^3}):t\in [0,1]\}\right)=[0,1]$). If we place the points $(0,\rd)$ and $(1,1)$ on the base edge, the automorphism group, seen as $x(X,Y)=(\xi X,Y);y(X,Y)=(Y,X)$, gives the placement of the rest of the points, as seen in figure \ref{fermatpuntos}. (Note that, in figure \ref{Fermat}, the action of $x=(123)$ is given by 120º rotation around the middle point and the action of $y=(14)(25)(36)$ is given by 180º rotation around the white point to the right of the middle point).

Using the dessin $B$, we are going to give equations for the dessin $E_0$, which is a covering of $B$ of $2^4$ sheets. We are going to construct a dessin of degree two over $B$, and the regular dessin that covers it will be the dessin $E_0$. To produce one of these dessins, we just need an index 2 subgroup of $B$ containing $E_0$. Since $B/E_0\cong K \cong \F_4^2$, we can take the subgroup $H_0$ of this quotient given by $\{(a,b)\in \F_4^2:b\in \F_2 \}$, which is an index 2 subgroup. To this subgroup corresponds a dessin, which can be constructed as a covering of $B$. To do this, we will use that an index 2 subgroup of $B$ (which is a free group on 19 generators, since it is the fundamental group of a curve of genus 1 with 18 points removed) can be determined by which of its generators are in the subgroup. It is straightforward to give, for a point on the Fermat curve, a loop going around it in $\Fd$, and we can compute its image in $K$, since we did this in the previous section. Recall that $x^3$ in $B/E_0\cap E_1$ maps to $(a_4+a_5,a_5+a_6)$, which modulo $E_0$ is $a_4+a_5$, and the isomorphism we gave maps this to $(0,1)$. Similarly, $y^2$ maps to $(a_1+a_2+a_4+a_5,a_1+a_2+a_4+a_5)$, and to  $(1,1)$. Also, we need the image of $(xy)^6$: in the cartographic groups $\CC_0$, $\CC_1$ and $\CC_2$, its image is $a_1+a_2+a_3+a_4+a_5+a_6$, so it is in the center, and it is contained in $E_0\cap E_1\cap E_2$. The result is the following:
$$\begin{array}{ccc|ccc}
\mathrm{Point} & \mathrm{Generator} & \mathrm{Image\ in\ }\F_4^2&\mathrm{Point} & \mathrm{Generator} & \mathrm{Image\ in\ }\F_4^2\\
\hline 
(0,\rd)      & x^3             & (0,1)     & (1,1)         & y^2             & (1,1) \\
(0,\xi\rd)   & (x^3)^{(x^y)^2} & (0,\xi)   & (\xi,1)       & (y^2)^{x^2}     & (\xi,1) \\
(0,\xi^2\rd) & (x^3)^{x^y}     & (0,\xi^2) & (\xi^2,1)     & (y^2)^x         & (\xi^2,1)\\
(\rd,0)      & (x^3)^{y}       & (1,0)     & (1,\xi)       & (y^2)^{x^2y}    & (1,\xi) \\
(\xi\rd,0)   & (x^3)^{yx^2}    & (\xi,0)   & (\xi,\xi)     & (y^2)^{x^2yx^2} & (\xi,\xi) \\
(\xi^2\rd,0) & (x^3)^{yx}      & (\xi^2,0) & (\xi^2,\xi)   & (y^2)^{x^2yx}   & (\xi^2,\xi) \\
(1:-1:0)     & (xy)^6          & (0,0)     & (1,\xi^2)     & (y^2)^{xy}      & (1,\xi^2)\\
(\xi:-1:0)   & ((xy)^6)^{x^2}  & (0,0)     &(\xi,\xi^2)    & (y^2)^{xyx^2}   & (\xi,\xi^2)\\
(\xi^2:-1:0) & ((xy)^6)^x      & (0,0)     & (\xi^2,\xi^2) & (y^2)^{xyx}     & (\xi^2,\xi^2)\\
- & y^{-1}x^{-1}yxyxyx^{-1} & (1,1) & - & y^{-1}x^{-1}y^{-1}x^{-1}y^{-1}xyx & (\xi^2,1)
\end{array}
$$
\noindent The last two generators correspond to loops generating the fundamental group of the torus, and they are mapped in the permutation groups to $(a_3+a_6,a_1+a_2+a_3+a_5,a_1+a_4+a_5+a_6)$ and $(a_1+a_6,a_1+a_2+a_3+a_5,a_2+a_4+a_5+a_6)$, which correspond to $(1,1,0,\xi)$ and $(\xi^2,1,0,\xi)$. The generators which are not in the subgroup $H$ are then $\{(0,\xi\rd),(0,\xi^2\rd),(1,\xi),(\xi,\xi),(\xi^2,\xi),(1,\xi^2),(\xi,\xi^2),(\xi^2,\xi^2)\}$.

Let us look at the dessin's fields of functions. $B$ has $\QQ(B)=\QQ(X)[Y]/(X^3+Y^3-2)$ as field of functions (and as a dessin, it is the extension $\QQ(t=1-(1-X^3)^2)\subset \QQ(B)$). The field of functions of the dessin $H$ is an extension of degree 2 of $\QQ(B)$, and it will therefore be generated by an element $Z$, such that $Z^2\in \QQ(B)$.

Now, we need to use the information we have on the ramification points of the map from the degree two covering of $B$, which we will call $B_0$ to find out which is the element $Z$ that generates the extension. Let us call $\phi=Z^2\in \QQ(B)$.

\begin{lema}
Let $\QQ(B)$ be the field of functions of a curve $B$, and let $\QQ(B_0)=\QQ(B)(Z)$ be the field of functions of a curve $B_0$ such that $Z^2=\phi\in \QQ(B)$. Let $\pi:B_0\longrightarrow B$ be the covering map. Then, $\pi$ ramifies over a point $P$ if and only if $\ord_P(\phi)$ is odd.
\end{lema}

\begin{proof}
Suppose $\ord_P(\phi)$ is odd. Multiplying by the square of a uniformizing parameter for $P$, we can assume that $\ord_P(\phi)=1$. Now, if $\nu$ is a valuation extending $\ord_P$ to $\QQ(B_0)$, then we must have that $2\nu(Z)=\nu(Z^2)=\nu(\phi)$. Since $\nu(\phi)$ must be positive, it must equal at least 2. But, since the degree of the map is 2, there must be only one extension of the valuation, with ramification index 2.

Now, if $\ord_P(\phi)$ is even, we can assume that it is 0, doing the same as in the previous case. For a valuation $\nu$ extending $\ord_P(\phi)$, we must have $\nu(Z)=0$. Now, let $\pi$ be a uniformizing parameter for $\ord_P$. Let $\phi(P)=a^2$, and let $\O_P$ be the local ring for $P$. Then, $\O_P[Z]$ is a ring extending $\O_P$ and it has two maximal ideals, namely $(Z-a,\pi)$ and $(Z+a,\pi)$. Therefore, the valuation $\ord_P$ extends in at least two different ways, and since the maximum number of extensions it can have is 2, we must have that there are two of them, each unramified.
\end{proof}

\noindent So, we have that the covering is of the form $\QQ(B)\subset \QQ(B)[Z]$, where $Z^2=\phi\in \QQ(B)$. For the map to ramify at the list of points we want it to ramify, $\phi$ must have odd order at precisely the points $\{(0,\xi\rd),(0,\xi^2\rd),(1,\xi),\\(\xi,\xi),(\xi^2,\xi),(1,\xi^2),(\xi,\xi^2),(\xi^2,\xi^2)\}$. One such function is
$$
\phi=\frac{(Y-\xi)(Y-\xi^2)}{(Y-\xi\rd)(Y-\xi^2\rd)}
$$
The problem is, this is not the only function with this property, not even modulo multiplication by squares. If we take the functions
$$
\psi_1=\frac{X+Y-2}{X+Y-2\xi};\psi_2=\frac{X+Y-2\xi}{X+Y-2\xi^2};\psi_3=\psi_2/\psi_1=\frac{X+Y-2\xi}{X+Y-2\xi^2}
$$
We have that their divisors of zeroes and poles are the following:
$$(\psi_1)=2(1,1)-2(\xi,\xi);(\psi_2)=2(1,1)-2(\xi^2,\xi^2);(\psi_3)=2(\xi,\xi)-2(\xi^2,\xi^2)$$
Since their order at every point is even, but they are not squares (because they would have to be squares of functions of degree 1, and there are no such functions on an elliptic curve), the functions $\phi\psi_i$ are also functions with odd order at the prescribed points, which are different from $\phi$ and from each other modulo squares. Let us prove that these are the only functions with even multiplicity at every point, modulo squares.

Suppose a function $\psi$ has divisor of zeros and poles $\sum 2n_iP_i$. If we denote the sum in the elliptic curve by $\oplus$, we know that this means that $\bigoplus 2n_iP_i=0$ (see, for example, \cite{Miranda}). Now, if $\psi$ is not a square, this means that $\bigoplus n_iP_i\neq 0$, otherwise $\bigoplus n_iP_i$ would be a principal divisor and a function with such divisor would be the square root of $\psi$ except for a multiplicative constant. Therefore, $\bigoplus n_iP_i$ is a point of order 2, of which there are three. Now, if we take $(0,\rd)$ as the origin for our curve, we have that the divisors of the $\psi_i$'s halved are all different in the group of the curve, since $(1,1)-(\xi,\xi)=(\xi\sqrt[3]{4},-\rd)$, $(1,1)-(\xi^2,\xi^2)=(\xi^2\sqrt[3]{4},-\rd)$ and $(\xi,\xi)-(\xi^2,\xi^2)=(\sqrt[3]{4},-\rd)$. Therefore, we have that $\bigoplus n_iP_i=\frac{1}{2}(\psi_j)$, for some $j$. From this follows that $\psi\psi_j^{-1}$ is a square, since it is the square of a function with divisor $\oplus n_iP_i-\frac{1}{2}(\psi_j)$, which is $0$ in the Picard group, and therefore it is a principal divisor.

Thus we are left with 4 possibilities for the function we are looking for. Let us take a look at two modules. The first one is $B/B^2$ (where $B^2=\langle g^2:g\in B\rangle$). Since $B$ is freely generated by 18 elements, $B/B^2\cong (\zn{2})^{18}$. $B^2$ is a characteristic subgroup of $B$, and therefore it is fixed by every automorphism of $B$, for example, by conjugation by an element of $\Fd$. Furthermore, since $B/B^2$ is commutative, every element of $B$ acts as the identity on the quotient $B/B^2$, so it is an $\Fd/B$-module, or an $F$-module.

The other module consists of the functions $f\in \QQ(B)^\times$ such that $\QQ(B)\left[\sqrt f\right]/\QQ(B)$ is ramified over at most the 18 points on the dessin. We take these functions with the product as a group law, and we quotient out by $(\QQ(B))^2$, so $f\sim f'$ if and only if $\QQ(B)\left[\sqrt f\right]$ and $\QQ(B)\left[\sqrt{f'}\right]$ are the same fields of functions. This group, which we will call $R$ obviously has exponent 2. Now, we can see that $B/B^2$ and $R$ are dual to one another.

If we take the field $\mathcal{K}/\QQ(t)$, which is the biggest extension unramified outside of $\{0,1,\infty\}$, we can see $B$ as $\mathrm{Gal}(\mathcal K/\QQ(B))$. Now, $B^2$ is the intersection of all the index 2 subgroups of $B$, and therefore the field of functions of the dessin corresponding to the group $B^2$ will be the field $\QQ(B^2)$ which is the field generated by all the field extensions of $\QQ(B)$ of degree 2 unramified outside of $\{0,1,\infty\}$.

Now, if we take an element $g\in B/B^2=\mathrm{Gal}(\QQ(B^2)/\QQ(B))$, and a function class $f\in R$, we can define a pairing with image $\{\pm 1\}$ in the following way: choose a square root $h$ of $f$, which generates a degree 2 extension of $\QQ(B)$ unramified outside of the 18 points, and therefore contained in $\QQ(B^2)$. If $h^g=h$, then define $\langle f,g\rangle=1$, and if $h^g=-h$, define $\langle f,g\rangle=-1$. In other words, $\langle f,g\rangle=\frac{h^g}{h}$. It is clear that this definition does not depend upon the choice of the square root, since $(-h)^g=-h^g$. Also, the pairing is bilinear: if one takes two functions $f_1$ and $f_2$ with square roots $h_1$ and $h_2$, respectively, then $\langle f_1f_2,g \rangle=\frac{(h_1h_2)^g}{h_1h_2}=\frac{h_1^g}{h_1}\frac{h_2^g}{h_2}=\langle f_1,g\rangle \langle f_2,g\rangle$, and if one takes two automorphisms $g_1$ and $g_2$, then
$$\langle f,g_1g_2\rangle=
\frac{h^{g_1g_2}}{h}
=\frac{h^{g_1g_2}}{h^{g_2}}\frac{h^{g_2}}{h}
=\left(\frac{h^{g_1}}{h}\right)^{g_2}\frac{h^{g_2}}{h}=\langle f,g_1\rangle \langle f,g_2\rangle$$
Because $\langle f,g_1\rangle$ is fixed by $g$. Finally, since the square roots of these functions generate the field $\QQ(B^2)$, $\langle f,g\rangle=1$ for all $f\in R$ implies that $g\in B^2$, and dually, if the square root of a function is fixed by every element of $\mathrm{Gal}(\QQ(B^2)/\QQ(B))$, then it must lie in $\QQ(B)$, and the function is a square. Overall, the pairing is perfect, which means that $R$ and $B/B^2$ are dual as $\Zn{2}$-vector spaces. In particular, $R\cong (\zn{2})^{18}$.

Now, let us look at the $F$-module structure: $B/B^2$ is a left $F$-module with the action by conjugation and $R$ is a left $F$-module with the action of $F=\Gal(\QQ(B)/\QQ(t))$. Now, let $\sigma\in F,f\in R,g\in B/B^2$. If $h$ is a square root of $f$, then, for any $\widetilde \sigma \in \Fd$ such that its class in $\Fd/B=F$ is $\sigma$, $\widetilde\sigma(h)$ will be a square root of $f^{\wt \sigma}=f^\sigma$. Therefore,
$$
\langle f^\sigma,g\rangle=\frac{h^{\wt \sigma g}}{h^{\wt \sigma}}
$$
Now, since the right hand side equals $\pm 1$, it is fixed by $\widetilde \sigma^{-1}$, so
$$
\langle f^\sigma,g\rangle=\left(\frac{h^{\wt \sigma g}}{h^{\wt \sigma}}\right)^{\widetilde\sigma^{-1}}=\frac{h^{g^{\widetilde\sigma^{-1}}}}{h}=\langle f,g^{\sigma^{-1}}\rangle
$$
Now, from this fact we can obtain a lot of information from the submodule structure of the modules: Since they are dual, to each subspace $S<B/B^2$ we can associate to it its dual subspace $S^*=\{g\in G:\langle S,g\rangle=1\}$, and vice versa. This correspondence is inclusion-reversing and it satisfies the identity $S^{**}=S$, and it is also the Galois correspondence: to a subspace $S<R$ we associate the subgroup that fixes $\QQ(B)[\{\sqrt f:f\in R \}]$.

Now, if we take a subspace $S<B/B^2$, and an element $\sigma\in F$, we will have, from the previous identity, that $(S^\sigma)^*=(S^*)^\sigma$. If we denote by $\overline S$ the module generated by $S$, and $\mathring S$ to be the largest submodule inside $S$, we will have
$$
\left(\overline S\right)^*=\left(\sum_{\sigma \in F} S^\sigma\right)^*=\bigcap_{\sigma \in F} (S^\sigma)^*=\bigcap_{\sigma \in F} (S^*)^\sigma=\mathring{(S^*)}
$$
And also
$$
\left(\mathring S\right)^*=\left(\bigcap_{\sigma \in F} S^\sigma\right)^*=\sum_{\sigma \in F} (S^\sigma)^*=\sum_{\sigma \in F} (S^*)^\sigma=\overline{(S^*)}$$
And these identities work both for $S\subset B/B^2$ and $S\subset R$. Now, we are looking for a group $B_0$ of $B$ in particular. This group has a special property: with the list of generators in our hands, we can check that it is invariant under conjugation by $x$. This means that $x$ must fix the function whose square root generates the field, since the subspace duality is the Galois correspondence. Of the four functions $\phi,\phi\psi_1,\phi\psi_2,\phi\psi_3$, only $\phi$ is fixed by $x$ (modulo squares). Therefore, the field we are looking for must be generated by a function $Z$ such that
$$
Z^2=\phi=\frac{(Y-\xi)(Y-\xi^2)}{(Y-\xi\rd)(Y-\xi^2\rd)}
$$
Finally, the field $\QQ(E_0)$ will be the normal closure of the extension $\QQ(B_0)/\QQ(t)$. This means adding all the Galois conjugates to $Z^2$ (or taking the module generated by $\phi$). Since this module has dimension 4, it will be generated by the conjugates of $\phi$. A basis is given by
$$
\left\{
\phi_1=\phi=\frac{(Y-\xi)(Y-\xi^2)}{(Y-\xi\rd)(Y-\xi^2\rd)};\phi_2=\phi^{x^y}=\frac{(Y-1)(Y-\xi)}{(Y-\rd)(Y-\xi\rd)};\right.
$$ $$
\left.
\phi_3=\phi^y=\frac{(X-\xi)(X-\xi^2)}{(X-\xi\rd)(X-\xi^2\rd)};\phi_4=\phi^{x^yy}=\frac{(X-1)(X-\xi)}{(X-\rd)(X-\xi\rd)}
\right\}
$$
Now, we can explicitly give the field of functions of the dessin $E_0$, which is a regular dessin whose field of moduli is not abelian. It is
$$
\QQ(E_0)=\frac{\QQ(t)[X,Y,Z_1,Z_2,Z_3,Z_4]}{\left(
t-(1-(X^3-1)^2),X^3+Y^3-2,Z_1^2-\phi_1,Z_2^2-\phi_2,Z_3^2-\phi_3,Z_4^2-\phi_4
\right)}$$
The conjugate dessins are obtained by the action of $\gal$: this means that they are built by taking the action of $\gal$ on the coefficients of the $\phi_i$'s, which are the only functions involved with irrational coefficients. We can draw pictures of the ramification points of $\QQ(B_0)/\QQ(B)$, in order to visualize why the dessin $E_0$ changes under the Galois action. In figure \ref{fermatpuntos} we can mark the points where the extension is ramified, but $E_0$ is the common cover (the intersection) of all six dessins conjugate to $B_0$, which means having $\Fd$ act on these points. For example, the action of $y$ is a 180º rotation around $(1,1)$, and it is not clear if we saw a picture of the rotated points whether they are conjugate to the same dessin. However, if we draw the universal cover of the dessin, which is $\C$, since the dessin has genus 1, we can visualize it a lot better, like in figure \ref{fer}.

\begin{figure}[h!]
\centering
\includegraphics[width=0.8\textwidth]{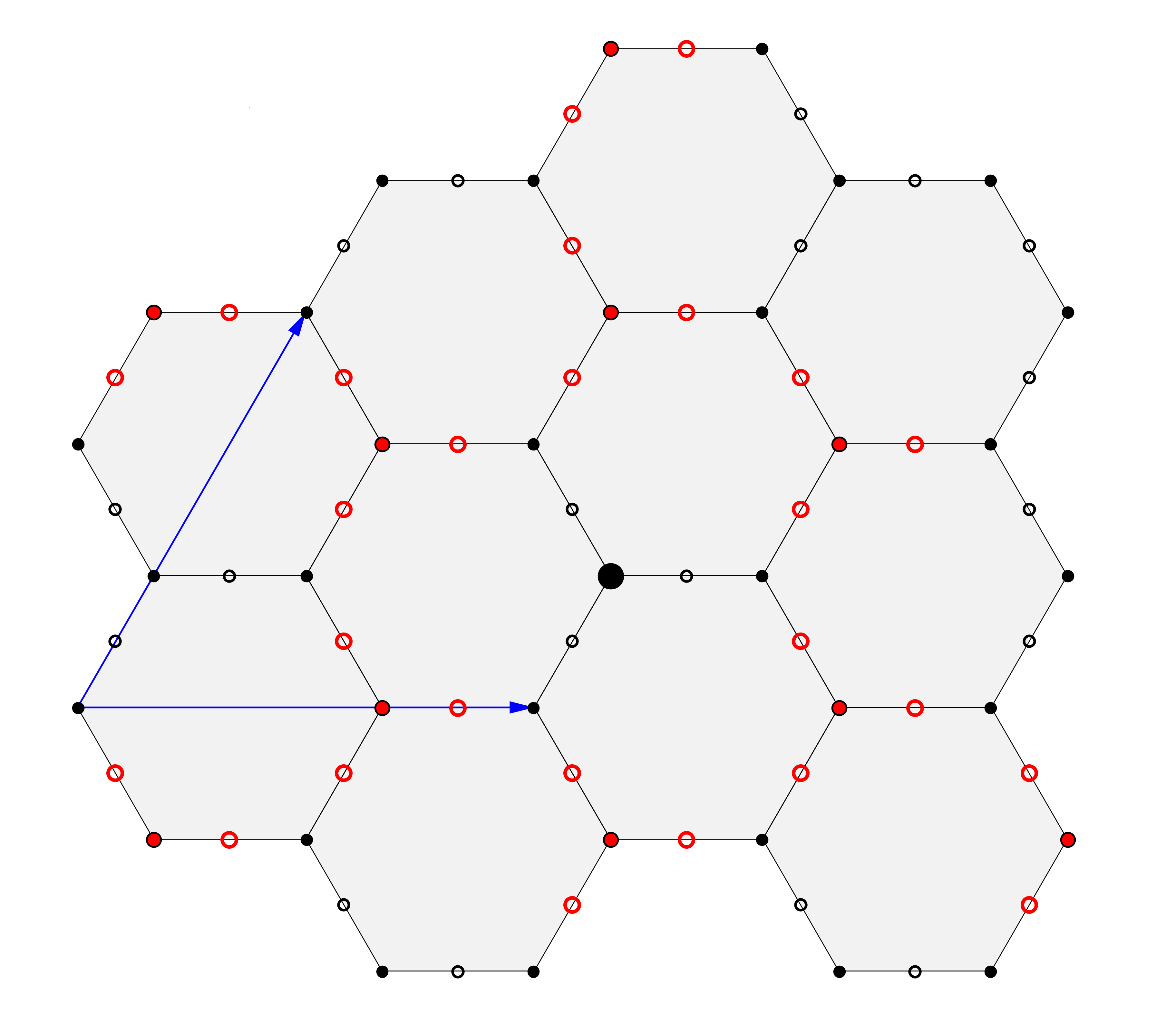}
\captionof{figure}{The universal cover of $B$, with the points where $B_0$ ramifies marked in red.}\label{fer}
\end{figure}

The fundamental domain is given by the two vectors, and the red points are the ones over which $B_0$ ramifies. Now, if we imagine that $(0,\rd)$ is the biggest black point, the three hexagons around it make up the dessin as we have drawn it in figure \ref{fermatpuntos}. Now, the action of an element of $\Fd$, such as $y$, can be lifted to this cover, as a rotation: if we take the $y$ rotation, we are left with essentially the same image, and the effect is the same as changing the fundamental domain. This way, we can produce every dessin conjugate to $B_0$, and they are all within this image.

However, if we take the Galois action on the dessin $B_0$, the points over which it ramifies are also affected by the same action, and we are left with the pictures in figure \ref{ferreticulo2}.

\begin{figure}[h!]
\centering
\begin{center}
\includegraphics[width=0.4\textwidth]{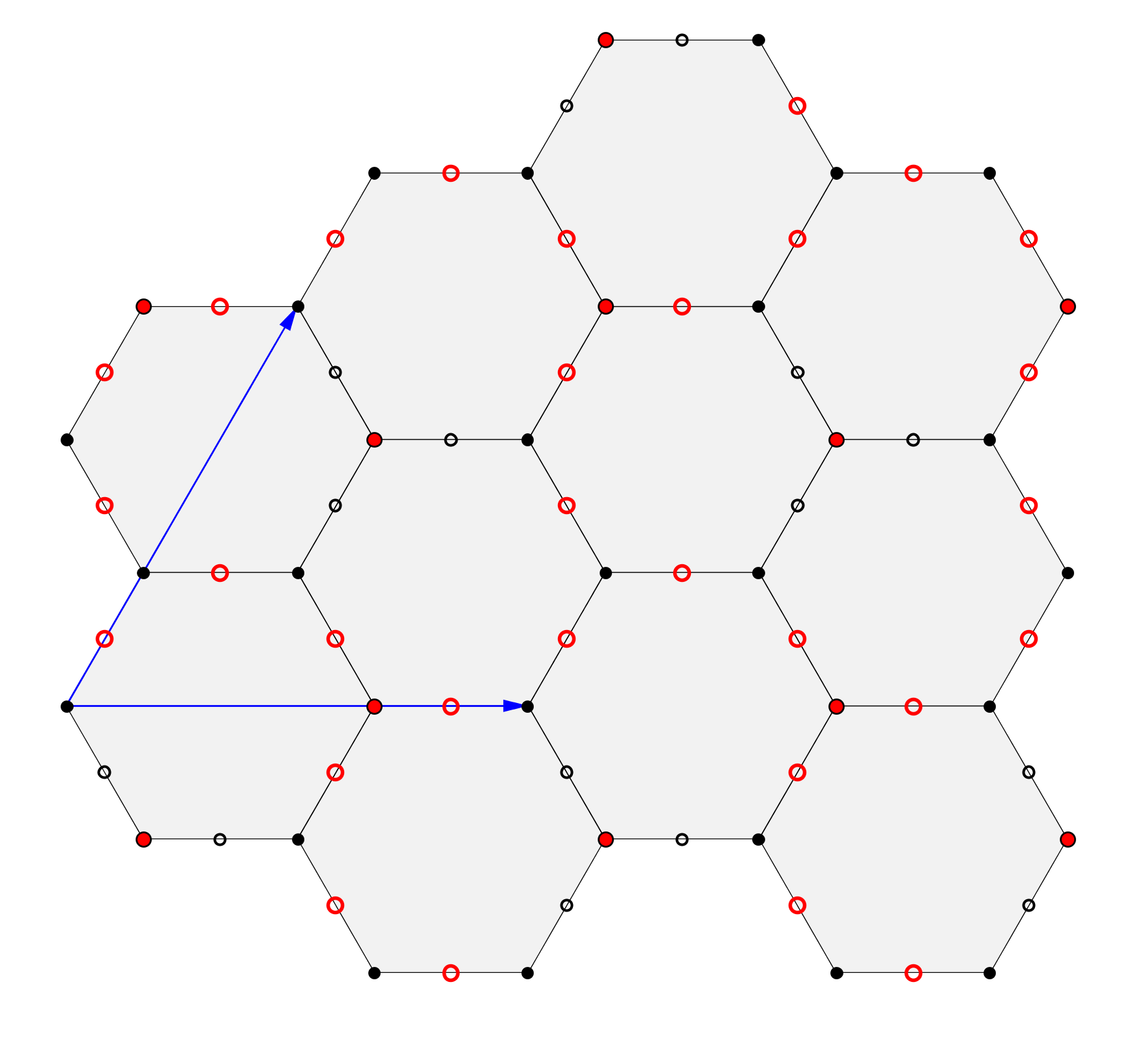}\includegraphics[width=0.4\textwidth]{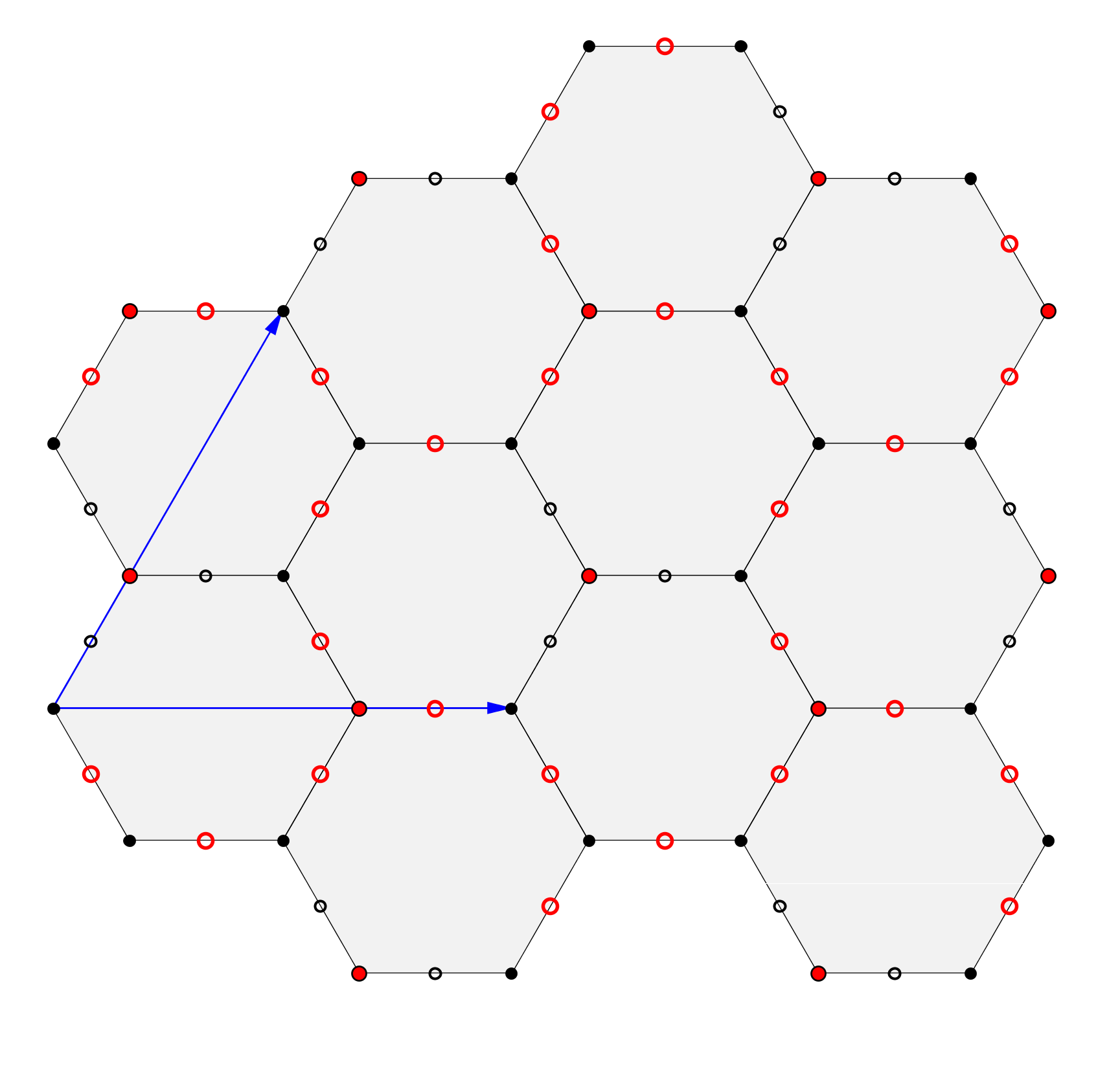}
\caption{The universal cover of $B$, with the points where $B_0^\sigma$ ramifies, where $\sigma\in \Gal(\Q(\rd,\xi)/\Q)$.}\label{ferreticulo2}
\end{center}
\end{figure}


These three are clearly not isomorphic to each other: there is no map that will preserve the red points. This just proves the fact that $B_0$ is not fixed by the Galois action. It could be that its core, $\cap_{g\in \Fd} B_0^g$ was fixed, but it is not the case since we have proved it.

\section{The field of moduli of the underlying curve}

We have constructed the regular dessin given by the subgroup $E_0<\Fd$. This dessin corresponds to some Belyi pair $(C,f)$. We have also proven that when $\sigma\in \gal$ doesn't fix $\rd$, the Belyi pair $(C^\sigma,f^\sigma)$ is not isomorphic to $(C,f)$. However, it could be that the curves are isomorphic, i.e. $C^\sigma\cong C$. In this section, we are going to prove that this is not the case, so the field of moduli of the curve is also $\Q(\rd)$.

The following is a sketch of the proof:

If the curves $C$ and $C^\sigma$ were isomorphic, then the curve $C$ would have two regular dessins of the same degree. By a theorem of Wolfart \cite{abc}, a curve with a regular dessin also has one unique regular dessin of maximum degree, which is given by the map $C\lra C/\Aut(C)$. We are going to prove that the dessin we have on $C$ is already maximal, and therefore the curve cannot have more dessins of the same degree.

The way to do this is using the result in \cite{multiply}: if a curve has a regular dessin given by the map $C\lra C/H$, where $H<\Aut(C)$, then the maximal dessin is given by doing the following: consider the dessin as a subgroup $N$ of a triangle group $\Delta_1$. Then, the maximal dessin is given by including this $\Delta_1$ in another triangle group $\Delta_2$ as a finite index subgroup. In order for the dessin obtained in this way to be regular, $N$ must be a normal subgroup of $\Delta_2$ under this inclusion. Our dessin is given in the triangle group $\Delta(6,4,6)$, which can only be included in $\Delta(6,8,2)$ (the complete list of triangle group inclusions can be found in \cite{singerman}). We prove that this inclusion doesn't make the subgroup of $\Delta(6,4,6)$ normal in $\Delta(6,8,2)$, so the dessin we have must be of maximal degree.

Let us now talk about triangle groups and dessins with types. For a more detailed treatment of the subject, one can consult Wolfart's survey \cite{abc}, section 2.2 in \cite{jule03}, or Girondo and González's  textbook \cite{GG}.

\begin{definicion}
We say that a dessin, or a Belyi pair $(C,f)$, is of \textbf{type} $(l,m,n)$ if the ramification order of every point in $f^{-1}(0)$ divides $l$, the ramification order at $f^{-1}(1)$ divides $m$ and the ramification at $f^{-1}(\infty)$ divides $n$.
\end{definicion}

\noindent In this definition, we are going to allow for some of $l,m,n$ to equal $\infty$, with the convention that every number divides $\infty$. For example, every dessin is of type $(\infty,\infty,\infty)$.

In terms of drawings, this means that the order of every black vertex divides $l$, the order of every white vertex divides $m$ and the number of sides on each face divides $2n$ (if the ramification index at a point in the preimage of $\infty$ is $e$, the corresponding face has $2e$ sides: $e$ that are oriented black-white in counterclockwise order and $e$ in the opposite direction). In terms of the monodromy, since $x$ takes the edges and makes them turn around black points, it means that $x^l=1$ in the cartographic group, and analogously, that $y^m=1$. Since the action of $z=(xy)^{-1}$ is turning around a face, it means that $z^n=1$. Recall that the cartographic group is $\Fd/\cor_{\Fd} H$, so the subgroup associated to the dessin contains the normal subgroup given by
$$
N_{(l,m,n)}=\langle\langle x^l,y^m,z^l\rangle\rangle
$$
Where the double brackets stand for the normal subgroup generated by the elements. Thus, by the correspondence between subgroups and dessins, dessins of type $(l,m,n)$ are in correspondence with subgroups of finite index $H$ such that $N_{(l,m,n)}<H<\Fd$. Since the subgroups $H$ are closed, they also contain the closure of $N_{(l,m,n)}$, which we will call $\wh{N_{(l,m,n)}}$.

The closure of a normal subgroup is also normal, so the quotient $\Fd\lra \Fd/\wh{N_{(l,m,n)}}$ is well-defined and we can give it the quotient topology, which will make it a profinite group, which we will call $\wh{\Delta(l,m,n)}$. It is easy to check that it is the profinite completion of the group
$$
\Delta(l,m,n)=F_2/N_{(l,m,n)}
$$
Therefore, dessins of type $(l,m,n)$ are in correspondence with open subgroups of $\tri$.

We are going to rely heavily on the uniformization theorem (which we have also used before, since it is used to prove that every compact Riemann surface is algebraic).

\begin{teor}[Uniformization theorem]
There are only three simply connected Riemann surfaces, up to biholomorphism. They are
\begin{itemize}
\item $\P^1$, the Riemann sphere.
\item $\C$, the complex plane.
\item $\H=\{z\in \C:\mathrm{Re}(z)>0\}$, which is biholomorphic to the unit disc in $\C$.
\end{itemize}
\end{teor}
\noindent Using this theorem, one can consider a compact Riemann surface $C$ and its universal cover, which is also a Riemann surface and it is simply connected, so it will be one in the list. Then, $C$ will be the quotient of one of these surfaces by the action of the fundamental group of $C$, which acts on the universal cover as biholomorphisms, and it does so properly and discontinuously. Studying the groups that act properly and discontinuously on these three surfaces, as it is done in \cite{GG}, leads to concluding that the universal cover of a surface is $\P^1$ if and only if the surface is the sphere, it is $\C$ if it has genus 1, and it is $\H$ if it has genus greater than 1.

If one has a regular dessin $(C,f)$ of genus at least 2, one can consider the universal cover of $C$, which, as we have said, is $\H$. Then, $C$ is the quotient of $\H$ by the action of its fundamental group $\Gamma$ (since the universal cover is a regular cover whose associated subgroup is $1$). Let $\pi$ be the projection from $\H$ to $\H/\Gamma$.

\begin{lema}
Let $C\cong \H/\Gamma$ be a curve of genus at least 2 and its universal cover. Then, $\Aut(C)$ lifts to $N(\Gamma)$, where $N$ is the normalizer of $\Gamma$ in the group of automorphisms of $\H$. That is, if $\wt \phi\in N(\Gamma)$, there exists some $\phi\in \Aut(C)$ such that $\pi\circ \wt \phi=\phi \circ \pi$, and reciprocally, if $\wt \phi\in N(G)$, then there exists some $\phi\in \Aut(C)$ such that the identity holds.

This gives an isomorphism between $\Aut(C)$ and $N(\Gamma)/\Gamma$.
\end{lema}

\begin{proof}
We are going to prove that the automorphism group of $C$, $\Aut(C)$, lifts to $\H$ inside of $N(\Gamma)$. Suppose we have an automorphism $\phi\in \Aut(C)$. We can consider $\phi\circ \pi:\H\lra C$. This map will lift to $\H$ by the lifting lemma, giving some $\wt\phi:\H\lra \H$ such that $\pi\circ \wt \phi=\phi \circ \pi$. This implies that $\wt \phi$ will be invariant by the action of $\Gamma$, because, if we take some $\gamma\in \Gamma$, then,
$$
\pi\circ\wt \phi\circ\gamma=\phi \circ\pi\circ\gamma=\phi\circ \pi=\pi \circ \wt \phi
$$
Therefore, $\wt \phi$ and $\wt \phi\circ \gamma$ are liftings of the same map, so they must differ in the base point, and there must exist some $\gamma'\in \Gamma$ that takes one base point to the other, so $\gamma'\circ \wt \phi=\wt \phi\circ \gamma$. Therefore,
$$
\wt \phi \circ \gamma \circ \wt \phi^{-1}=\gamma'\in \Gamma
$$
Which is what we wanted to prove: the lifting lies in $N(\Gamma)$.

Now for the reciprocal. Suppose $\wt \phi\in N(\Gamma)$. Then, for every $\gamma\in \Gamma$, there exists some $\gamma'\in \Gamma$ such that $\gamma'\circ \wt \phi=\wt \phi\circ \gamma$. Therefore, $\pi\circ \wt\phi$ is invariant under $\Gamma$: for any $\gamma\in \Gamma$,
$$
\pi\circ \wt\phi\circ \gamma=\pi\circ\gamma' \circ \wt\phi=\pi\circ \wt \phi
$$
Therefore, $\wt \phi$ is well-defined on the quotient $\H/\Gamma$, and it descends to some $\phi$ that will satisfy the required identity.

Note that an isomorphism $\wt \phi\in N(\Gamma)$ descends to the identity if and only if it is in $\Gamma$. Therefore, $\Aut(C)\cong N(\Gamma)/\Gamma$.
\end{proof}

\noindent This should not come as a surprise, since we have already proven that the automorphism group of a dessin is $N_{F_2}(H)/H$. By Hurwitz's automorphism theorem, the order of $\Aut(C)$ satisfies $|\Aut(C)|<84(g-1)$, and in particular, it is finite, and so is the index of $\Gamma$ in $N(\Gamma)$.

Now, consider the dessin $(C,f)$. The map $f$ induces another map $f \circ \pi:\H\lra \P^1$. In Wolfart's paper \cite{abc} and in \cite{GG}, it is proven that this map is of the form $\H\lra \H/\Delta$, where $\Delta$ is a triangle group.

The triangle groups in $\Aut(\H)$ are defined as follows: take $(l,m,n)$ such that $\frac{1}{l}+\frac{1}{m}+\frac{1}{n}<1$ (if this is not the case, one can do the same in $\C$ or $\P^1$, but we are working in the case where the genus is greater than 1). Then, construct in $\H$, with the hyperbolic metric, a triangle with angles $2\pi/l$, $2\pi/m$, $2\pi/n$. Such a triangle exists, since there are triangles with any angles provided their sum is less than $2\pi$. Take the group generated by the reflections on the sides of this triangle. If we call them $a,b,c$, then the triangle group is the index 2 subgroup generated by $ab,bc,ca$. It can be proven that this subgroup is isomorphic to the $\Delta(l,m,n)$ we have already defined, and also that every two subgroups of $\Aut(\H)$ isomorphic to the triangle group are conjugate in $\Aut(\H)$ \cite{GG}.

Suppose we have a regular dessin on a curve $C=\H/\Gamma$ of type $(l,m,n)$, such that the ramification indices are precisely $(l,m,n)$. Let us prove that the corresponding $\Delta$ is actually $\Delta(l,m,n)$: consider the unramified covers of $C$. When composed with the Belyi map on $C$, these covers also give covers of type $(l,m,n)$. Reciprocally, every cover of type $(l,m,n)$ that covers $C$ must be unramified over $C$, since, for any point $P$ and any two functions $f,g$, $e_P (f\circ g)=e_{g(P)}(f)e_P(g)$. Therefore, if we call $U$ the universal cover of $\P^1\setminus \{0,1,\infty\}$, the subgroup of $F_2$ associated to the covering $U\lra \H$ is precisely $N_{(l,m,n)}$. Also, suppose a regular dessin corresponds to a subgroup $N\tl F_2$, so that the Belyi map is $U/N\lra U/F_2$. Then, the map factors through $U/N_{(l,m,n)}$, since $N\supset N_{(l,m,n)}$. Then, we have the chain of coverings
$$
U\lra U/N_{(l,m,n)}\lra U/H \lra U/F_2\cong \P^1
$$
Also, $U/H$ descends to $U/N_{(l,m,n)}$ as the group $\o{H}=H/N_{(l,m,n)}$, and the covering is then just
$$
\H/\o{H}\lra \H/\Delta({l,m,n)}
$$
In Wolfart's paper \cite{abc}, we have the following theorem (theorem 4).

\begin{teor}
Let $C$ be a curve of genus at least 2. The following are equivalent:
\begin{enumerate}
\item $C$ has a regular dessin. We say that $C$ is \textbf{quasiplatonic}.
\item When $C$ is seen as the quotient of $\H$ by the fundamental group $\Gamma$ of $C$, there is some triangle group $\Delta$ such that $\Gamma <\Delta <N(\Gamma)$.
\item $N(\Gamma)$ is a triangle group.
\item The map $C\lra C/\Aut(C)$ is a Belyi map (so $\Aut(C)\backslash C$ is isomorphic to $\P^1$ and the map is ramified over at most three points).
\end{enumerate}
\end{teor}

\noindent Let us apply this to the curve in our example: the curve $C$ that has the dessin given by the subgroup $E_0<\Fd$. Take $\sigma \in \gal$ such that it doesn't fix $\rd$. If the conjugate curve $C^\sigma$ was isomorphic to $C$, then $C$ would have two different regular dessins: the one given by $E_0$ and the one given by $E_0^\sigma$. Both of these are given by maps $C\lra C/H$, where $H$ is a subgroup of $\Aut(C)$. By the theorem, it follows that there is a dessin of maximum degree, given by $C\lra C/\Aut(C)$, and it must be unique. Therefore, if $C^\sigma\cong C$, the curve $C$ must have another regular dessin, of greater degree then $E_0$.

We can see $C$ as $\H/\Gamma$, so that the Belyi map $f$ is $\H/\Gamma \lra \H/\Delta$, where $\Gamma<\Delta<N(\Gamma)$. Now, $(C,f)$ has type $(6,4,6)$, so the triangle group is $\Delta(6,4,6)$.

Therefore, we have that $\Gamma<\Delta(6,4,6)<N(\Gamma)$. If the dessin we are considering is not the regular dessin with the biggest degree, then $N(\Gamma)$ must strictly contain $\Delta(6,4,6)$. Also, by the theorem, it is another triangle group. Now, Singerman \cite{singerman} lists all the possible finite index inclusions between triangle groups, and \cite{multiply} gives a list with inclusions that generate every inclusion. Looking at this list, we see that the only triangle group that contains $\Delta(6,4,6)$ is $\Delta(6,8,2)$. If we give names to the generators so that
$$
\Delta(6,4,6)=\langle x,y,z|x^6=y^4=z^6=xyz=1\rangle
$$
$$
\Delta(6,8,2)=\langle \wt x,\wt y,\wt z|\wt x^6=\wt y^8=\wt z^2=\wt x\wt  y\wt z=1\rangle
$$
Then, the inclusion is given by
$$
x=\wt x$$ $$y=\wt y^2 $$ $$z=\wt z \wt x \wt z
$$
So $\Gamma\tl \Delta(6,4,6)<\Delta(6,8,2)$ in this way, and it is clear that the inclusion between the triangle groups is of index 2. Now, the regular dessin given by $E_0$ will be the maximal one, and therefore $C\not \cong C^\sigma$, if and only if $\Gamma$ is not a normal subgroup of $\Delta(6,8,2)$. We will prove that this is the case.

As we said before, $\Gamma$ is actually the group $E_0/N_{(6,4,6)}$.  Therefore, $\Gamma$ will be normal in $\Delta(6,8,2)$ if and only if $E_0/N_{(6,4,6)}$ is normal in $\Delta(6,8,2)$, with the inclusion written above.

$\Delta(6,4,6)$ is indeed normal in $\Delta(6,8,2)$, because it has index 2, so conjugation by elements of $\Delta(6,8,2)$ gives an automorphism of $\Delta(6,4,6)$. We need to see if this automorphism fixes $E_0/N_{(6,4,6)}$.

Let us look at the action of the group by conjugation. In particular, let us look at the action of $\wt y$ (which we are choosing because it doesn't lie in $\Delta(6,4,6)$). We have that ${\wt y}^{-1}y\wt y=y^{\wt y}=({\wt y}^2)^{\wt y}=y$. Also,
Since $\wt z=(\wt x\wt y)^{-1}$ has order 2, $\wt x \wt y=\wt y^{-1} \wt x^{-1}$, so
$$
x^{\wt y}=\wt y^{-1} \wt x\wt y=\wt y^{-1}\wt y^{-1} \wt x^{-1}=\wt y^{-2}\wt x^{-1}=y^{-1}x^{-1}
$$
Therefore, the automorphism of $E_0/N_{(6,4,6)}$ induced by $\wt y$ is given by
$$
x\longmapsto y^{-1}x^{-1}$$ $$
y\longmapsto y
$$
Let us see if it fixes $E_0$. Take the subgroup $F/N_{(6,4,6)}$, which contains $E_0$. Recall that it is the normal subgroup generated by $\{x^3,y^2,[x,x^y]\}$. In particular, it is contained in the group $A=\langle \langle x,y^2\rangle\rangle$, so $E_0$ is also contained in this group. 

Recall that $E_0$ contained the element $x^3(x^3)^y(x^3)^{y^2}$. However, we are going to see that $(x^3y^2(x^3y^2)^x(x^3y^2)^{x^2})^{\wt y}\notin A$. Indeed, we have that
$$
(x^3y^2(x^3y^2)^x(x^3y^2)^{x^2})^{\wt y}=
(y^{-1}x^{-1})^3y^2
\left((
y^{-1}x^{-1})^3y^2
\right)^{y^{-1}x^{-1}}\left((
y^{-1}x^{-1})^3y^2
\right)^{(y^{-1}x^{-1})^2}
$$
And the word on the right contains an odd number of $y$'s, so it is not in $A$. Therefore, $(x^3y^2(x^3y^2)^x(x^3y^2)^{x^2})^{\wt y}$ isn't even in $A$, so it can't belong to $E_0$. Therefore, $E_0$ is not normal in $\Delta(6,8,2)$, and we have proven that the dessin on $C$ is the maximal one. Therefore, the curve $C^\sigma$ is different from $C$.

We have thus found a quasiplatonic curve of genus 61 with field of moduli $\Q(\rd)$.

\section{A different example}

Let us give another example of a regular dessin that has a field of moduli that is not abelian. Consider the following dessins which have the following sets of ramification indices: the points over 0 (the black vertices) have indices $(2,2,1,1)$, the points over $1$ have indices $(3,2,1)$ and there is one point over $\infty$ with index 6. In particular, its Euler characteristic is $4+3-6+1=2$, so the genus of the dessin is $0$, and they are trees on $\P^1$.

Some fiddling around shows that there are three different dessins with these indices, namely the ones depicted in figure \ref{trees}.

\begin{center}
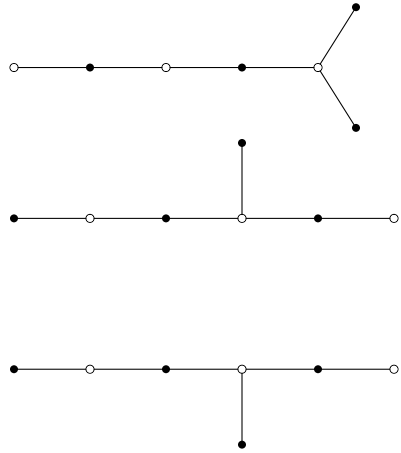

\begin{tikzpicture}[line cap=round,line join=round,>=triangle 45,x=1cm,y=1cm]
\clip(-1,-3.5) rectangle (7,3);
\draw (1,2)-- (5,2);
\draw (1,0)-- (6,0);
\draw (1,-2)-- (6,-2);
\draw (4,0)-- (4,1);
\draw(4,-2)--(4,-3);
\draw(5,2)--(5.5,2.8);
\draw(5,2)--(5.5,1.2);
\puntonegro{(2,2)}
\puntonegro{(4,2)}
\puntonegro{(5.5,2.8)}
\puntonegro{(5.5,1.2)}
\puntonegro{(1,0)}
\puntonegro{(3,0)}
\puntonegro{(5,0)}
\puntonegro{(1,-2)}
\puntonegro{(3,-2)}
\puntonegro{(5,-2)}
\puntonegro{(4,1)}
\puntonegro{(4,-3)}
\puntoblanco{(1,2)}
\puntoblanco{(3,2)}
\puntoblanco{(5,2)}
\puntoblanco{(2,-2)}
\puntoblanco{(4,-2)}
\puntoblanco{(6,-2)}
\puntoblanco{(2,0)}
\puntoblanco{(4,0)}
\puntoblanco{(6,0)}
\end{tikzpicture}
\captionof{figure}{The three trees with vertex orders $(2,2,1,1)$, $(3,2,1)$ and $(6)$.}\label{trees}
\end{center}

\noindent Each of these dessins is given by a Belyi function $f_i:\P^1\lra \P^1$, for $i=0,1,2$. Since the only point in the preimage of $\infty$ is $\infty$, the functions are actually polynomials. They are calculated in \cite{voed} (also in \cite{sch94}, and it is example 4.58 in \cite{GG}), and the formula for $f_i$ is
$$
f_i(z)=1-z^3(z+1)^2(z+a_i)
$$
Where $a_i$ runs through the roots of the irreducible polynomial
$$
P=25x^3-12x^2-24x-16
$$
It turns out that the roots of this polynomial are
$$
a_0=\frac{4+18\rd+6\sqrt[3]{4}}{25}
$$
And its conjugates
$$
a_1=\frac{4+18\xi\rd+6\xi^2\sqrt[3]{4}}{25}
$$
$$
a_2=\frac{4+18\xi^2\rd+6\xi\sqrt[3]{4}}{25}
$$
Where $\xi^3=1$.$\gal$ permutes these roots, and therefore it induces a permutation of the three dessins. Also, their fields of moduli are $\Q(a_i)=\Q(\xi^{i}\rd)$. If we call $K=\Q(a_0,a_1,a_2)=\Q(\rd,\xi)$, we have that $\Gal(K/\Q)\cong S_3$. Thus, the Galois group of the extensions acts on the dessin as the usual action of $S_3$ on three points.

Let us look at their regular covers, and prove that they are different. This will mean that, since the Galois action maps regular covers to regular covers, their fields of moduli are also $\Q(a_i)$.

We are going to see that the cartographic group is $S_6$. This can be found in \cite{malle}, where a list of genus 0 dessins of degree up to 13 is given, along with their cartographic groups and their fields of definition.

Take their associated subgroups $H_0,H_1,H_2<\Fd$. The cartographic group is the image in $S_6$ of the homomorphism $\Fd/\cor_{\Fd}(H_i)\lra S_6$ induced by the action. For example, for a certain way of numbering the edges in figure \ref{trees}, these maps are given by
$$
\begin{array}{ccccc}
& & \text{Top} & \text{Middle} & \text{Bottom}\\
x& \longmapsto& (14)(56)& (14)(26) & (14)(36)\\
y&\longmapsto &(123)(45)&(123)(45)&(123)(45)
\end{array}
$$
In every dessin, the permutation given by $z=(xy)^{-1}$ is a cycle of order 6, since the exterior face has 6 edges, and $y^3$ is a transposition, so together they generate the whole group $S_6$.

Let us call $\o{H_i}=\cor_{\Fd}(H_i)$. Since the Galois group acts on these as the full permutation group, if two of them are equal, they are all equal. Suppose then that they are all equal. Take an element $\sigma$ that has order three in $\Gal(K/\Q)$, so it permutes all three roots. For example, take the one that induces the permutation $H_0\rightarrow H_1\rightarrow H_2$. If $\sigma$ fixed $\o{H_i}$, then it would induce an automorphism of the quotient $\Fd/\o{H_i}$. This quotient is the cartographic group of the dessins, so it is isomorphic to $S_6$, for example by the isomorphism given by the monodromy of the first dessin. Furthermore, this automorphism is outer: it maps $H_0/\o{H_0}$, which is the stabilizer of 1 by the action of $S_6$, to a subgroup $H_1$, that isn't conjugate to $H_1$, for if they were conjugate, the degree 6 dessins would be isomorphic.

However, $\Out (S_6)$ has order 2. Therefore, $\sigma^4=\sigma\in \Inn(S_6)$ (for $\sigma$ has order 3). This contradicts the possibility that $\o{H_0}=\o{H_1}=\o{H_2}$, so these subgroups are different and they are interchanged by the Galois action. Therefore, their fields of moduli are $\Q(a_i)=\Q(\xi^i\rd)$. Their genus is also 61, since their Euler characteristic is $720\cdot(1/2+1/6+1/6-1)=-120$.

\begin{Referencias}
\bitem{belyi}
\noun{Belyi, G. V.}, \textit{On Galois Extensions of a Maximal Cyclotomic Field}, Mathematics of the USSR-Izvestiya 14 (2): 247, 1980

\bitem{cojowo}
\noun{Conder, M., Jones, G., Streit, M., Wolfart, J.}, \textit{Galois actions on regular dessins of small genera}, Rev. Mat. Iberoam. 28 (2012), no. 4, 1–19

\bitem{couv}
\noun{Couveignes, J.-M.}, \textit{Calcul et rationalité de fonctions de Belyi en genre 0}, Ann. de l'Inst. Fourier, 994, vol. 44, no. 1, 1-38

\bitem{rabbit}
\noun{Couveignes, J.-M.}, \textit{Dessins from a geometric point of view}, The Grothendieck theory of dessins d'enfants pp. 79-114, London Math. Soc. Lecture notes 200, Cambridge University Press, 1994.

\bitem{farkas}
\noun{Farkas, H., Kra, I.}, \textit{Riemann Surfaces}, Graduate texts in mathematics, Vol. 71, Springer, 1992

\bitem{multiply}
\noun{Girondo, E.}, \textit{Multiply quasiplatonic Riemann surfaces}, Experimental Mathematics, Vol. 12, No. 4, 2003

\bitem{GG}
\noun{Girondo, E., González-Diez, G.}, \textit{Introduction to Compact Riemann Surfaces and Dessins d'Enfants}, London Mathematical Society Student Texts, Cambridge University Press, 2011

\bitem{genus}
\noun{Girondo, E., González-Diez, G.}, \textit{A note on the action of the absolute Galois group on dessins}, Bulletin of the London Mathematical Society, Vol. 39 Issue 5, p721, 2007

\bitem{jule03}
\noun{González-Diez, G., Jaikin-Zapirain, A.}, \textit{The absolute Galois group acts faithfully on regular dessins and on Beauville surfaces}, preprint, available at \url{http://www.uam.es/personal_pdi/ciencias/gabino/Jule03.pdf}, 2013

\bitem{gro}
\noun{Grothendieck, A.}, \textit{Esquisse d'un Programme}, 1984. It is published in Geometric Galois Actions. Around Grothendieck's Esquisse d'un Programme, edited by L. Schneps and P. Lochak, Cambridge University Press, 1997.

\bitem{Guillot}
\noun{Guillot, P.}, \textit{An elementary approach to Grothendieck's dessins d'enfants and the Grothendieck-Teichmüller group}, {\tt arXiv:1309.1968}

\bitem{freepro}
\noun{Herfort, W., Ribes, L.}, \textit{Torsion elements and centralizers in free products of profinite groups}, Journal für die reine und angewandte Mathematik. Vol. 1985, 358, pp. 155–161

\bitem{Lando}
\noun{Lando, S. K., Zvonkin, A. K.}, Graphs on Surfaces and their Applications, Encyclopaedia of Mathematical Sciences: Lower-Dimensional Topology II 141, Berlin, New York, Springer-Verlag, 2004

\bitem{malle}
\noun{Malle, G.}, \textit{Fields of definition of some three point ramified field extensions}, The Grothendieck theory of dessins d'enfants, London Math. Soc. Lecture notes 200, Cambridge University Press, 1994. 

\bitem{Miranda}
\noun{Miranda, R.}, \textit{Algebraic curves and Riemann surfaces}, Graduate studies in mathematics, American Mathematical Society, 1995

\bitem{Munkres}
\noun{Munkres, J.}, \textit{Topology}, Prentice Hall, 2000

\bitem{ucdavis}
\noun{Osserman, B.}, \textit{Infinite Galois theory}, notes available at \url{https://www.math.ucdavis.edu/~osserman/classes/250C/}

\bitem{profinite}
\noun{Rives, L., Zaleskii, P.}, \textit{Profinite groups}, A Series of Modern Surveys in Mathematics, Springer-Verlag, 2000

\bitem{serre}
\noun{Serre, J. P.}, \textit{Galois cohomology}, Springer, 2002

\bitem{sch94}
\noun{Schneps, L.}, \textit{Dessins d'enfants on the Riemann sphere}, The Grothendieck theory of dessins d'enfants, London Math. Soc. Lecture notes 200, Cambridge University Press, 1994.

\bitem{voed}
\noun{Shabat, G., Voevodsky, V.}, \textit{Drawing curves over number fields}, The Grothendieck Festschrift, Progress in Mathematics Volume 88, 1990, pp 199-227

\bitem{singerman}
\noun{Singerman, D.}, \textit{Finitely maximal Fuchsian groups.}, J. London Math. Soc. (2) 6 (1972), 29–38.

\bitem{weil}
\noun{Weil, A.}, \textit{The field of definition of a variety}, American Journal of Mathematics, Vol. 78, No. 3, 1956

\bitem{abc}
\noun{Wolfart, J.,} \textit{ABC for polynomials, dessins and uniformization - a survey.} Proceedings der ELAZ-Konferenz 2004, pp. 314-346 Hrsg. W. Schwarz, J. Steuding. Also available at \url{http://www.math.uni-frankfurt.de/\~wolfart/}


\end{Referencias}

\end{document}